\documentclass[a4paper]{article}
\usepackage[textwidth=18cm,textheight=25cm]{geometry}
\usepackage{blindtext}
\usepackage{amsmath} 
\usepackage{amsthm}
\usepackage{bm}
\usepackage{texnames}
\usepackage{mflogo}
\usepackage{amssymb}
\usepackage{listings}
\usepackage{url}
\usepackage{indentfirst}
\usepackage{graphicx}
\usepackage{color}
\usepackage[T1]{fontenc}
\usepackage{textcomp}
\usepackage{CJKutf8}
\newtheorem{theorem}{Theorem}
\newtheorem*{theorem*}{Theorem}

\newtheorem{proposition}{Proposition}

\numberwithin{equation}{section}
\graphicspath{{pics/}}

\begin{document}
\begin{CJK}{UTF8}{gbsn}

\title{SHOCK INTERACTION IN SPHERE SYMMETRY}
\author{YUXUAN WANG}
\date{}

\maketitle

\begin{abstract}
We investigate the interaction of two oncoming shock waves in spherical symmetry for an ideal barotropic fluid. Our research problem is how to establish a local in time solution after the interaction point and determine the state behind the shock waves. This problem is a double free boundary problem, as the position of the shock waves in space time is unknown. Our work is based on a number of previous studies, including Lisibach's, who studied the case of plane symmetry. To solve this problem, we will use an iterative regime to establish a local in time solution.
\end{abstract}

\tableofcontents

\section{Introduction}

\subsection{Euler equations}

Conservation laws are fundamental principles in physics that describe the behavior of physical systems.  They state that certain physical properties, such as mass, energy, and momentum, are conserved and do not change over time within an isolated system. Conservation laws are expressed mathematically as partial differential equations, which can be used to study nonlinear wave phenomena. However, understanding conservation laws in higher dimensions is complicated, so the focus is mainly on the one dimensional case.

Euler equations are a special class of conservation laws, which are used to describe the adiabatic inviscid flow of fluids. These equations were first proposed by Euler. One of the most interesting features of compressible Euler equations is the phenomenon of shock formation. This problem was first studied in one space dimension by Riemann in 1860 \cite{riemann1860fortpflanzung}. He proved that from smooth data, shocks can form in finite time.

In fact, Euler equations can be utilized in both incompressible and compressible flow scenarios.  The incompressible Euler equations encompass the Cauchy equations for mass conservation and momentum balance, in addition to the condition that the flow velocity constitutes a solenoidal field.  On the other hand, the compressible Euler equations comprise equations for mass conservation, momentum balance, and energy balance, along with an appropriate constitutive equation for the specific energy density of the fluid.

The mathematical properties of the incompressible and compressible Euler equations differ significantly.  When the fluid density is constant, the incompressible equations can be expressed as a quasilinear advection equation for the fluid velocity, along with an elliptic Poisson's equation for the pressure.  Conversely, the compressible Euler equations represent a quasilinear hyperbolic system of conservation equations.

The compressible and most general Euler equations can be expressed concisely in differential convective form using the material derivative notation:

\begin{equation}\label{eq1.3}
\left\{
\begin{aligned}
\frac{D\rho}{Dt}&=-\rho\nabla\cdot\textbf{u},\\
\frac{D\textbf{u}}{Dt}&=-\frac{\nabla p}{\rho}+\textbf{g},\\
\frac{De}{Dt}&=-\frac{p}{\rho}\nabla\cdot\textbf{u}.
\end{aligned}
\right.
\end{equation}
Here $\textbf{u}$ is the flow velocity vector, $\rho$ is the fluid mass density, $p$ is the pressure, $\textbf{g}$ represents body accelerations (per unit mass) acting on the continuum, $e$ is the specific internal energy (internal energy per unit mass). If we expand the material derivative, \eqref{eq1.3} can be expressed as:

\begin{equation}\label{eq1.4}
\left\{
\begin{aligned}
\frac{\partial \rho}{\partial t}+\textbf{u}\cdot \nabla\rho+\rho\nabla\cdot\textbf{u}&=0,\\
\frac{\partial\textbf{u}}{\partial t}+\textbf{u}\cdot\nabla\textbf{u}+\frac{\nabla p}{\rho}&=\textbf{g},\\
\frac{\partial e}{\partial t}+\textbf{u}\cdot\nabla e+\frac{p}{\rho}\nabla\cdot\textbf{u}&=0.
\end{aligned}
\right.
\end{equation}
The equations above thus represent conservation of mass, momentum, and energy.

\subsection{Shock wave interaction}

This paper investigates the interaction of two oncoming shocks in the case of spherical symmetry in three dimensional space, where the fluid is assumed to be ideally barotropic. The problem is formulated as follows: two shock waves, an inward propagating spherical shock wave and an outward propagating spherical shock wave, are coming face to face with the same center of the sphere. As time progresses, the distance between the two shock waves decreases until they collide. After the collision, the two shock waves move away from each other, leaving behind a growing region. The time of collision is assumed to be $t=0$, and the distance between the two shock waves and the center of their sphere at the time of collision is $r=r_0$. Two sets of initial values are given, corresponding to $r\geq r_0$ and $r\leq r_0$, respectively. This paper solely focuses on the interaction of the two spherical shock waves and does not delve into their formation, development, and propagation.

Upon crossing the shock curve, the solution undergoes a discontinuous jump. The relationship between the physical quantities on each side of the shock curve that must be satisfied is referred to as the jump conditions. Consequently, we postulate that at time $t=0$, the two sets of initial values corresponding to $r\geq r_0$ and $r\leq r_0$ have distinct physical states at $r=r_0$, with their values differing at the point of interaction. Furthermore, by applying the jump conditions to the two sets of initial values, we can calculate the speeds of the two resulting shocks at the interaction point. Each shock wave moves at a subsonic speed relative to its state behind and at a supersonic speed relative to its state ahead, which is known as the determinism condition. The definitions of the state behind and the state ahead of shock waves can be found in \cite{christodoulou2016shock}.

Our main work is as follows: first we formulate the shock interaction problem, and then we construct the solution of the shock interaction problem by establishing an iterative regime. Finally, we study the uniqueness and regularity of the solution.

\begin{figure}[htbp]
  \centering
  \includegraphics{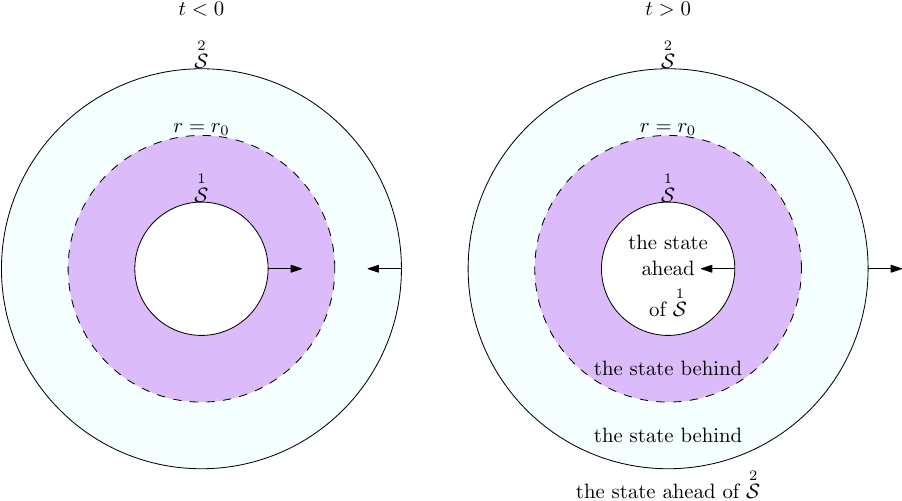}
  \caption{The moment before the interaction (left) and the moment after the interaction (right).}
  \label{p0}
\end{figure}

\begin{theorem*}[Rough statement of the theorems]\label{theorem0}

Given two sets of smooth initial values on $r\le r_0$ and $r\ge r_0$, there exists a unique $C^2$ solution to the shock interaction problem in the characteristic coordinate system in the state behind. Furthermore, the solution is smooth with respect to the characteristic coordinates.

\end{theorem*}
The precise statement of the theorems is given in Theorem \ref{existence}, Theorem \ref{uniqueness} and Theorem \ref{higher regularity}.

\subsection{Some recent results on shock formation and shock development}

The study of shock waves originated in the context of gas dynamics. The book by Courant and Friedrichs \cite{courant1948supersonic} presented a coherent mathematical exposition of materials from the physics and engineering literature accumulated in the past, paving the way for the development of a general theory by Lax \cite{https://doi.org/10.1002/cpa.3160100406}. The notion of Hugoniot locus, in gas dynamics, is traced back to the work of Riemann \cite{riemann1860fortpflanzung} and Hugoniot \cite{article}, but the definition of shock curves in the general setting is due to Lax \cite{https://doi.org/10.1002/cpa.3160100406}. The significance of systems with coinciding shock and rarefaction wave curves was first recognized by Temple \cite{Temple1983SystemsOC}, who conducted a thorough study of their noteworthy properties. A detailed discussion is also contained in Serre \cite{article1}. For gas dynamics, the statement that admissible shocks should be subsonic relative to their left state and supersonic relative to their right state is found in the pioneering paper of Riemann \cite{riemann1860fortpflanzung}. This principle was postulated as a general shock admissibility criterion, namely the Lax $E$-condition, by Lax \cite{https://doi.org/10.1002/cpa.3160100406}.

The study of shock formation for the Euler equations has a long history, especially in the case of one dimensional space. The 1D isentropic Euler equations serve as an exemplar of a $2\times 2$ system of conservation laws. Through the utilization of Riemann invariants, Lax \cite{lax1964development} demonstrated that shocks can form in finite time from smooth data in general $2\times 2$ genuinely nonlinear hyperbolic systems. Majda \cite{majda2012compressible} further offered a geometric proof that encompassed $2\times 2$ systems with linear degeneracy. Subsequently, John \cite{1974Formation} established the occurrence of finite-time shock formation in $n\times n$ genuinely nonlinear hyperbolic systems, with Liu \cite{liu1979development} subsequently extending this finding. Klainerman-Majda \cite{klainerman1980formation}, on the other hand, substantiated the formation of singularities in second-order quasilinear wave equations, which encompass the nonlinear vibrating string. For a more comprehensive list of 1D results, see Dafermos's book \cite{dafermos2005hyperbolic}.

The first proof of shock formation for the compressible Euler equations in the multi-dimensional setting was originally demonstrated by Christodoulou \cite{christodoulou2007formation} in relativistic fluids under the condition of irrotational flow. Subsequently, Christodoulou-Miao \cite{christodoulou2014compressible} applied the same methodology to investigate shock formation in a non-relativistic setting also with irrotational flow. Christodoulou's approach revolves around the use of a novel eikonal function (see also Christodoulou-Klainerman \cite{ChristodoulouKlainerman+1994} and Klainerman-Rodnianski \cite{10.1215/S0012-7094-03-11711-1}), whose level sets correspond to characteristics of the flow. By introducing the concept of inverse foliation density, which is inversely proportional to time-weighted derivatives of the eikonal function, Christodoulou proved that shocks form when the inverse foliation density vanishes (i.e., characteristics cross). Furthermore, he established that no other breakdown mechanism can occur prior to this shock formation. The proof relies on the utilization of a geometric coordinate system, along which the solution has long-time existence and remains bounded. Consequently, the shock is constructed by means of the singular (or degenerate) transformation from geometric to Cartesian coordinates. For the restricted shock development problem, in which the Euler solution is continued beyond the time of first singularity while neglecting vorticity production, please refer to Section 1.6 of \cite{christodoulou2017shock}. Starting with piecewise regular initial data for which there is a closed curve of discontinuity, across which the density and normal component of velocity experience a jump, Majda \cite{majda1983existence}, \cite{majda1983stability}, \cite{majda2012compressible} proved that such a shock can always be continued for a short interval of time, but with derivative loss.  M\'{e}tivier \cite{metivier2001stability} later refined this result and reduced the derivative loss to only a $1/2$-derivative. The existence and stability of this multidimensional shock propagation problem in the vanishing viscosity limit were studied by Gues-M\'{e}tivier-Williams-Zumbrun \cite{gues2005existence}.

The initial findings on the formation of shocks in 2D quasilinear wave equations, which do not meet Klainerman's null condition, were established by Alinhac \cite{Alinhac1999}, \cite{alinhac1999blowup}, providing a comprehensive description of the phenomenon of blowup. The analysis of shock formation in quasilinear wave equations has been greatly influenced by the geometric framework presented in \cite{christodoulou2007formation}. Holzegel-Klainerman-Speck-Wong \cite{holzegel2016small} have explained the mechanism behind stable shock formation in specific types of quasilinear wave equations with small data in three dimensions. Speck \cite{speck2016shock} has further generalized and unified previous work on the formation of singularities in both covariant and non-covariant scalar wave equations of a particular form. He demonstrated that when the nonlinear terms fail Klainerman's null condition, shocks arise in solutions derived from a set of small data, which can be seen as a converse of the well-known result of Christodoulou-Klainerman \cite{ChristodoulouKlainerman+1994}, which showed that the existence of global solutions for small data is guaranteed when the classic null condition is satisfied. Miao-Yu \cite{Miao2017} proved shock formation in quasilinear wave equations derived from the least action principle and satisfying the null condition, using what is known as short pulse data. The first proof of shock formation in fluid flows with vorticity was presented by Luk-Speck \cite{Luk2018}, for the 2D isentropic Euler equations with vorticity. The inclusion of nontrivial vorticity in their analysis not only allows for a broader range of data but also involves two families of waves, sound waves and vorticity waves, leading to interactions between multiple characteristics. Their proof utilizes Christodoulou's geometric framework from \cite{christodoulou2007formation}, \cite{christodoulou2014compressible}, while also introducing new methodologies to address the vorticity waves mentioned earlier.  They establish new estimates for the regularity of the transported vorticity-divided-by-density and heavily rely on a novel framework for describing the 2D compressible Euler equations as a coupled system of covariant wave and transport equations. In \cite{Luk2018}, Luk-Speck consider solutions to the Euler equations, which are small perturbations of a specific class of outgoing simple plane waves.

Next, we present some results on Riemann problems for the compressible Euler system and the nonlinear wave system. In the case of 2D, for example,  some researchers have investigated shock reflection problems \cite{bae2009regularity}, \cite{canic2006free}, \cite{chen2010global}, \cite{chen2018mathematics}, \cite{chen2020convexity}, \cite{chen2019uniqueness}, \cite{jegdic2006transonic}, \cite{kim2013transonic}, \cite{tesdall2007triple}, \cite{zheng2006two}, shock diffraction problems \cite{chen2014shock}, supersonic flows around a convex wedge \cite{bae2013prandtl}, \cite{elling2008supersonic}, the interaction of transonic shock and rarefaction waves \cite{kim2012interaction}, and the interactions of rarefaction waves \cite{bang2009interaction}, \cite{hu2014semi}, \cite{hu2015interaction}, \cite{kim2016two}, \cite{lei2007complete}, \cite{li2012interaction}, \cite{li2011characteristic}, \cite{li2009interaction}, \cite{song2009semi}. 

For higher and one-dimensional cases, many scholars have done research. For example, Luo-Yu \cite{luo2023stability} proved the non-nonlinear stability of the Riemann problem for multi-dimensional isentropic Euler equations in the regime of rarefaction waves and they studied multi-dimensional perturbations of the classical Riemann problem for isentropic Euler equations in \cite{luo2023stability1}. Lisibach conducted a study on shock reflection and interaction problems in one-dimensional space. In \cite{lisibach2021shock}, he examined the shock reflection of an ideal barotropic fluid with respect to a fixed solid wall under plane symmetry, while in \cite{lisibach2022shock}, he investigated the interaction problem of two oncoming shock waves in the case of plane symmetry. Lisibach utilized various mathematical techniques, such as introducing characteristic coordinates and constructing function spaces, to prove the local existence, uniqueness, and higher order regularity of the solution in these two problems.

In this paper, we investigate the problem of shock wave interaction in spherical symmetry, which is more complex than the plane symmetry studied in \cite{lisibach2022shock}. While Lisibach's work provides valuable insights, some details, such as the selection of characteristic coordinates, are omitted. To fill this gap, we present a more detailed proof, building on the ideas and goals of Lisibach's work and constructing a similar iteration. However, the equations satisfied by the Riemann invariants in our problem are more complex, requiring the introduction of additional function spaces and iterative variables. By proving the convergence of the iteration sequences in the function space we construct, we establish the local existence and uniqueness of the solution to the interaction problem, as well as its higher order regularity. Our results will contribute to a better understanding of the interaction of spherical shock waves in 3 dimensional space.

\subsection{Overview of the article}

Finally, we summarize the work done in each section of the paper.

\textbf{Section} \ref{2}: Section 2 introduces the notations used in the article.

\textbf{Section} \ref{3}: Section 3 establishes the description of the equations of motion.  Specifically, Section 3.1 presents the Euler equations satisfied by the ideal barotropic fluid in the case of 3 dimensional spherical symmetry. Section 3.2 defines the Riemann invariants and characteristic speed and calculates the equations satisfied by Riemann invariants in $t-r$ coordinate system.  Section 3.3 derives the jump conditions and provides specific expressions of the determinism condition. 

\textbf{Section} \ref{4}: In Section 4, the quantities in the state ahead of the two emerged shocks are given in Section 4.1, and the jump conditions are used to calculate the values at the interaction point.  Section 4.2 provides a complete description of the shock interaction problem and the objectives of this paper.

\textbf{Section} \ref{5}: In Section \ref{5.1}, we provide a detailed explanation of the introduction of characteristic coordinates. In Sections \ref{5.2} and \ref{5.3}, we present the boundary conditions and jump conditions in the characteristic coordinates and calculate the initial values that the Riemann invariants satisfy at the interaction point. In Section \ref{5.4}, we derive the equations satisfied by $\frac{\partial r}{\partial u}$ and $\frac{\partial r}{\partial v}$ by integrating along the incoming and outcoming characteristics, which are crucial for the later proof. In Section \ref{5.5}, we provide a complete mathematical statement of the shock interaction problem.

\textbf{Section} \ref{6}: In Section \ref{6.1}, We provide an iterative regime. Section \ref{6.2} explains the iterative steps and provides details of the proofs that follow. Sections \ref{6.3} and \ref{6.4} prove that the five previously selected variables are in the corresponding function spaces in each iteration step, and that they converge under the defined norm by selecting appropriate constants and taking advantage of the small condition of $\varepsilon$. In Section \ref{6.5}, we prove the local existence theorem of the solution. Finally, Sections \ref{6.6} and \ref{6.7} prove the uniqueness of the solution in the corresponding function spaces by providing the asymptotic form satisfied by the solution, and using a method similar to the proof of convergence in Section \ref{6.4}.

\textbf{Section} \ref{7}: In this section, we aim to establish the infinite differentiability of the solution with respect to the characteristic coordinates through a mathematical induction proof.

\begin{figure}[htbp]
  \centering
  \includegraphics{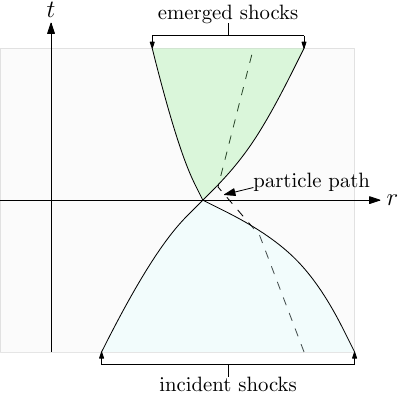}
  \caption{The interaction of shock waves}
  \label{p1}
\end{figure}

\section{Notation}\label{2}

In this section, we summarize some symbol conventions that will be used later.

\begin{itemize}
\item[$\bullet$] To distinguish between quantities related to the left moving shock curve and the right moving shock curve, we utilize superscript 1 (e.g., $\overset{1}{\alpha}$) for the former and superscript 2 (e.g., $\overset{2}{\beta}$) for the latter.
\item[$\bullet$] In order to represent the quantities in the state ahead, we will add an asterisk $(*)$ to the upper right of the functions. For instance, ${\overset{1}{\alpha}}^*$ and ${\overset{2}{\beta}}^*$ will be used to denote these functions. Additionally, the components of these functions will be represented by the coordinates $t$ and $r$.
\item[$\bullet$] The state behind is represented by the region between the two emerged shocks, with its corresponding components of functions denoted by the coordinates $u$ and $v$.
\item[$\bullet$] In our analysis, we utilize the subscript $+$ to denote the quantities along each shock curve in the state behind, while the subscript $-$ is employed to represent the quantities along each shock curve in the state ahead.
\item[$\bullet$] We represent $\left|f(u)\right|\le C(A)u^n$ as $f(u)=\mathcal{O}_A(u^n)$, where $C(A)$ is a continuous non-decreasing function with respect to $A$.

\end{itemize}

\begin{figure}[htbp]
  \centering
  \includegraphics{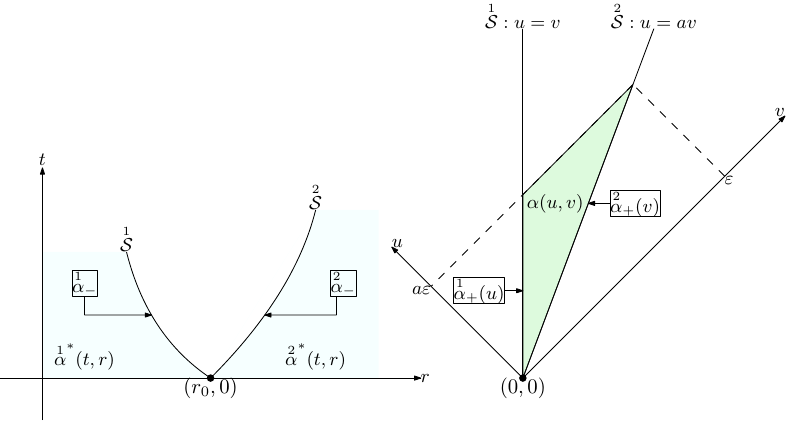}
  \caption{A description of notation}
  \label{p2}
\end{figure}

Figure~\ref{p2} provides a visual representation of the notations used in this paper. The left picture corresponds to the notations of the state ahead, and the right picture corresponds to the notations of the state behind.

\section{Euler equations}\label{3}

\subsection{Derivation of Euler equations}\label{3.1}

We begin by considering the inviscid flow of compressible fluids in 3 dimensional space, where the fluid velocity $\textbf{v}=\textbf{v}(\textbf{x},t)$, pressure $p=p(\textbf{x},t)$, and density $\rho=\rho(\textbf{x},t)$. Let $\Omega(t)$ be the region occupied by the fluid in $\mathbb{R}^3$ at time $t$. The mass of the fluid in $\Omega(t)$ is $\int_{\Omega(t)}\rho dV$. By the conservation of mass, the integral does not depend on $t$, so we have $\frac{d}{dt}\int_{\Omega(t)}\rho dV=0.$

Using the transport formula in fluid mechanics, we obtain

\begin{equation*}
\int_{\Omega (t)} \frac{\partial\rho}{\partial t}dV+\int_{\partial\Omega(t)}\rho \textbf{v}\cdot\textbf{n} dS=0,
\end{equation*}
where $\textbf{n}$ is the unit outward normal vector of $\partial\Omega(t)$. Applying the Green formula, we have

\begin{equation*}
\int_{\Omega(t)}(\frac{\partial \rho}{\partial t}+div(\rho\textbf{v}))dV=0.
\end{equation*}
Since $\Omega(t)$ is arbitrary, we obtain the continuity equation

\begin{equation}\label{eq3.1}
\frac{\partial \rho}{\partial t}+div(\rho\textbf{v})=0.
\end{equation}
Next, we consider the conservation of momentum. By Newton's law, we have $\frac{d\textbf{P}}{dt}=\textbf{F},$ where $\textbf{P}$ is the momentum and $\textbf{F}$ is the external force. By definition, we have

\begin{equation*}
\frac{d}{dt}\int_{\Omega(t)}\rho\textbf{v}dV=-\int_{\partial\Omega(t)}p\textbf{n}dS=-\int_{\Omega(t)}\nabla p dV.
\end{equation*}
Expanding this equation in component form, we obtain

\begin{equation*}
\frac{d}{dt}\int_{\Omega(t)}\rho v^i dV=-\int_{\Omega(t)}\frac{\partial p}{\partial x^i} dV=\int_{\Omega(t)}(\frac{\partial(\rho v^i)}{\partial t}+div(\rho v^i\textbf{v}))dV.
\end{equation*}
By the arbitrariness of $\Omega(t)$, using the continuity equation obtained earlier, we can obtain the Euler equation

\begin{equation}\label{eq3.2}
\frac{\partial v^i}{\partial t}+\textbf{v}\cdot\nabla v^i=-\frac{1}{\rho}\frac{\partial p}{\partial x^i}.
\end{equation}

In rectangular coordinates, the continuity and Euler equations can be written in vector form as
\begin{equation*}
\begin{cases}
&\frac{\partial\rho}{\partial t}+\nabla\cdot(\rho\textbf{v})=0,\\
&\frac{\partial\textbf{v}}{\partial t}+\textbf{v}\cdot\nabla\textbf{v}=-\frac{1}{\rho}\nabla p,
\end{cases}
\end{equation*}
or

\begin{equation*}
\begin{cases}
&\frac{\partial\rho}{\partial t}+div_g(\rho\textbf{v})=0,\\
&\frac{\partial\textbf{v}}{\partial t}+\nabla_{\textbf{v}}\textbf{v}=-\frac{1}{\rho}\nabla p,
\end{cases}
\end{equation*}
where $g$ is the flat metric in the 3 dimensional Euclidean space, $div$ represents the divergence, and the two $\nabla$ operators in the second equation represent the contact and gradient, respectively.

In spherical coordinates, assuming spherical symmetry, we obtain the continuity and Euler equations as

\begin{equation}\label{eq3.3}
\partial_t \rho +\partial_r(\rho w)+\frac{2}{r}(\rho w)=0,
\end{equation}
\begin{equation}\label{eq3.4}
\partial_t w+w\partial_r w+\frac{1}{\rho}\partial_r p=0,
\end{equation}
where $w$ is the radial velocity. We assume that the fluid is barotropic, so $p=p(\rho)$ is a given smooth function that satisfies $\frac{dp}{d\rho}(\rho)>0$. We do not consider entropy.

\subsection{Riemann invariants of the main part of the system}\label{3.2}

In this section, we define the Riemann invariants $\alpha$ and $\beta$ as follows:

\begin{equation}\label{eq3.5}
\alpha:=\int^\rho \frac{\eta(\rho^{'})}{\rho^{'}} d\rho^{'}+w,\ \ \ \beta:=\int^\rho \frac{\eta(\rho^{'})}{\rho^{'}} d\rho^{'}-w,
\end{equation}
where $\eta:=\sqrt{\frac{dp}{d\rho}}$ represents the speed of sound. Using this definition, we obtain $w=\frac{\alpha-\beta}{2}.$

We also define
\begin{equation}\label{eq3.6}
c_{out}:=w+\eta,\ \ \ c_{in}:=w-\eta,
\end{equation}
\begin{equation}\label{eq3.7}
L_{out}:=\partial_t+c_{out}\partial_r,\ \ \ L_{in}:=\partial_t+c_{in}\partial_r.
\end{equation}
By direct calculation, we can derive the expressions for $L_{out}\alpha$ and $L_{in}\beta$:

\begin{equation}\label{eq3.8}
\begin{aligned}
L_{out}\alpha&=\partial_t\alpha+c_{out}\partial_r\alpha=-\frac{2\eta w}{r},\ \ \ L_{in}\beta&=\partial_t\beta+c_{in}\partial_r\beta=-\frac{2\eta w}{r}.
\end{aligned}
\end{equation}

Next, we compute the Jacobian matrix of $\alpha$ and $\beta$ with respect to $\rho$ and $w$

\begin{equation}\label{eq3.9}
\frac{\partial(\alpha,\beta)}{\partial(\rho,w)}=
\begin{pmatrix}
\frac{\eta}{\rho} & 1\\
\frac{\eta}{\rho} & -1
\end{pmatrix}.
\end{equation}
This matrix is nonsingular, so we can compute its inverse

\begin{equation}\label{eq3.10}
\frac{\partial(\rho,w)}{\partial(\alpha,\beta)}=
\begin{pmatrix}
\frac{\rho}{2\eta} & \frac{\rho}{2\eta}\\
\frac{1}{2} & -\frac{1}{2}
\end{pmatrix}.
\end{equation}
This concludes our discussion of the Riemann invariants of the main part of the system.

\subsection{Jump conditions and determinism condition}\label{3.3}

In this section, we discuss the jump conditions and determinism condition for the shock wave. Assuming that there is a differentiable curve $\xi(t)$, referred to as the shock curve, the quantities describing the fluid become discontinuous across this curve. And in the closures on each side of this curve, the differential equations are satisfied. The conservation of mass and momentum leads to the following equations, which describe the discontinuities and are known as the jump conditions:

\begin{equation}\label{eq3.11}
[\rho]V=[\rho w],
\end{equation}
\begin{equation}\label{eq3.12}
[\rho w]V=[\rho w^2+p],
\end{equation}
where $V=\frac{d\xi}{dt}$ is the speed of the shock, and $[f]$ denotes the jump that occurs after the function $f$ crosses the curve $\xi$, given by  $[f]=f_+-f_-,$ where $f_+$ and $f_-$ are the quantities in the state behind and ahead of the shock, respectively.

\begin{figure}[htbp]
  \centering
  \includegraphics{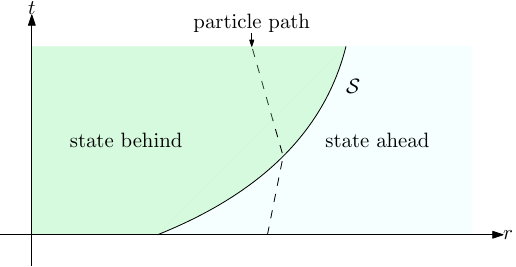}
  \caption{The state ahead and behind of the shock.}
  \label{p3}
\end{figure}

The state ahead and behind of the shock are illustrated in Figure \ref{p3}. Furthermore, the determinism condition states that the speed of the shock is supersonic with respect to the state ahead and subsonic with respect to the state behind.

\section{Shock wave interaction problem}\label{4}

\subsection{The states ahead of the emerged shocks}\label{4.1}

In this section, we discuss the shock wave interaction problem. We consider two sets of initial values of $\rho$ and $w$ on $r\ge r_0$ and $r\le r_0$ at time $t=0$. Each set of initial values has a future development, denoted by ${\overset{i}{\rho}}^*(t,r)$, ${\overset{i}{w}}^*(t,r)$, $i=1,2$.

They are solutions to the equations of motion and satisfy the initial values at $t=0$. These developments of initial data are bounded in the past by $t = 0$ and in the future by $t=\overset{i}{T},i=1,2$. Furthermore, the development of the data given for $r\ge r_0$ is bounded to the left by a right moving characteristic originating at the interaction point $(t,r)=(0,r_0)$ and the development of the data given for $r \le r_0$ is bounded to the right by a left moving characteristic originating at the interaction point $(t,r)=(0,r_0)$. We define the left moving characteristic as $\overset{1}{\mathcal{B}}$ and the right moving characteristic as $\overset{2}{\mathcal{B}}$.

At the point of interaction, the quantities ahead of the shocks are given by
\begin{equation}\label{eq4.1}
\overset{i}{\rho}_-={\overset{i}{\rho}}^*(0,0),\ \ \ \overset{i}{w}_-={\overset{i}{w}}^*(0,0),\ \ \ i=1,2.
\end{equation}
At the interaction point in the state behind, we obviously have
\begin{equation}\label{eq4.2}
\overset{1}{\rho}_+=\overset{2}{\rho}_+,\ \ \ \overset{1}{w}_+=\overset{2}{w}_+.
\end{equation}
Using the jump conditions \eqref{eq3.11} and \eqref{eq3.12}, we can calculate the four values at the interaction point in the state behind, denoted by 
\begin{equation}\label{eq4.3}
\overset{i}{\rho}_+:=\rho_0,\ \ \ \overset{i}{w}_+:=w_0,\ \ \ \overset{i}{V}:=\overset{i}{V}_0,\ \ \ i=1,2.
\end{equation}

These four quantities should satisfy the determinism condition corresponding to $\overset{1}{\mathcal{S}}$ and $\overset{2}{\mathcal{S}}$. After changing the frame of reference, we can set $w_0=0$ without affecting the determinism condition. The determinism condition for the left moving shock curve is 
\begin{equation}\label{eq4.4}
-\eta_0<\overset{1}{V}_0<({\overset{1}{c_{in0}}}^*)_0,
\end{equation}
where $\eta_0$ is the sound speed of the state behind, $({\overset{1}{c_{in0}}}^*)_0$  is the characteristic speed of the left moving characteristic in the state ahead of the left moving shock originating at the interaction point. All the above 3 quantities are negative, and we have

\begin{equation}\label{eq4.5}
\eta_0=\eta(\rho_0),\ \ \ ({\overset{1}{c_{in0}}}^*)_0={\overset{1}{w}}^*(0,0)-\eta({\overset{1}{\rho}}^*(0,0)).
\end{equation}
Similarly, the determinism condition for the right moving shock curve is 
\begin{equation}\label{eq4.6}
({\overset{2}{c_{out0}}}^*)_0<\overset{2}{V}_0<\eta_0,
\end{equation}
where $({\overset{2}{c_{out0}}}^*)_0$ is the characteristic speed of the right moving characteristic in the state ahead of the right moving shock originating at the interaction point. All the above 3 quantities are positive, and we have

\begin{equation}\label{eq4.7}
({\overset{2}{c_{out0}}}^*)_0={\overset{2}{w}}^*(0,0)+\eta({\overset{2}{\rho}}^*(0,0)).
\end{equation}

We obviously have
\begin{equation}\label{eq4.8}
\overset{1}{V}_0<0<\overset{2}{V}_0.
\end{equation}
The future development of the two sets of initial values, the solution of the jump conditions at the interaction point, and the determinism condition constitute the shock wave interaction problem, as shown in Figure \ref{p4}.

\begin{figure}[htbp]
  \centering
  \includegraphics{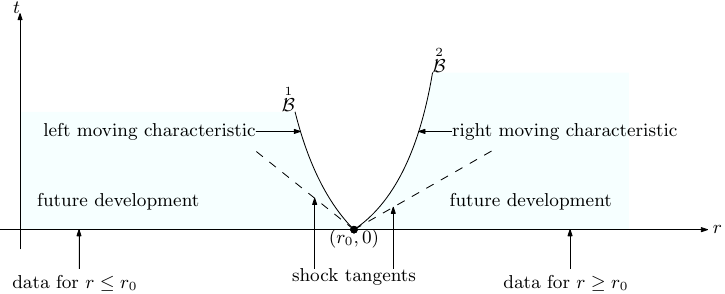}
  \caption{Future development of the two sets of initial values}
  \label{p4}
\end{figure}

\subsection{The description of the problem}\label{4.2}

The shock wave interaction problem seeks to determine two world lines, denoted by $\overset{i}{\mathcal{S}}$ for $i=1,2$, that originate from the interaction point $(t,r)=(0,r_0)$ and correspond to the future development of the two sets of initial values respectively, together with the solution of the equations of motion in the space time region limited by $\overset{i}{\mathcal{S}}$, such that crossing $\overset{i}{\mathcal{S}}$, the new solution in the state behind jumps with respect to the old solution in the future development, and the jump satisfies the jump conditions. In the region on the left side of $\overset{1}{\mathcal{S}}$, referred to as the state ahead of $\overset{1}{\mathcal{S}}$, the future development solution that corresponds to the initial value in $r\le r_0$ holds. In addition, in the region on the right side of $\overset{2}{\mathcal{S}}$, referred to as the state ahead of $\overset{2}{\mathcal{S}}$, the future development solution that corresponds to the initial value in $r\ge r_0$ holds. The region between $\overset{1}{\mathcal{S}}$ and $\overset{2}{\mathcal{S}}$ corresponds to the new solution, which is called the state behind of each $\overset{i}{\mathcal{S}}$. The determinism condition requires that each $\overset{i}{\mathcal{S}}$ is supersonic relative to the state ahead and subsonic relative to the state behind. We are going to bound the state behind also by a right moving characteristic, as shown in the dark shaded area in Figure \ref{p5}.

\begin{figure}[htbp]
  \centering
  \includegraphics{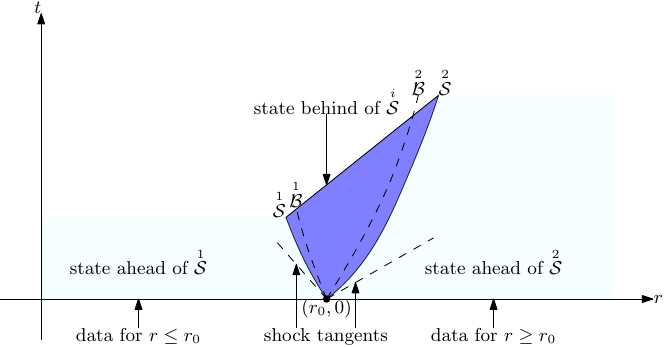}
  \caption{Shock wave interaction problem.}
  \label{p5}
\end{figure}

\section{Characteristic coordinates}\label{5}

\subsection{Choice of coordinates}\label{5.1}

By introducing coordinates $u$ and $v$, such that $u$ is constant along
integral curves of $L_{out}$ and $v$ is constant along integral curves of $L_{in}$. Then in $(u,v)$ coordinates, it is clear that $L_{out}u=0$ and $L_{in}v=0$. We define the value of $v$ on $r=r_0(t\ge 0)$ as $v:=\frac{\eta_0}{2}t$. Since the value of $v$ is unchanged on the integral curve corresponding to $L_{in}$, we can obtain the value of $v$ in the state behind. Naturally, we can obtain the value of $v$ in the state behind along $\overset{1}{\mathcal{S}}$. Let $u$ and $v$ on $\overset{1}{\mathcal{S}}$ be equal. Because $u$ is constant along the integral curve of $L_{out}$, we can obtain the value of $u$ and $v$ for the entire region in the state behind. In addition, $t$ and $r$ are obviously smooth with respect to the coordinates so determined. By calculating the Jacobian matrix of $t$ and $r$ with respect to $u$ and $v$, we get the following characteristic equations:

\begin{equation}\label{eq5.1}
\frac{\partial r}{\partial v}=c_{out}\frac{\partial t}{\partial v},\ \ \ \frac{\partial r}{\partial u}=c_{in}\frac{\partial t}{\partial u}.
\end{equation}

Through the transformation of coordinates, we can satisfy the following four conditions: 

(1) The origin $(u,v)=(0,0)$ corresponds to the interaction point $(t,r)=(0,r_0)$;

(2) The left moving shock $\overset{1}{\mathcal{S}}$ corresponds to $u=v$;

(3) $\frac{\partial r}{\partial v}(0,0)=1$;

(4) The right moving shock $\overset{2}{\mathcal{S}}$ corresponds to $u=av$, where $0<a<1$.

The feasibility of choosing these coordinates is as follows: Condition (1) is straightforward, requiring only that the origin of the coordinates be established at the interaction point $(0,r_0)$. Regarding condition (2), the expression corresponding to $\overset{1}{\mathcal{S}}$ already satisfies $u=v$. Furthermore, $(\alpha,\beta,t,r)$ is smooth with respect to $(u,v)$. Suppose that $r=r_0(t\ge 0)$ in the $(u,v)$ coordinate system is expressed as $u=g(v)$, where $g$ is a smooth function. According to the determinism condition, $0<g^{'}(0)<\infty$. At the origin, we have $\frac{\partial v}{\partial t}(0,r_0)=\frac{\eta_0}{2}$, $\frac{\partial u}{\partial t}(0,r_0)=\frac{\eta_0}{2}g^{'}(0)$.

Applying $L_{out}$ and $L_{in}$ to $u$ and $v$, respectively, we obtain

\begin{equation*}
L_{out} u=0=\frac{\partial u}{\partial t}+c_{out}\frac{\partial u}{\partial r},\ \ \ L_{in} v=0=\frac{\partial v}{\partial t}+c_{in}\frac{\partial v}{\partial r}.
\end{equation*}
Thus, at the origin, we get

\begin{equation}\label{eq5.2}
\frac{\partial(u,v)}{\partial(t,r)}(0,r_0)=\begin{pmatrix}
\frac{\eta_0}{2}g^{'}(0) & -\frac{g^{'}(0)}{2}\\
\frac{\eta_0}{2} & \frac{1}{2}
\end{pmatrix}.
\end{equation}
This matrix is non-degenerate, and its inverse matrix is

\begin{equation}\label{eq5.3}
\frac{\partial(t,r)}{\partial(u,v)}(0,0)=\begin{pmatrix}
\frac{1}{\eta_0 g^{'}(0)} & \frac{1}{\eta_0}\\
-\frac{1}{g^{'}(0)} & 1
\end{pmatrix}.
\end{equation}
Thus, we have satisfied condition (3).

To satisfy condition (4), we assume that the expression of the right moving shock curve in the $(u,v)$ coordinate system is $u=f(v)$, where $f(0)=0$ and $f$ is smooth. We compute the value of $f^{'}(0)$ by considering a vector that is tangent to the left moving shock curve $\overset{1}{\mathcal{S}}$ at the point of interaction: $\overset{1}{T}=\left(\partial_t+\overset{1}{V}_0 \partial_r\right)_{(0,0)}$, where $\overset{1}{V}_0$ is the speed of $\overset{1}{\mathcal{S}}$ at the interaction point calculated by jump conditions. Using the transformation of coordinates, we obtain an expression for $\overset{1}{T}$ in terms of the $(u,v)$ coordinates:

\begin{equation*}
\overset{1}{T}=\left(\frac{\eta_0}{2}g^{'}(0)-\frac{1}{2}g^{'}(0)\overset{1}{V}_0\right)\partial_u+\left(\frac{\eta_0}{2}+\frac{\overset{1}{V}_0}{2}\right)\partial_v.
\end{equation*}
Since the slope of $\overset{1}{S}$ in the $(u,v)$ coordinates is 1, we can equate the coefficients of the $\partial_u$ and $\partial_v$ terms to obtain

\begin{equation}\label{eq5.4}
g^{'}(0)=\frac{\eta_0+\overset{1}{V}_0}{\eta_0-\overset{1}{V}_0}.
\end{equation}

Similarly, we can compute a vector that is tangent to the right moving shock curve $\overset{2}{\mathcal{S}}$ at the point of interaction and obtain

\begin{equation}\label{eq5.5}
f^{'}(0)=\frac{\eta_0+\overset{1}{V}_0}{\eta_0-\overset{1}{V}_0}\frac{\eta_0-\overset{2}{V}_0}{\eta_0+\overset{2}{V}_0}.
\end{equation}

By the determinism condition, we know that 

\begin{equation}\label{eq5.6}
0<f^{'}(0)<1.
\end{equation}

We can apply a coordinate transformation to $(u,v)$ by defining $(\tilde{u},\tilde{v})=(\phi(u),\phi(v))$, where $\phi$ is a function that is strictly increasing near the origin and belongs to class $C^2$ such that $\phi(0)=0$ and $\phi^{'}(0)=1$. Therefore, $t$ and $r$ as functions of $\tilde{u}$ and $\tilde{v}$ belong to $C^2$. Moreover, we can show that conditions (1), (2), and (3) are still satisfied under this transformation. To construct $\phi$, we define iterative sequences 

\begin{equation*}
\phi_n(x):=\frac{f^{(n)}(x)}{a^n},\ \ \ x\in[0,\varepsilon^*],
\end{equation*}
where $a=f^{'}(0)$. Here $f^{(n)}$ defines the $n$-fold composition of the function $f$.

We aim to verify the equation of the right moving shock curve $\overset{2}{\mathcal{S}}$ in the new $(\tilde{u},\tilde{v})$ coordinate system. Let us assume that when $\varepsilon^*$ is sufficiently small, $\phi_n(x)$ converges uniformly, and its limit function is defined as $\phi(x)$, which satisfies the previously defined requirements. Then, by the definition of $\phi_n(x)$, we can obtain $\phi_n\circ f=a\phi_{n+1}.$

Upon taking the limit with respect to $n$, we can get condition (4): $\tilde{u}=\phi\circ f\circ\phi^{-1}(\tilde{v})=a\tilde{v}$. To prove the feasibility of condition (4), we first consider the expression for $\phi_{n+1}(x)$ obtained from the definition of $\phi_n(x)$, which gives $\phi_{n+1}(x)=\frac{f^{(n+1)}(x)}{a^{n+1}}$.

Expanding $f(x)$ and assuming $\varepsilon^*>0$, $\delta>0$ and $C>0$ such that

$$a(1+\delta)^2<1,|f(x)|\le (1+\delta)ax, \left|\frac{f(x)-ax}{x^2}\right|\le C,\left|f^{''}(x)\right|\le C,\forall |x|\le\varepsilon^*.$$
Then we can derive the following inequalities for $x\in [0,\varepsilon^*]$:

$$|f^{(n)}(x)|\le (1+\delta)^n a^n x,$$
$$|\phi_n(x)|=\left|\frac{f^{(n)}(x)}{a^n}\right|\le (1+\delta)^n x.$$

Thus, we find that

$$\left|\frac{\phi_{n+1}(x)-\phi_n(x)}{a^n \phi_n(x)}\right|=\left|\frac{f(a^n\phi_n(x))-a(a^n\phi_n(x))}{(a^n\phi_n(x))^2}\cdot a^{-1} \phi_n(x)\right|\le C\left|\phi_n(x)\right|.$$
This is because $a^n\phi_n(x)=f^{(n)}(x)\le x\le \varepsilon^*$. We need to clarify that the $C$ used here differs from the previous $C$ by a constant factor. Thus, we obtain

$$\left|\phi_{n+1}(x)-\phi_n(x)\right|\le Ca^n\left|\phi_n(x)\right|^2.$$
To prove that $\phi_n(x)$ converges uniformly, we consider the series $\mathop{\sum}\limits_{n=N}^{\infty}|\phi_{n+1}(x)-\phi_n(x)|$ for some integer $N>0$. Using the previous conclusion, we can derive the inequality

$$\sum_{n=N}^{\infty}|\phi_{n+1}(x)-\phi_n(x)|\le C \sum_{n=N}^{\infty}a^n (1+\delta)^{2n}(\varepsilon^*)^2.$$
Since $a(1+\delta)^2<1$, the above series converges, implying that $\phi_n(x)$ converges uniformly. So we have $\phi(0)=\mathop{lim}\limits_{n\rightarrow\infty}\phi_n(0)=0$. Taking the derivative of $\phi_n(x)$, we can obtain $\phi_n^{'}(x)=\frac{f^{'}(f^{(n-1)}(x))\cdot\cdot\cdot f^{'}(x)}{a^n}$. Using this result, we can rewrite the expression for $\phi_{n+1}^{'}(x)-\phi_n^{'}(x)$ as follows:

\begin{equation*}
\phi_{n+1}^{'}(x)-\phi_n^{'}(x)=\frac{[f^{'}(f^{(n)}(x))-a]f^{'}(f^{(n-1)}(x))\cdot\cdot\cdot f^{'}(x)}{a^{n+1}}.
\end{equation*}
Next, we use the fact that $\left|f^{'}(x)-f^{'}(0)\right|=\left|\int_0^x f^{''}(\xi)d\xi\right|\le Cx$, where $C$ is a fixed constant, to derive upper bounds for $\left|f^{'}(f^{(n)}(x))-a\right|$. Specifically, if we set $m=a(1+\delta)$, then we have $0<m<1$, and

\begin{equation*}
\left|f^{'}(f^{(n)}(x))-a\right|\le C\left|f^{(n)}(x)\right|\le C(1+\delta)^n a^n x =Cm^n x,
\end{equation*}
\begin{equation*}
\left|f^{'}(f^{(n-1)}(x))\right|\le a+C m^{n-1}x,
\end{equation*}
\begin{equation*}
\cdot\cdot\cdot,
\end{equation*}
\begin{equation*}
\left|f^{'}(x)\right|\le a+Cx.
\end{equation*}

We set $C_1:=C\varepsilon^*$ and $C_2:=\frac{C_1}{a}$. Using these constants, we can get the following estimates for the expression of $\mathop{\sum}\limits_{n=N}^{\infty}\left|\phi_{n+1}^{'}(x)-\phi_n^{'}(x)\right|$:

\begin{equation*}
\begin{aligned}
\mathop{\sum}\limits_{n=N}^{\infty}\left|\phi_{n+1}^{'}(x)-\phi_n^{'}(x)\right|&\le \mathop{\sum}\limits_{n=N}^{\infty} Cm^n x\frac{(a+Cm^{n-1}x)\cdot\cdot\cdot(a+Cx)}{a^{n+1}}\\
&\le \mathop{\sum}\limits_{n=N}^{\infty} C_1 m^n\frac{(a+C_1 m^{n-1})\cdot\cdot\cdot(a+C_1)}{a^{n+1}}\\
&\le \mathop{\sum}\limits_{n=N}^{\infty} C_2 m^n\underbrace{(1+C_2 m^{n-1})\cdot\cdot\cdot (1+C_2)}_{h_n}.
\end{aligned}
\end{equation*}
Note that $h_n$ is greater than 1. If we take the logarithm of it, we get

\begin{equation*}
log h_n=log (1+C_2)+\cdot\cdot\cdot +log(1+C_2 m^{n-1})\le \mathop{\sum}\limits_{n=0}^{\infty} log (1+C_2 m^n)\le \mathop{\sum}\limits_{n=0}^{\infty} C_2 m^n<log M,
\end{equation*}
where $M$ is a fixed constant. That is, the sequence $\{h_n\}$ has a uniform upper bound. Therefore, $\mathop{\sum}\limits_{n=N}^{\infty}\left| \phi_{n+1}^{'}(x)-\phi_n^{'}(x)\right|$ converges uniformly. According to the uniformly convergent property of the derivative sequence, the limit function $\phi$ is $C^1$ on the interval $[0,\varepsilon^*]$ and $\phi_n^{'}(x)$ converges uniformly to $\phi^{'}(x)$. So we have $\phi^{'}(0)=\mathop{lim}\limits_{n\rightarrow \infty} \phi_n^{'}(0)=1$.

We first note that $\phi(x)$ is strictly increasing in $[0,\varepsilon^*]$ when $\varepsilon^*$ is small enough and we can prove that $\phi\in C^2([0,\varepsilon^*])$. To demonstrate this, we take the derivative of $\phi_n^{'}(x)$ once more and then obtain $\phi_n^{''}(x)$, which has $n$ terms:

\begin{equation*}
\phi_n^{''}(x)=\frac{f^{''}(f^{(n-1)}(x))[f^{'}(f^{(n-2)}(x))\cdot\cdot\cdot f^{'}(x)]^2}{a^n}+\cdot\cdot\cdot+\frac{f^{'}(f^{(n-1)}(x))\cdot\cdot\cdot f^{''}(x)}{a^n}.
\end{equation*}
This leads us to the expression for $\left|\phi_{n+1}^{''}(x)\right|$ which has $(n+1)$ terms:

\begin{equation*}
\phi_{n+1}^{''}(x)=\frac{f^{''}(f^{(n)}(x))[f^{'}(f^{(n-1)}(x))\cdot\cdot\cdot f^{'}(x)]^2}{a^{n+1}}+\cdot\cdot\cdot+\frac{f^{'}(f^{(n)}(x))\cdot\cdot\cdot f^{''}(x)}{a^{n+1}}.
\end{equation*}
If we subtract these two equations, we arrive at the following expression:

\begin{equation*}
\left|\phi_{n+1}^{''}(x)-\phi_n^{''}(x)\right|=\left|\frac{f^{''}(f^{(n)}(x))[f^{'}(f^{(n-1)}(x))\cdot\cdot\cdot f^{'}(x)]^2}{a^{n+1}}+\frac{f^{'}(f^{(n)}(x))-a}{a}\phi_n^{''}(x)\right|.
\end{equation*}

In order to estimate the above equation, we show that $\phi_n^{''}(x)$ is uniformly bounded in the interval $[0,\varepsilon^*]$, by utilizing the expression for $\phi_n^{''}(x)$ and demonstrating that its $k$th item ($1\le k\le n$) is less than or equal to $CM^2 a^{k-2}$.

\begin{equation*}
\begin{aligned}
&\left|\frac{f^{'}(f^{(n-1)}(x))\cdot\cdot\cdot f^{''}(f^{(k-1)}(x))[f^{'}(f^{(k-2)}(x))\cdot\cdot\cdot f^{'}(x)]^2}{a^n}\right|\\
\le& \frac{C}{a^n} (a+Cm^{n-1}x)\cdot\cdot\cdot (a+Cm^k x)(a+Cm^{k-2}x)^2 \cdot\cdot\cdot(a+Cx)^2\\
\le& C a^{k-2} (1+C_2m^{n-1})\cdot\cdot\cdot (1+C_2m^k)(1+C_2m^{k-2})^2\cdot\cdot\cdot(1+C_2)^2\\
\le& Ch_n^2 a^{k-2}\\
\le& CM^2 a^{k-2}.
\end{aligned}
\end{equation*}
We conclude that $\phi_n^{''}(x)$ has a uniform bound
 
​\begin{equation*}
\left|\phi_n^{''}(x)\right|\le \mathop{\sum}\limits_{k=1}^{n} CM^2 a^{k-2}\le \frac{CM^2}{a(1-a)}.
\end{equation*}
Therefore,

\begin{equation*}
\begin{aligned}
\left|\phi_{n+1}^{''}(x)-\phi_n^{''}(x)\right|&\le C\frac{(a+Cm^{n-1}x)^2\cdot\cdot\cdot (a+Cx)^2}{a^{2n}}a^{n-1}+\frac{Cm^n x}{a}\left|\phi_n^{''}(x)\right|\\
&\le C (1+C_2 m^{n-1})^2\cdot\cdot\cdot (1+C_2)^2 a^{n-1}+\frac{C_2 C M^2}{a(1-a)} m^n\\
&\le CM^2 \left(a^{n-1}+\frac{C_2}{a(1-a)}m^n\right).
\end{aligned}
\end{equation*}
Thus, we get the series $\mathop{\sum}\limits_{n=N}^{\infty}\left|\phi_{n+1}^{''}(x)-\phi_n^{''}(x)\right|$ to converge uniformly. By using the uniformly convergent property of the derivative sequence, we observe that on the interval $[0,\varepsilon^*]$, the limiting function $\phi$ is $C^2$ and $\phi_n^{''}(x)$ converges uniformly to $\phi^{''}(x)$.

Next, we analyze the shock interaction problem in the region bounded by two shock curves $\overset{1}{\mathcal{S}}$ and $\overset{2}{\mathcal{S}}$, and the outgoing characteristic $u=a\varepsilon(\varepsilon\le \varepsilon^*)$. Let us define this region as $T_\varepsilon$ and express it mathematically as

\begin{equation}\label{eq5.7}
T_\varepsilon=\{(u,v)\in\mathbb{R}^2:0\le u\le v\le \frac{u}{a}\le \varepsilon\}.
\end{equation}

By performing a coordinate transformation, we can express $(\alpha,\beta,t,r)$ in terms of the transformed coordinates $(\tilde{u},\tilde{v})$ in $T_\varepsilon$, which is a set of $C^2$ functions with respect to $(\tilde{u},\tilde{v})$. For simplicity, we still refer to the new coordinates as $(u,v)$, and thus satisfy condition (4). A graphical representation of the region $T_\varepsilon$ is shown in Figure \ref{p6}.

\begin{figure}[htbp]
  \centering
  \includegraphics{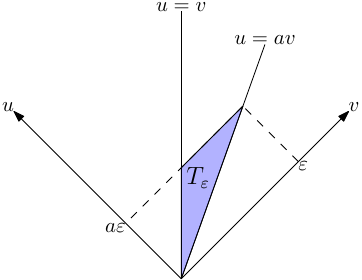}
  \caption{The region $T_\varepsilon$.}
  \label{p6}
\end{figure}

\subsection{Boundary conditions along the shock curve}\label{5.2}

This section discusses the boundary conditions for $\overset{1}{\mathcal{S}}$ and $\overset{2}{\mathcal{S}}$. The speeds of the two shocks are denoted as $\overset{i}{V}$, which satisfy the following boundary conditions:

\begin{equation}\label{eq5.8}
\overset{1}{V} d\overset{1}{t}_+=d\overset{1}{r}_+,\ \ \ \overset{2}{V} d\overset{2}{t}_+=d\overset{2}{r}_+.
\end{equation}
For any function $f(u,v)$, we use the notations $\overset{1}{f}_+(u)=f(u,u)$, $\overset{2}{f}_+(v)=f(av,v)$. Note that $u\in [0,a\varepsilon]$ and $v\in [0,\varepsilon]$.

Define

\begin{equation}\label{eq5.9}
\overset{1}{\Gamma}(u):=\frac{\overset{1}{c_{out+}}(u)}{\overset{1}{c_{in+}}(u)}\frac{\overset{1}{V}(u)-\overset{1}{c_{in+}}(u)}{\overset{1}{c_{out+}}(u)-\overset{1}{V}(u)},\ \ \ \overset{2}{\Gamma}(v):=a\frac{\overset{2}{c_{out+}}(v)}{\overset{2}{c_{in+}}(v)}\frac{\overset{2}{V}(v)-\overset{2}{c_{in+}}(v)}{\overset{2}{c_{out+}}(v)-\overset{2}{V}(v)}.
\end{equation}
According to the characteristic equations \eqref{eq5.1} obtained above, the boundary conditions \eqref{eq5.8} can be rewritten as

\begin{equation}\label{eq5.10}
\frac{\partial r}{\partial v}(u,u)=\frac{\partial r}{\partial u}(u,u)\overset{1}{\Gamma}(u),\ \ \ \frac{\partial r}{\partial v}(av,v)=\frac{\partial r}{\partial u}(av,v)\overset{2}{\Gamma}(v).
\end{equation}

The previous calculation yields the following partial derivative at point (0,0): $\frac{\partial r}{\partial u}(0,0)=-\frac{1}{g^{'}(0)}=-\frac{\eta_0-\overset{1}{V}_0}{\eta_0+\overset{1}{V}_0}$. By \eqref{eq5.10}, we can obtain

\begin{equation}\label{eq5.11}
\overset{1}{\Gamma}_0=\overset{2}{\Gamma}_0=\frac{1}{\frac{\partial r}{\partial u}(0,0)}=-g^{'}(0),
\end{equation}
which we refer to as $\Gamma_0$, where the subscript $0$ denotes the value at the interaction point. According to the determinism condition in \eqref{eq4.4}, it follows that $-\eta_0<\overset{1}{V}_0<w_0=0$. Thus, we get

\begin{equation}\label{eq5.12}
\Gamma_0=\overset{1}{\Gamma}_0=-\frac{\eta_0+\overset{1}{V}_0}{\eta_0-\overset{1}{V}_0}<0,
\end{equation}
and

\begin{equation}\label{eq5.13}
\frac{\partial r}{\partial u}(0,0)=\frac{1}{\Gamma_0}<0.
\end{equation}

The Riemann invariants $\alpha$ and $\beta$ along the left moving shock curve are

\begin{equation*}
\overset{1}{\alpha}_+(u)=\alpha(u,u),\ \ \ \overset{1}{\beta}_+(u)=\beta(u,u),
\end{equation*}
while the values along the right moving shock curve are

\begin{equation*}
\overset{2}{\alpha}_+(v)=\alpha(av,v),\ \ \ \overset{2}{\beta}_+(v)=\beta(av,v).
\end{equation*}

Based on previous calculations and the coordinate transformation, we can deduce that:

\begin{equation}\label{eq5.14}
L_{out}\alpha=\partial_t\alpha+c_{out}\partial_r\alpha=-\frac{2\eta w}{r}=\frac{\frac{\partial\alpha}{\partial v}}{\frac{\partial t}{\partial v}},
\end{equation}
\begin{equation}\label{eq5.15}
L_{in}\beta=\partial_t\beta+c_{in}\partial_r\beta=-\frac{2\eta w}{r}=\frac{\frac{\partial\beta}{\partial u}}{\frac{\partial t}{\partial u}}.
\end{equation}
These indicate that

\begin{equation}\label{eq5.16}
\frac{\partial\alpha}{\partial v}=\frac{\partial t}{\partial v} A(\alpha,\beta,r),\ \ \ \frac{\partial\beta}{\partial u}=\frac{\partial t}{\partial u} A(\alpha,\beta,r).
\end{equation}
Here

\begin{equation}\label{eq5.17}
A(\alpha,\beta,r)=-\frac{2\eta w}{r}.
\end{equation}
Using the definitions of $\overset{i}{\alpha}_+$ and $\overset{i}{\beta}_+$, $i=1,2$, we can write:

\begin{equation*}
\begin{cases}
&\overset{1}{\beta}_+(u)=\overset{2}{\beta}_+(u)+\int_{au}^u \left(\frac{\partial t}{\partial u}A\right)(\tau,u) d\tau,\\
&\overset{2}{\alpha}_+(v)=\overset{1}{\alpha}_+(av)+\int_{av}^v\left(\frac{\partial t}{\partial v} A\right)(av,s)ds.
\end{cases}
\end{equation*}

Note that the superscripts $1$ and $2$ denote whether the function corresponds to the left moving or right moving shock curve. Additionally, they indicate the range of values of the independent variables. For instance, $\overset{1}{\beta}_+(u)$ means that $u\in [0,a\varepsilon]$, while $\overset{2}{\beta}_+(u)$ means that $u\in [0,\varepsilon]$.

\subsection{Jump conditions and initial value calculation}\label{5.3}

The equations \eqref{eq3.11} and \eqref{eq3.12} that define the jump conditions are equivalent to

\begin{equation}\label{eq5.18}
V=\frac{[\rho w]}{[\rho]},
\end{equation}
\begin{equation}\label{eq5.19}
0=[\rho w]^2-[\rho w^2+p][\rho]=:I(\rho_+,\rho_-,w_+,w_-).
\end{equation}
By the definition of Riemann invariants, we can consider $\rho$ and $w$ as smooth functions of $\alpha$ and $\beta$. We can then define the function

\begin{equation}\label{eq5.20}
J(\alpha_+,\beta_+,\alpha_-,\beta_-):=I\left(\rho(\alpha_+,\beta_+),\rho(\alpha_-,\beta_-),w(\alpha_+,\beta_+),w(\alpha_-,\beta_-)\right),
\end{equation}
which means that the second jump condition is equivalent to

\begin{equation}\label{eq5.21}
J(\alpha_+,\beta_+,\alpha_-,\beta_-)=0.
\end{equation}
Following from the definition of Riemann invariants, we have $w_0=\frac{\alpha_0-\beta_0}{2}=0$, implying that $\alpha_0=\beta_0$. This means that at the point of interaction,

\begin{equation}\label{eq5.22}
J(\beta_0,\beta_0,\overset{i}{\alpha}_{-0},\overset{i}{\beta}_{-0})=0,\ \ \ i=1,2,
\end{equation}
where the values in the state ahead at the interaction point are given by $\overset{i}{\alpha}_{-0}={\overset{i}{\alpha}}^*(0,0)$, ${\overset{i}{\beta}}_{-0}={\overset{i}{\beta}}^*(0,0)$, $i=1,2$.

These quantities can be obtained by the solution in the state ahead. Calculating further, we obtain

\begin{equation}\label{eq5.23}
\frac{\partial J}{\partial\alpha_+}(\alpha_+,\beta_+,\alpha_-,\beta_-)=-\frac{[\rho]\rho_+}{2\eta_+}(c_{out+}-V)^2,
\end{equation}

and

\begin{equation}\label{eq5.24}
\frac{\partial J}{\partial\beta_+}(\alpha_+,\beta_+,\alpha_-,\beta_-)=-\frac{[\rho]\rho_+}{2\eta_+}(V-c_{in+})^2.
\end{equation}

The right hand side of the above two equations depends on the variables $\rho_{\pm}$ and $w_{\pm}$, as well as $\alpha_{\pm}$ and $\beta_{\pm}$. Furthermore, based on the earlier analysis in \eqref{eq4.8}, we can infer that $\overset{1}{V}_0-\overset{1}{c_{out}}_{+0}=\overset{1}{V}_0-\eta_0<0$ and $\overset{2}{V}_0-\overset{2}{c_{in}}_{+0}=\overset{2}{V}_0+\eta_0>0$ hold true at the point of interaction.

So at the point of interaction, we have

\begin{equation}\label{eq5.25}
\frac{\partial J}{\partial\alpha_+}(\beta_0,\beta_0,\overset{1}{\alpha}_{-0},\overset{1}{\beta}_{-0})\neq 0,
\end{equation}
\begin{equation}\label{eq5.26}
\frac{\partial J}{\partial \beta_+}(\beta_0,\beta_0,\overset{2}{\alpha}_{-0},\overset{2}{\beta}_{-0})\neq 0.
\end{equation}
By the implicit function theorem, there exist smooth functions $\overset{1}{H}$ and $\overset{2}{H}$ that describe the relations among the Riemann invariants. When $\varepsilon$ is small enough, we have

\begin{equation}\label{eq5.27}
\overset{1}{\alpha}_+(u)=\overset{1}{H}(\overset{1}{\beta}_+(u),\overset{1}{\alpha}_-(u),\overset{1}{\beta}_-(u)),
\end{equation}
\begin{equation}\label{eq5.28}
\overset{2}{\beta}_+(v)=\overset{2}{H}(\overset{2}{\alpha}_+(v),\overset{2}{\alpha}_-(v),\overset{2}{\beta}_-(v)).
\end{equation}
So we have

\begin{equation}\label{eq5.29}
J\left(\overset{1}{H}(\overset{1}{\beta}_+,\overset{1}{\alpha}_-,\overset{1}{\beta}_-),\overset{1}{\beta}_+,\overset{1}{\alpha}_-,\overset{1}{\beta}_-\right)=0,
\end{equation}
\begin{equation}\label{eq5.30}
J\left(\overset{2}{\alpha}_+,\overset{2}{H}(\overset{2}{\alpha}_+,\overset{2}{\alpha}_-,\overset{2}{\beta}_-),\overset{2}{\alpha}_-,\overset{2}{\beta}_-\right)=0.
\end{equation}
This means that, at the point of interaction, $\beta_0$ is a common value of $\overset{1}{H}$ and $\overset{2}{H}$, namely $\beta_0=\overset{1}{H}(\beta_0,\overset{1}{\alpha}_{-0},\overset{1}{\beta}_{-0})=\overset{2}{H}(\beta_0,\overset{2}{\alpha}_{-0},\overset{2}{\beta}_{-0})$.

From \eqref{eq5.27}, we can obtain 

\begin{equation}\label{eq5.31}
\begin{aligned}
\frac{d \overset{1}{\alpha}_+}{du}(u)=&\overset{1}{F}(\overset{1}{\beta}_+(u),\overset{1}{\alpha}_-(u),\overset{1}{\beta}_-(u))\frac{d\overset{1}{\beta}_+}{du}(u)\\
&+\overset{1}{M}_1(\overset{1}{\beta}_+(u),\overset{1}{\alpha}_-(u),\overset{1}{\beta}_-(u))\frac{d\overset{1}{\alpha}_-}{du}(u)\\
&+\overset{1}{M}_2(\overset{1}{\beta}_+(u),\overset{1}{\alpha}_-(u),\overset{1}{\beta}_-(u))\frac{d\overset{1}{\beta}_-}{du}(u)
\end{aligned}
\end{equation}
by taking the derivative of $\overset{1}{\alpha}_+$ with respect to $u$, where

\begin{equation}\label{eq5.32}
\overset{1}{F}(\beta_+,\alpha_-,\beta_-):=\frac{\partial\overset{1}{H}}{\partial\beta_+}(\beta_+,\alpha_-,\beta_-)=-\frac{\frac{\partial J}{\partial\beta_+}\left(\overset{1}{H}(\beta_+,\alpha_-,\beta_-),\beta_+,\alpha_-,\beta_-\right)}{\frac{\partial J}{\partial\alpha_+}\left(\overset{1}{H}(\beta_+,\alpha_-,\beta_-),\beta_+,\alpha_-,\beta_-\right)}.
\end{equation}
Expressions for $\overset{1}{M}_1$ and $\overset{1}{M}_2$ are similar, and based on equations \eqref{eq5.23} and \eqref{eq5.24}, we obtain the following expression:

\begin{equation}\label{eq5.33}
\overset{1}{F}(\overset{1}{\beta}_+,\overset{1}{\alpha}_-,\overset{1}{\beta}_-)=-\left(\frac{\overset{1}{V}-\overset{1}{c_{in+}}}{\overset{1}{c_{out+}}-\overset{1}{V}}\right)^2.
\end{equation}
Similarly, we have

\begin{equation}\label{eq5.34}
\frac{d\overset{2}{\beta}_+}{dv}(v)=\overset{2}{F}(\overset{2}{\alpha}_+(v),\overset{2}{\alpha}_-(v),\overset{2}{\beta}_-(v))\frac{d\overset{2}{\alpha}_+}{dv}(v)+\cdot\cdot\cdot,
\end{equation}
where $\cdot\cdot\cdot$ corresponds to the last two lines in equation \eqref{eq5.31}, where

\begin{equation}\label{eq5.35}
\overset{2}{F}(\alpha_+,\alpha_-,\beta_-):=\frac{\partial \overset{2}{H}}{\partial\alpha_+}(\alpha_+,\alpha_-,\beta_-)=-\frac{\frac{\partial J}{\partial\alpha_+}\left(\alpha_+,\overset{2}{H}(\alpha_+,\alpha_-,\beta_-),\alpha_-,\beta_-\right)}{\frac{\partial J}{\partial\beta_+}\left(\alpha_+,\overset{2}{H}(\alpha_+,\alpha_-,\beta_-),\alpha_-,\beta_-\right)}.
\end{equation}
Similarly, according to \eqref{eq5.23} and \eqref{eq5.24}, we have

\begin{equation}\label{eq5.36}
\overset{2}{F}(\overset{2}{\alpha}_+,\overset{2}{\alpha}_-,\overset{2}{\beta}_-)=-\left(\frac{\overset{2}{c_{out+}}-\overset{2}{V}}{\overset{2}{V}-\overset{2}{c_{in+}}}\right)^2.
\end{equation}

At the interaction point, according to \eqref{eq5.5}, we can know

\begin{equation}\label{eq5.37}
\overset{1}{F}_0\overset{2}{F}_0=\left(\frac{\overset{1}{V}_0+\eta_0}{\eta_0-\overset{1}{V}_0}\frac{\eta_0-\overset{2}{V}_0}{\overset{2}{V}_0+\eta_0}\right)^2=a^2,
\end{equation}
where $\overset{i}{F}_0=\overset{i}{F}(\beta_0,\overset{i}{\alpha}_{-0},\overset{i}{\beta}_{-0}),\ \ \ i=1,2$.

We define

\begin{equation}\label{eq5.38}
\alpha^{'}_0:=\frac{d\overset{1}{\alpha}_+}{du}(0),\ \ \ \beta^{'}_0:=\frac{d\overset{2}{\beta}_+}{dv}(0).
\end{equation}
Since we have $A_0=A(\beta_0,\beta_0,r_0)=0$ at the interaction point, based on the definitions of $\overset{i}{\alpha}_+$ and $\overset{i}{\beta}_+$, we have $\frac{\partial\alpha}{\partial v}(0,0)=\frac{\partial\beta}{\partial u}(0,0)=0$ and

\begin{equation*}
\begin{cases}
&\alpha^{'}_0=\frac{d\overset{1}{\alpha}_+}{du}(0)=\frac{1}{a}\frac{d\overset{2}{\alpha}_+}{dv}(0),\\
&\beta^{'}_0=\frac{d\overset{2}{\beta}_+}{dv}(0)=\frac{d\overset{1}{\beta}_+}{du}(0),
\end{cases}
\end{equation*}
which implies

\begin{equation}\label{eq5.39}
\alpha^{'}_0=\frac{\partial\alpha}{\partial u}(0,0),\ \ \ \beta^{'}_0=\frac{\partial\beta}{\partial v}(0,0).
\end{equation}
With these equations, we can write

\begin{equation}\label{eq5.40}
\alpha^{'}_0=\overset{1}{F}_0\beta^{'}_0+\overset{1}{M}_{10}\frac{d\overset{1}{\alpha}_-}{du}(0)+\overset{1}{M}_{20}\frac{d\overset{1}{\beta}_-}{du}(0),
\end{equation}
\begin{equation}\label{eq5.41}
\beta_0^{'}=a\overset{2}{F}_0\alpha^{'}_0+\overset{2}{M}_{10}\frac{d\overset{2}{\alpha}_-}{dv}(0)+\overset{2}{M}_{20}\frac{d\overset{2}{\beta}_-}{dv}(0).
\end{equation}

For the Riemann invariants in the state ahead, we have

\begin{equation}\label{eq5.42}
\overset{1}{\alpha}_-(u)={\overset{1}{\alpha}}^*(\overset{1}{t}_+(u),\overset{1}{r}_+(u)),\ \ \ \overset{1}{\beta}_-(u)={\overset{1}{\beta}}^*(\overset{1}{t}_+(u),\overset{1}{r}_+(u)).
\end{equation}
By taking the derivative at the interaction point, we can derive

\begin{equation*}
\begin{cases}
&\frac{d\overset{1}{\alpha}_-}{du}(0)=(\frac{\partial{\overset{1}{\alpha}}^*}{\partial t})_0\frac{d\overset{1}{t}_+}{du}(0)+(\frac{\partial{\overset{1}{\alpha}}^*}{\partial r})_0\frac{d\overset{1}{r}_+}{du}(0),\\
&\frac{d\overset{1}{\beta}_-}{du}(0)=(\frac{\partial{\overset{1}{\beta}}^*}{\partial t})_0\frac{d\overset{1}{t}_+}{du}(0)+(\frac{\partial{\overset{1}{\beta}}^*}{\partial r})_0\frac{d\overset{1}{r}_+}{du}(0).
\end{cases}
\end{equation*}

The partial derivatives of ${\overset{1}{\alpha}}^*(t,r)$ and ${\overset{1}{\beta}}^*(t,r)$ can be obtained from the solutions in the state ahead of the left moving shock. Specifically, taking the derivative with respect to $u$ in the equations $\overset{1}{t}_+(u)=t(u,u)$ and $\overset{1}{r}_+(u)=r(u,u)$ yields

\begin{equation*}
\frac{d\overset{1}{t}_+}{du}(u)=\frac{\partial t}{\partial u}(u,u)+\frac{\partial t}{\partial v}(u,u),\ \ \ \frac{d\overset{1}{r}_+}{du}(u)=\frac{\partial r}{\partial u}(u,u)+\frac{\partial r}{\partial v}(u,u).
\end{equation*}

Since the value of $\Gamma_0$ has already been obtained in \eqref{eq5.12}, using \eqref{eq5.11}, we can give the Jacobian matrix of $t$ and $r$ with respect to $u$ and $v$ at the point $(0,0)$ as

\begin{equation}\label{eq5.43}
\frac{\partial(t,r)}{\partial(u,v)}(0,0)=\begin{pmatrix}
-\frac{1}{\eta_0}\frac{1}{\Gamma_0} & \frac{1}{\eta_0}\\
\frac{1}{\Gamma_0} & 1\\
\end{pmatrix}.
\end{equation}
Therefore, we have $\frac{d\overset{1}{t}_+}{du}(0)=\frac{1}{\eta_0}\left(-\frac{1}{\Gamma_0}+1\right)$, $\frac{d\overset{1}{r}_+}{du}(0)=\frac{1}{\Gamma_0}+1$, and 

\begin{equation}\label{eq5.44}
\frac{d\overset{1}{\alpha}_-}{du}(0)=\left(\frac{\partial{\overset{1}{\alpha}}^*}{\partial t}\right)_0\frac{1}{\eta_0}\left(-\frac{1}{\Gamma_0}+1\right)+\left(\frac{\partial{\overset{1}{\alpha}}^*}{\partial r}\right)_0\left(\frac{1}{\Gamma_0}+1\right),
\end{equation}
\begin{equation}\label{eq5.45}
\frac{d\overset{1}{\beta}_-}{du}(0)=\left(\frac{\partial{\overset{1}{\beta}}^*}{\partial t}\right)_0\frac{1}{\eta_0}\left(-\frac{1}{\Gamma_0}+1\right)+\left(\frac{\partial{\overset{1}{\beta}}^*}{\partial r}\right)_0\left(\frac{1}{\Gamma_0}+1\right).
\end{equation}

Similarly, consider $\overset{2}{t}_+(v)=t(av,v)$ and $\overset{2}{r}_+(v)=r(av,v)$, we arrive at

\begin{equation*}
\frac{d\overset{2}{t}_+}{dv}(v)=a\frac{\partial t}{\partial u}(av,v)+\frac{\partial t}{\partial v}(av,v),\ \ \ \frac{d\overset{2}{r}_+}{dv}(v)=a\frac{\partial r}{\partial u}(av,v)+\frac{\partial r}{\partial v}(av,v).
\end{equation*}
Therefore, we have

\begin{equation}\label{eq5.46}
\frac{d\overset{2}{\alpha}_-}{dv}(0)=\left(\frac{\partial{\overset{2}{\alpha}}^*}{\partial t}\right)_0\frac{1}{\eta_0}\left(-\frac{a}{\Gamma_0}+1\right)+\left(\frac{\partial{\overset{2}{\alpha}}^*}{\partial r}\right)_0\left(\frac{a}{\Gamma_0}+1\right),
\end{equation}
\begin{equation}\label{eq5.47}
\frac{d\overset{2}{\beta}_-}{dv}(0)=\left(\frac{\partial{\overset{2}{\beta}}^*}{\partial t}\right)_0\frac{1}{\eta_0}\left(-\frac{a}{\Gamma_0}+1\right)+\left(\frac{\partial{\overset{2}{\beta}}^*}{\partial r}\right)_0\left(\frac{a}{\Gamma_0}+1\right).
\end{equation}

Equations \eqref{eq5.40} and \eqref{eq5.41} can be expressed in terms of a matrix equation as follows:

\begin{equation}\label{eq5.48}
M\begin{pmatrix}
\alpha^{'}_0\\
\beta^{'}_0
\end{pmatrix}=\begin{pmatrix}
a_0\\
b_0
\end{pmatrix},
\end{equation}
where

\begin{equation*}
M=\begin{pmatrix}
1 & -\overset{1}{F}_0\\
-a\overset{2}{F}_0 & 1
\end{pmatrix},
\end{equation*}
\begin{equation*}
\begin{cases}
&a_0=\overset{1}{M}_{10}\frac{d\overset{1}{\alpha}_-}{du}(0)+\overset{1}{M}_{20}\frac{d\overset{1}{\beta}_-}{du}(0),\\
&b_0=\overset{2}{M}_{10}\frac{d\overset{2}{\alpha}_-}{dv}(0)+\overset{2}{M}_{20}\frac{d\overset{2}{\beta}_-}{dv}(0).
\end{cases}
\end{equation*}
Using equation \eqref{eq5.37}, we can show that the determinant of $M$ is given by $det(M)=1-a\overset{1}{F}_0\overset{2}{F}_0=1-a^3>0$. Since $det(M)$ is nonzero, $M$ is invertible. Thus, we can solve for $\alpha^{'}_0$ and $\beta^{'}_0$ as

\begin{equation}\label{eq5.49}
\begin{pmatrix}
\alpha^{'}_0\\
\beta^{'}_0
\end{pmatrix}=M^{-1}\begin{pmatrix}
a_0\\
b_0
\end{pmatrix}.
\end{equation}
Moreover, because the matrix $M$ is known, we can calculate the inverse matrix of $M$. And $a_0$ and $b_0$ can be obtained by the values in the state ahead at the point of interaction, the above equation can be regarded as the definition of the constants $\alpha^{'}_0$ and $\beta^{'}_0$.

\subsection{Equations for $\frac{\partial r}{\partial u}$ and $\frac{\partial r}{\partial v}$}\label{5.4}

In this section, we derive equations for $\frac{\partial r}{\partial u}$ and $\frac{\partial r}{\partial v}$ at a point $(u,v)$ in $T_\varepsilon$. To obtain $\frac{\partial r}{\partial u}$, we integrate along a right moving characteristic from the point $(u,u)$ on the left moving shock curve $\{u=v\}$ to the point $(u,v)$. This yields the following equation:

\begin{equation}\label{eq5.50}
\frac{\partial r}{\partial u}(u,v)=\frac{\partial r}{\partial u}(u,u)+\int_u^v\frac{\partial^2 r}{\partial u\partial v}(u,v^{'})dv^{'}.
\end{equation}
We integrate along the characteristic, and the first term in the right side of the above equation can be rewritten as

\begin{equation*}
\frac{\partial r}{\partial u}(u,u)=\frac{1}{\overset{1}{\Gamma}(u)}\left(\overset{2}{\Gamma}(u)\frac{\partial r}{\partial u}(au,u)+\int_{au}^u\frac{\partial^2 r}{\partial u\partial v}(u^{'},u)du^{'}\right).
\end{equation*}
So we have

\begin{equation}\label{eq5.51}
\frac{\partial r}{\partial u}(u,v)=\frac{\overset{2}{\Gamma}(u)}{\overset{1}{\Gamma}(u)}\frac{\partial r}{\partial u}(au,u)+\frac{1}{\overset{1}{\Gamma}(u)}\int_{au}^u\frac{\partial^2 r}{\partial u\partial v}(u^{'},u)du^{'}+\int_u^v\frac{\partial^2 r}{\partial u\partial v}(u,v^{'})dv^{'}.
\end{equation}

Similarly, we integrate from the point $(av,v)$ on the right moving shock curve along the left moving characteristic to the point $(u,v)$ to obtain the equation for $\frac{\partial r}{\partial v}(u,v)$:

\begin{equation}\label{eq5.52}
\frac{\partial r}{\partial v}(u,v)=\frac{\partial r}{\partial v}(av,v)+\int_{av}^u\frac{\partial^2 r}{\partial u\partial v}(u^{'},v)du^{'}.
\end{equation}
Using the boundary conditions, the first term on the right side of the above equation can be rewritten as

\begin{equation*}
\frac{\partial r}{\partial v}(av,v)=\overset{2}{\Gamma}(v)\left(\frac{1}{\overset{1}{\Gamma}(av)}\frac{\partial r}{\partial  v}(av,av)+\int_{av}^v\frac{\partial^2 r}{\partial u\partial v}(av,v^{'})dv^{'}\right).
\end{equation*}
So we have

\begin{equation}\label{eq5.53}
\frac{\partial r}{\partial v}(u,v)=\frac{\overset{2}{\Gamma}(v)}{\overset{1}{\Gamma}(av)}\frac{\partial r}{\partial v}(av,av)+\overset{2}{\Gamma}(v)\int_{av}^v\frac{\partial^2 r}{\partial u\partial v}(av,v^{'})dv^{'}+\int_{av}^u \frac{\partial^2 r}{\partial u\partial v}(u^{'},v)du^{'}.
\end{equation}

By taking the derivative of the characteristic equations in \eqref{eq5.1} and eliminating $t$, the following equation can be obtained:

\begin{equation*}
\frac{\partial^2 r}{\partial u\partial v}=\frac{1}{c_{out}-c_{in}}\left(\frac{c_{out}}{c_{in}}\frac{\partial c_{in}}{\partial v}\frac{\partial r}{\partial u}-\frac{c_{in}}{c_{out}}\frac{\partial c_{out}}{\partial u}\frac{\partial r}{\partial v}\right).
\end{equation*}
Introducing the variables 

\begin{equation}\label{eq5.54}
\mu:=\frac{1}{c_{out}-c_{in}}\frac{c_{out}}{c_{in}}\frac{\partial c_{in}}{\partial v},\ \ \ \nu:=-\frac{1}{c_{out}-c_{in}}\frac{c_{in}}{c_{out}}\frac{\partial c_{out}}{\partial u},
\end{equation}
the above equation takes the form

\begin{equation}\label{eq5.55}
\frac{\partial^2 r}{\partial u\partial v}(u,v)=\mu(u,v)\frac{\partial r}{\partial u}(u,v)+\nu(u,v)\frac{\partial r}{\partial v}(u,v).
\end{equation}

Figure \ref{p7} shows the integral paths of $\frac{\partial r}{\partial u}(u,v)$ (left) and $\frac{\partial r}{\partial v}(u,v)$ (right).

\begin{figure}[htbp]
  \centering
  \includegraphics{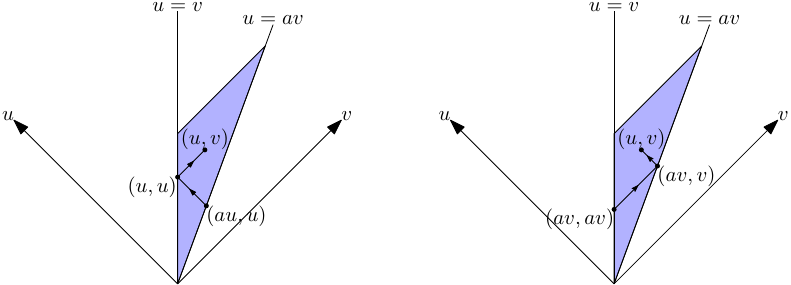}
  \caption{The left is the integral path of $\frac{\partial r}{\partial u}(u,v)$, and the right is the integral path of $\frac{\partial r}{\partial v}(u,v)$.}
  \label{p7}
\end{figure}

\subsection{Formulation of the problem}\label{5.5}

The problem of shock interaction can be formulated as follows: Find a solution of the following equations in $T_\varepsilon$:

\begin{equation}\label{eq5.56}
\frac{\partial r}{\partial u}=c_{in}\frac{\partial t}{\partial u},
\end{equation}
\begin{equation}\label{eq5.57}
\frac{\partial r}{\partial v}=c_{out}\frac{\partial t}{\partial v},
\end{equation}
\begin{equation}\label{eq5.58}
\frac{\partial\alpha}{\partial v}=A\frac{\partial t}{\partial v},
\end{equation}
\begin{equation}\label{eq5.59}
\frac{\partial\beta}{\partial u}=A\frac{\partial t}{\partial u}.
\end{equation}
Furthermore, we require that the boundary conditions 

\begin{equation}\label{eq5.60}
\frac{d\overset{1}{r}_+}{du}(u)=\overset{1}{V}(u)\frac{d\overset{1}{t}_+}{du}(u)
\end{equation}
\begin{equation}\label{eq5.61}
\frac{d\overset{2}{r}_+}{dv}(v)=\overset{2}{V}(v)\frac{d\overset{2}{t}_+}{dv}(v)
\end{equation}
along the two shocks be satisfied, where $\overset{1}{r}_+(u)=r(u,u)$, $\overset{1}{t}_+(u)=t(u,u)$, $\overset{2}{r}_+(v)=r(av,v)$, $\overset{2}{t}_+(v)=t(av,v)$ and

\begin{equation}\label{eq5.62}
\overset{i}{V}=\frac{[\overset{i}{\rho}\overset{i}{w}]}{[\overset{i}{\rho}]},\ \ \ i=1,2,
\end{equation}
where $\overset{1}{\rho}_{\pm}(u)=\rho(\overset{1}{\alpha}_{\pm}(u),\overset{1}{\beta}_{\pm}(u))$, $\overset{1}{w}_{\pm}(u)=w(\overset{1}{\alpha}_{\pm}(u),\overset{1}{\beta}_{\pm}(u))$, $\overset{2}{\rho}_{\pm}(v)=\rho(\overset{2}{\alpha}_{\pm}(v),\overset{2}{\beta}_{\pm}(v))$, $\overset{2}{w}_{\pm}(v)=w(\overset{2}{\alpha}_{\pm}(v),\overset{2}{\beta}_{\pm}(v))$, where $\rho(\alpha,\beta)$ and $w(\alpha,\beta)$ are smooth functions of their components and $\overset{1}{\alpha}_+(u)=\alpha(u,u)$, $\overset{1}{\beta}_+(u)=\beta(u,u)$, $\overset{1}{\alpha}_-(u)={\overset{1}{\alpha}}^*(\overset{1}{t}_+(u),\overset{1}{r}_+(u))$, $\overset{1}{\beta}_-(u)={\overset{1}{\beta}}^*(\overset{1}{t}_+(u),\overset{1}{r}_+(u))$, $\overset{2}{\alpha}_+(v)=\alpha(av,v)$, $\overset{2}{\beta}_+(v)=\beta(av,v)$, $\overset{2}{\alpha}_-(v)={\overset{2}{\alpha}}^*(\overset{2}{t}_+(v),\overset{2}{r}_+(v))$, $\overset{2}{\beta}_-(v)={\overset{2}{\beta}}^*(\overset{2}{t}_+(v),\overset{2}{r}_+(v))$.

Additionally, across the two shocks, we need the jump conditions to hold as follows:

\begin{equation}\label{eq5.63}
J(\overset{i}{\alpha}_+,\overset{i}{\beta}_+,\overset{i}{\alpha}_-,\overset{i}{\beta}_-)=0,\ \ \ i=1,2.
\end{equation}
And at the interaction point $(u,v)=(0,0)$, we should have

\begin{equation}\label{eq5.64}
\alpha(0,0)=\beta(0,0)=\beta_0,\ \ \ \frac{\partial\alpha}{\partial u}(0,0)=\alpha^{'}_0,\ \ \ \frac{\partial\beta}{\partial v}(0,0)=\beta^{'}_0,
\end{equation}
\begin{equation}\label{eq5.65}
\overset{i}{V}(0)=\overset{i}{V}_0,\ \ \ i=1,2.
\end{equation}

Finally, the determinism condition has to be satisfied:

\begin{equation}\label{eq5.66}
\overset{1}{c_{in}}_+(u)<\overset{1}{V}(u)<\overset{1}{c_{in}}_-(u),
\end{equation}
\begin{equation}\label{eq5.67}
\overset{2}{c_{out}}_-(v)<\overset{2}{V}(v)<\overset{2}{c_{out}}_+(v),
\end{equation}
where $\overset{1}{c_{in}}_\pm(u)=c_{in}(\overset{1}{\alpha}_{\pm}(u),\overset{1}{\beta}_{\pm}(u))$, $\overset{2}{c_{out}}_\pm(v)=c_{out}(\overset{2}{\alpha}_{\pm}(v),\overset{2}{\beta}_{\pm}(v))$.

\section{Solution of the interaction problem}\label{6}

Next, we will prove the existence(Section \ref{6.5}), uniqueness(Section \ref{6.7}), and higher regularity(Section \ref{7}) of the solution to the shock interaction problem. We begin by stating the following theorems:

\begin{theorem}[Existence]\label{existence}
For sufficiently small $\varepsilon$, there exists a solution $\alpha$, $\beta$, $t$, $r$ in $C^2(T_\varepsilon)$ to the shock interaction problem. Furthermore, for sufficiently small $\varepsilon$, the Jacobian matrix $\frac{\partial(t,r)}{\partial(u,v)}$ does not vanish in $T_\varepsilon$, which implies that $\alpha$ and $\beta$ are $C^2$ functions of $(t,r)$ on the image of $T_\varepsilon$ by the map $(u,v)\longmapsto (t(u,v),r(u,v))$. 
\end{theorem}

\begin{theorem}[Uniqueness]\label{uniqueness}
Suppose that $(\alpha_1,\beta_1,t_1,r_1)$ and $(\alpha_2,\beta_2,t_2,r_2)$ are solutions to the shock interaction problem presented in Section 5.5, with the same future developments of the data. When $\varepsilon$ is sufficiently small, the two sets of solutions are consistent.
\end{theorem}

\begin{theorem}[Higher Regularity]\label{higher regularity}
For sufficiently small $\varepsilon$, the established solution of the interaction problem, comprising $\alpha$, $\beta$, $t$, $r$, is infinitely differentiable.
\end{theorem}

In this section, we will provide the proof of Theorem 1 and Theorem 2. The proof of Theorem 1 is based on the iterations presented in the following subsections. We have developed a constructive approach that uses an iterative scheme to obtain a solution to the problem.

We present the details of the iterative scheme in the following subsections, starting with the initial sequence and ending with the convergence of the iterations. Along this way, we prove that the solution obtained satisfies all the equations and boundary conditions presented in the previous section.

\subsection{Setup of the iteration scheme}\label{6.1}

In this work, the shock interaction problem is solved using an iterative method that is based on the functions $\alpha(u,v)$, $\beta(u,v)$ and $r(u,v)$. The iteration consists of the following steps:

(1) We start by initializing the sequence:

\begin{equation}\label{eq6.1}
r_0(u,v)=r_0+\frac{1}{\Gamma_0}u+v,\ \ \ \alpha_0(u,v)=\beta_0+\alpha^{'}_0 u,\ \ \ \beta_0(u,v)=\beta_0+\beta_0^{'}v,
\end{equation}
\begin{equation}\label{eq6.2}
\overset{1}{\tilde{\alpha}}_{+0}(u)=\alpha_0(u,u)=\beta_0+\alpha_0^{'}u,\ \ \ \overset{2}{\tilde{\beta}}_{+0}(v)=\beta_0(av,v)=\beta_0+\beta_0^{'}v,
\end{equation}
\begin{equation}\label{eq6.3}
\tilde{\alpha}_0(u,v)=\alpha_0(u,v)-\overset{1}{\tilde{\alpha}}_{+0}(u)=0,\ \ \ \tilde{\beta}_0(u,v)=\beta_0(u,v)-\overset{2}{\tilde{\beta}}_{+0}(v)=0.
\end{equation}

(2) Obviously the initialization sequence satisfies 

\begin{equation*}
r_0(u,v)\in C^2(T_\varepsilon),\ \ \ \tilde{\alpha}_0(u,v)\in C^2(T_\varepsilon),\ \ \ \tilde{\beta}_0(u,v)\in C^2(T_\varepsilon),\ \ \ \overset{1}{\tilde{\alpha}}_{+0}(u)\in C^2[0,a\varepsilon],\ \ \ \overset{2}{\tilde{\beta}}_{+0}(v)\in C^2[0,\varepsilon].
\end{equation*}

(3) We set

\begin{equation}\label{eq6.4}
t_m(u,v)=\int_0^u(\phi_m+\psi_m)(u^{'},u^{'})du^{'}+\int_u^v\psi_m(u,v^{'})dv^{'},
\end{equation}
where

\begin{equation}\label{eq6.5}
\phi_m(u,v)=\frac{1}{c_{in,m}(u,v)}\frac{\partial r_m}{\partial u}(u,v),\ \ \ \psi_m(u,v)=\frac{1}{c_{out,m}(u,v)}\frac{\partial r_m}{\partial v}(u,v).
\end{equation}
Here we use the notation

\begin{equation}\label{eq6.6}
c_{in,m}(u,v)=c_{in}(\alpha_m(u,v),\beta_m(u,v)),\ \ \ c_{out,m}(u,v)=c_{out}(\alpha_m(u,v),\beta_m(u,v)).
\end{equation}

(4) Define the Riemann invariants in the state ahead along the shock curves as

\begin{equation}\label{eq6.7}
\overset{1}{\alpha}_{-m}(u)={\overset{1}{\alpha}}^*(\overset{1}{t}_{+m}(u),\overset{1}{r}_{+m}(u)),\ \ \ \overset{1}{\beta}_{-m}(u)={\overset{1}{\beta}}^*(\overset{1}{t}_{+m}(u),\overset{1}{r}_{+m}(u)),
\end{equation}
\begin{equation}\label{eq6.8}
\overset{2}{\alpha}_{-m}(v)={\overset{2}{\alpha}}^*(\overset{2}{t}_{+m}(v),\overset{2}{r}_{+m}(v)),\ \ \ \overset{2}{\beta}_{-m}(v)={\overset{2}{\beta}}^*(\overset{2}{t}_{+m}(v),\overset{2}{r}_{+m}(v)),
\end{equation}
where

\begin{equation}\label{eq6.9}
\overset{1}{t}_{+m}(u)=t_m(u,u),\ \ \ \overset{1}{r}_{+m}(u)=r_m(u,u),\ \ \ \overset{2}{t}_{+m}(v)=t_m(av,v),\ \ \ \overset{2}{r}_{+m}(v)=r_m(av,v).
\end{equation}
Here ${\overset{i}{\alpha}}^*$ and ${\overset{i}{\beta}}^*$ ($i=1,2$) are given by the solution in the state ahead.

(5) Define the Riemann invariants in the state behind along the shock curves as

\begin{equation}\label{eq6.10}
\overset{1}{\alpha}_{+m}(u)=\alpha_m(u,u),\ \ \ \overset{1}{\beta}_{+m}(u)=\beta_m(u,u),
\end{equation}
\begin{equation}\label{eq6.11}
\overset{2}{\alpha}_{+m}(v)=\alpha_m(av,v),\ \ \ \overset{2}{\beta}_{+m}(v)=\beta_m(av,v).
\end{equation}

(6) We define the speeds of the two shock waves as

\begin{equation}\label{eq6.12}
\overset{i}{V}_m=\frac{[\overset{i}{\rho}_m\overset{i}{w}_m]}{[\overset{i}{\rho}_m]},\ \ \ i=1,2,
\end{equation}
where

\begin{equation}\label{eq6.13}
\overset{1}{\rho}_{\pm m}(u)=\rho(\overset{1}{\alpha}_{\pm m}(u),\overset{1}{\beta}_{\pm m}(u)),\ \ \ \overset{1}{w}_{\pm m}(u)=w(\overset{1}{\alpha}_{\pm m}(u),\overset{1}{\beta}_{\pm m}(u)),
\end{equation}
\begin{equation}\label{eq6.14}
\overset{2}{\rho}_{\pm m}(v)=\rho(\overset{2}{\alpha}_{\pm m}(v),\overset{2}{\beta}_{\pm m}(v)),\ \ \ \overset{2}{w}_{\pm m}(v)=w(\overset{2}{\alpha}_{\pm m}(v),\overset{2}{\beta}_{\pm m}(v)).
\end{equation}
We also define $\rho_m(u,v)=\rho(\alpha_m(u,v),\beta_m(u,v))$, $w_m(u,v)=w(\alpha_m(u,v),\beta_m(u,v))$. And here we use the notation $[f_m]=f_{+m}-f_{-m}$. In addition, we define

\begin{equation}\label{eq6.15}
\overset{i}{c_{in}}_{+,m}(u)=c_{in}(\overset{i}{\alpha}_{+m}(u),\overset{i}{\beta}_{+m}(u)),\ \ \ \overset{i}{c_{out}}_{+,m}(u)=c_{out}(\overset{i}{\alpha}_{+m}(u),\overset{i}{\beta}_{+m}(u)).
\end{equation}
In this manner, we have already provided all the quantities of the $m$th iteration step.

(7) Next, we define $r$ in the $(m+1)$th iteration step:

\begin{equation}\label{eq6.16}
\begin{aligned}
r_{m+1}(u,v)=&r_0+\frac{1}{\Gamma_0}u+v+\int_0^u \overset{1}{\Phi}_m(u^{'})du^{'}+\int_0^v\overset{2}{\Phi}_m(v^{'})dv^{'}+\int_0^u\int_u^{v^{'}} M_m(u^{'},v^{'})du^{'}dv^{'}\\
&+\int_0^v\int_{av^{'}}^u M_m(u^{'},v^{'})du^{'}dv^{'},
\end{aligned}
\end{equation}
where

\begin{equation}\label{eq6.17}
\overset{1}{\Phi}_m(u)=\int_0^u \overset{1}{\Lambda}_m(u^{'})du^{'}+\frac{1}{\overset{1}{\Gamma}_m(u)}\int_{au}^u M_m(u^{'},u)du^{'},
\end{equation}
\begin{equation}\label{eq6.18}
\overset{2}{\Phi}_m(v)=\int_0^v \overset{2}{\Lambda}_m(v^{'})dv^{'}+\overset{2}{\Gamma}_m(v)\int_{av}^v M_m(av,v^{'})dv^{'},
\end{equation}
where

\begin{equation}\label{eq6.19}
\overset{1}{\Lambda}_m(u)=\frac{d\overset{1}{\gamma}_m}{du}(u)\frac{\partial r_m}{\partial u}(au,u)+\overset{1}{\gamma}_m(u)a\frac{\partial^2 r_m}{\partial u^2}(au,u)+\overset{1}{\gamma}_m(u)M_m(au,u),
\end{equation}
\begin{equation}\label{eq6.20}
\overset{2}{\Lambda}_m(v)=\frac{d\overset{2}{\gamma}_m}{dv}(v)\frac{\partial r_m}{\partial v}(av,av)+\overset{2}{\gamma}_m(v)a\frac{\partial^2 r_m}{\partial v^2}(av,av)+\overset{2}{\gamma}_m(v)aM_m(av,v),
\end{equation}
\begin{equation}\label{eq6.21}
M_m(u,v)=\mu_m(u,v)\frac{\partial r_m}{\partial u}(u,v)+\nu_m(u,v)\frac{\partial r_m}{\partial v}(u,v),
\end{equation}
where we define

\begin{equation}\label{eq6.22}
\overset{1}{\gamma}_m(u):=\frac{\overset{2}{\Gamma}_m(u)}{\overset{1}{\Gamma}_m(u)},\ \ \ \overset{2}{\gamma}_m(v):=\frac{\overset{2}{\Gamma}_m(v)}{\overset{1}{\Gamma}_m(av)},
\end{equation}
\begin{equation}\label{eq6.23}
\mu_m:=\frac{1}{c_{out,m}-c_{in,m}}\frac{c_{out,m}}{c_{in,m}}\frac{\partial c_{in,m}}{\partial v},\ \ \ \nu_m:=-\frac{1}{c_{out,m}-c_{in,m}}\frac{c_{in,m}}{c_{out,m}}\frac{\partial c_{out,m}}{\partial u}.
\end{equation}
The definition of $\overset{i}{\Gamma}_m$ here is based on equation \eqref{eq5.9}. We add the subscript $m$ to all the quantities on the right side of the two equations in \eqref{eq5.9}, and take it as the definition of $\overset{i}{\Gamma}_m$. And we also use the notation 

\begin{equation}\label{eq6.24}
c_{out,m}(u,v)=c_{out}(\alpha_m(u,v),\beta_m(u,v)),\ \ \ c_{in,m}(u,v)=c_{in}(\alpha_m(u,v),\beta_m(u,v)).
\end{equation}

(8) Now let's give the definitions of $\alpha_{m+1}(u,v)$, $\beta_{m+1}(u,v)$ , $\overset{1}{\tilde{\alpha}}_{+,m+1}(u)$ and $\overset{2}{\tilde{\beta}}_{+,m+1}(v)$ :

\begin{equation}\label{eq6.25}
\overset{1}{\tilde{\alpha}}_{+,m+1}(u)=\overset{1}{H}(\overset{1}{\beta}_{+m}(u),\overset{1}{\alpha}_{-m}(u),\overset{1}{\beta}_{-m}(u)),
\end{equation}
\begin{equation}\label{eq6.26}
\overset{2}{\tilde{\beta}}_{+,m+1}(v)=\overset{2}{H}(\overset{2}{\alpha}_{+m}(v),\overset{2}{\alpha}_{-m}(v),\overset{2}{\beta}_{-m}(v)),
\end{equation}
\begin{equation}\label{eq6.27}
\alpha_{m+1}(u,v)=\overset{1}{\tilde{\alpha}}_{+,m+1}(u)+\int_u^v\left(A_m\frac{\partial t_m}{\partial v}\right)(u,v^{'})dv^{'},
\end{equation}
\begin{equation}\label{eq6.28}
\beta_{m+1}(u,v)=\overset{2}{\tilde{\beta}}_{+,m+1}(v)+\int_{av}^u\left(A_m\frac{\partial t_m}{\partial u}\right)(u^{'},v)du^{'},
\end{equation}
where

\begin{equation}\label{eq6.29}
A_m(u,v)=A(\alpha_m(u,v),\beta_m(u,v))=-\frac{2\eta_m(u,v) w_m(u,v)}{r_m(u,v)},
\end{equation}
where $\eta_m(u,v)=\eta(\rho_m(u,v))=\eta(\rho(\alpha_m(u,v),\beta_m(u,v))$, $w_m(u,v)=w(\alpha_m(u,v),\beta_m(u,v))$.

(9) Finally, we set

\begin{equation}\label{eq6.30}
\tilde{\alpha}_{m+1}(u,v)=\alpha_{m+1}(u,v)-\overset{1}{\tilde{\alpha}}_{+,m+1}(u),
\end{equation}
\begin{equation}\label{eq6.31}
\tilde{\beta}_{m+1}(u,v)=\beta_{m+1}(u,v)-\overset{2}{\tilde{\beta}}_{+,m+1}(v).
\end{equation}

\subsection{Remarks on the iteration scheme}\label{6.2}

\begin{itemize} 

\item[$\bullet$] Each 5-tuple $(\tilde{\alpha}_m,\tilde{\beta}_m,\overset{1}{\alpha}_{+m},\overset{2}{\beta}_{+m},r_m)$ corresponds uniquely to a 3-tuple $(\alpha_m,\beta_m,r_m)$. We need to show that each iteration sequence is in its corresponding function space as the iteration steps progress. Furthermore, the sequence $((\tilde{\alpha}_m,\tilde{\beta}_m,\overset{1}{\alpha}_{+m},\overset{2}{\beta}_{+m},r_m);m=0,1,2,\cdot\cdot\cdot)$ converges.

\item[$\bullet$] Based on the iteration in step (8), we know that $\alpha_{m+1}$ and $\beta_{m+1}$ satisfy the following relation:

\begin{equation}\label{eq6.32}
\frac{\partial\alpha_{m+1}}{\partial v}(u,v)=\left(A_m\frac{\partial t_m}{\partial v}\right)(u,v),\ \ \ \frac{\partial\beta_{m+1}}{\partial u}(u,v)=\left(A_m\frac{\partial t_m}{\partial u}\right)(u,v).
\end{equation}

\item[$\bullet$] In Section 6.3, we prove that

\begin{equation}\label{eq6.33}
\frac{\partial^2 r_{m+1}}{\partial u^2}(u,v)=\overset{1}{\gamma}_m(u)a\frac{\partial^2 r_m}{\partial u^2}(au,u)+\cdot\cdot\cdot,
\end{equation}
where the terms not shown later represent mixed second derivative terms, lower-order terms, or mixed third derivative terms that appear in integrals. Because $0<a<1$ and $\overset{1}{\gamma}_m(0)=1$, the coefficient of $\frac{\partial^2 r_m}{\partial u^2}(au,u)$ is strictly less than 1. This is used to prove that the iteration converges.

\item[$\bullet$] In step (4), it is necessary to ensure that the shock curves

\begin{flalign*}
&u\longmapsto(\overset{1}{t}_{+m}(u),\overset{1}{r}_{+m}(u)),\ \ \ u\in [0,a\varepsilon],\\
&v\longmapsto(\overset{2}{t}_{+m}(v),\overset{2}{r}_{+m}(v)),\ \ \ v\in [0,\varepsilon],
\end{flalign*}
in the space-time corresponding to the iteration of step $m$ are in the future development of their corresponding initial values. Otherwise, the values of the Riemann invariants ${\overset{i}{\alpha}}^*(\overset{i}{t}_{+m},\overset{i}{r}_{+m})$ and ${\overset{i}{\beta}}^*(\overset{i}{t}_{+m},\overset{i}{r}_{+m})$, where $i=1,2$, in the state ahead are meaningless.

\end{itemize}

\subsection{Inductive step }\label{6.3}

We choose closed balls in function spaces as follows:

\begin{flalign}
&\overset{1}{B}_B=\{f\in C^2[0,a\varepsilon]:f(0)=\beta_0,f^{'}(0)=\alpha^{'}_0,\Vert f\Vert_1\le B_1\},\label{eq6.34}\\
&\overset{2}{B}_B=\{f\in C^2[0,\varepsilon]:f(0)=\beta_0,f^{'}(0)=\beta^{'}_0,\Vert f\Vert_2\le B_2\},\label{eq6.35}\\
&\overset{1}{B}_N=\{f\in C^2(T_\varepsilon):f(0,0)=r_0,\frac{\partial f}{\partial u}(0,0)=\frac{1}{\Gamma_0},\frac{\partial f}{\partial v}(0,0)=1,\Vert f\Vert_0\le N_1\},\label{eq6.36}\\
&\overset{2}{B}_N=\{f\in C^2(T_\varepsilon):f(0,0)=0,\frac{\partial f}{\partial u}(0,0)=0,\frac{\partial f}{\partial v}(0,0)=0,\Vert f\Vert_0\le N_2\},\label{eq6.37}
\end{flalign}
where the norms are defined as

\begin{flalign}
&\Vert f\Vert_0:=max\{\mathop{sup}\limits_{T_\varepsilon}\left|\frac{\partial^2 f}{\partial u^2}\right|,\mathop{sup}\limits_{T_\varepsilon}\left|\frac{\partial^2 f}{\partial u\partial v}\right|,\mathop{sup}\limits_{T_\varepsilon}\left|\frac{\partial^2 f}{\partial v^2}\right|\},\label{eq6.38}\\
&\Vert f\Vert_1:=\mathop{sup}\limits_{[0,a\varepsilon]}\left|f^{''}\right|,\label{eq6.39}\\
&\Vert f\Vert_2:=\mathop{sup}\limits_{[0,\varepsilon]}\left|f^{''}\right|.\label{eq6.40}
\end{flalign}

\begin{proposition}\label{prop1}
When $\varepsilon$ is small enough, we can make the sequence

\begin{equation*}
(\overset{1}{\tilde{\alpha}}_{+m},\overset{2}{\tilde{\beta}}_{+m},r_m,\tilde{\alpha}_m,\tilde{\beta}_m);m=0,1,2,\cdot\cdot\cdot
\end{equation*}
contained in the function space $\overset{1}{B}_B\times \overset{2}{B}_B\times\overset{1}{B}_N\times\overset{2}{B}_N\times \overset{2}{B}_N$ by choosing appropriate constants $B_1$, $B_2$, $N_1$ and $N_2$.
\end{proposition}

\begin{proof} We use mathematical induction to prove. As defined in \eqref{eq6.1}-\eqref{eq6.3}, obviously, the initial sequence satisfies

\begin{equation*}
\overset{1}{\tilde{\alpha}}_{+0}(u)\in \overset{1}{B}_B,\ \ \ \overset{2}{\tilde{\beta}}_{+0}(v)\in \overset{2}{B}_B,\ \ \ r_0(u,v)\in \overset{1}{B}_N,\ \ \ \tilde{\alpha}_0(u,v)\in \overset{2}{B}_N,\ \ \ \tilde{\beta}_0(u,v)\in \overset{2}{B}_N.
\end{equation*}
Let us assume

\begin{equation*}
(\overset{1}{\tilde{\alpha}}_{+m},\overset{2}{\tilde{\beta}}_{+m},r_m,\tilde{\alpha}_m,\tilde{\beta}_m)\in\overset{1}{B}_B\times \overset{2}{B}_B\times\overset{1}{B}_N\times\overset{2}{B}_N\times \overset{2}{B}_N.
\end{equation*}
We just need to verify that

\begin{equation*}
(\overset{1}{\tilde{\alpha}}_{+,m+1},\overset{2}{\tilde{\beta}}_{+,m+1},r_{m+1},\tilde{\alpha}_{m+1},\tilde{\beta}_{m+1})\in\overset{1}{B}_B\times \overset{2}{B}_B\times\overset{1}{B}_N\times\overset{2}{B}_N\times \overset{2}{B}_N.
\end{equation*}

We first estimate $\alpha_m(u,v)$ and $\beta_m(u,v)$ as well as their derivatives. Noticing that $u\le v$ within $T_\varepsilon$, according to the induction hypothesis, we know that $\tilde{\alpha}_m(u,v)$ and $\tilde{\beta}_m(u,v)$ satisfy

\begin{flalign}
&\tilde{\alpha}_m(u,v),\tilde{\beta}_m(u,v)=\mathcal{O}_{N_2}(v^2),\label{eq6.41}\\
&\frac{\partial\tilde{\alpha}_m}{\partial u}(u,v),\frac{\partial\tilde{\alpha}_m}{\partial v}(u,v),\frac{\partial\tilde{\beta}_m}{\partial u}(u,v),\frac{\partial\tilde{\beta}_m}{\partial v}(u,v)=\mathcal{O}_{N_2}(v),\label{eq6.42}\\
&\frac{\partial^2\tilde{\alpha}_m}{\partial u^2}(u,v),\frac{\partial^2\tilde{\alpha}_m}{\partial u\partial v}(u,v),\frac{\partial^2\tilde{\alpha}_m}{\partial v^2}(u,v),\frac{\partial^2\tilde{\beta}_m}{\partial u^2}(u,v),\frac{\partial^2\tilde{\beta}_m}{\partial u\partial v}(u,v),\frac{\partial^2\tilde{\beta}_m}{\partial v^2}=\mathcal{O}_{N_2}(1).\label{eq6.43}
\end{flalign}
Similarly, $\overset{1}{\tilde{\alpha}}_{+m}(u)$ and $\overset{2}{\tilde{\beta}}_{+m}(v)$ satisfy
\begin{flalign}
&\overset{1}{\tilde{\alpha}}_{+m}(u)=\beta_0+\alpha_0^{'}u+\mathcal{O}_{B_1}(v^2),\ \ \ \overset{2}{\tilde{\beta}}_{+m}(v)=\beta_0+\beta_0^{'}v+\mathcal{O}_{B_2}(v^2),\label{eq6.44}\\
&\overset{1}{\tilde{\alpha}}_{+m}^{'}(u)=\alpha_0^{'}+\mathcal{O}_{B_1}(v),\ \ \ \overset{2}{\tilde{\beta}}_{+m}^{'}(v)=\beta_0^{'}+\mathcal{O}_{B_2}(v),\label{eq6.45}\\
&\overset{1}{\tilde{\alpha}}_{+m}^{''}(u)=\mathcal{O}_{B_1}(1),\ \ \ \overset{2}{\tilde{\beta}}_{+m}^{''}(v)=\mathcal{O}_{B_2}(1).\label{eq6.46}
\end{flalign}
Then as defined in \eqref{eq6.30},\eqref{eq6.31} and \eqref{eq6.3},

\begin{equation}\label{eq6.47}
\alpha_m(u,v)=\tilde{\alpha}_m(u,v)+\overset{1}{\tilde{\alpha}}_{+m}(u),
\end{equation}
\begin{equation}\label{eq6.48}
\beta_m(u,v)=\tilde{\beta}_m(u,v)+\overset{2}{\tilde{\beta}}_{+m}(v),
\end{equation}
we can get

\begin{flalign}
&\alpha_m(u,v)=\beta_0+\alpha_0^{'}u+\mathcal{O}_{B_1}(v^2)+\mathcal{O}_{N_2}(v^2)=\beta_0+\mathcal{O}(v),\label{eq6.49}\\
&\beta_m(u,v)=\beta_0+\beta_0^{'}v+\mathcal{O}_{B_2}(v^2)+\mathcal{O}_{N_2}(v^2)=\beta_0+\mathcal{O}(v),\label{eq6.50}
\end{flalign}
\begin{flalign}
&\frac{\partial\alpha_m}{\partial u}(u,v)=\alpha_0^{'}+\mathcal{O}_{B_1}(v)+\mathcal{O}_{N_2}(v),\label{eq6.51}\\
&\frac{\partial\alpha_m}{\partial v}(u,v)=\mathcal{O}_{N_2}(v),\label{eq6.52}\\
&\frac{\partial\beta_m}{\partial u}(u,v)=\mathcal{O}_{N_2}(v),\label{eq6.53}\\
&\frac{\partial\beta_m}{\partial v}(u,v)=\beta_0^{'}+\mathcal{O}_{B_2}(v)+\mathcal{O}_{N_2}(v),\label{eq6.54}\\
&\frac{\partial^2\alpha_m}{\partial u^2}(u,v)=\mathcal{O}_{N_2}(1)+\mathcal{O}_{B_1}(1),\label{eq6.55}\\
&\frac{\partial^2\alpha_m}{\partial u\partial v}(u,v),\frac{\partial^2\alpha_m}{\partial v^2}(u,v)=\mathcal{O}_{N_2}(1),\label{eq6.56}\\
&\frac{\partial^2\beta_m}{\partial v^2}(u,v)=\mathcal{O}_{N_2}(1)+\mathcal{O}_{B_2}(1),\label{eq6.57}\\
&\frac{\partial^2\beta_m}{\partial u^2}(u,v),\frac{\partial^2\beta_m}{\partial u\partial v}(u,v)=\mathcal{O}_{N_2}(1).\label{eq6.58}
\end{flalign}
We can use these to estimate $\overset{i}{\alpha}_{+m}$ and $\overset{i}{\beta}_{+m}$ $(i=1,2)$ as well as their derivatives. As defined in step (5) of the iteration, we can get

\begin{flalign}
&\overset{1}{\alpha}_{+m}(u)=\beta_0+\alpha_0^{'}u+\mathcal{O}_{B_1}(v^2)+\mathcal{O}_{N_2}(v^2),\label{eq6.59}\\
&\overset{1}{\alpha}_{+m}^{'}(u)=\alpha_0^{'}+\mathcal{O}_{B_1}(v)+\mathcal{O}_{N_2}(v),\label{eq6.60}\\
&\overset{1}{\alpha}_{+m}^{''}(u)=\mathcal{O}_{B_1}(1)+\mathcal{O}_{N_2}(1),\label{eq6.61}\\
&\overset{2}{\alpha}_{+m}(v)=\beta_0+a\alpha_0^{'}v+\mathcal{O}_{B_1}(v^2)+\mathcal{O}_{N_2}(v^2),\label{eq6.62}\\
&\overset{2}{\alpha}_{+m}^{'}(v)=a\alpha_0^{'}+\mathcal{O}_{B_1}(v)+\mathcal{O}_{N_2}(v),\label{eq6.63}\\
&\overset{2}{\alpha}_{+m}^{''}(v)=\mathcal{O}_{B_1}(1)+\mathcal{O}_{N_2}(1).\label{eq6.64}
\end{flalign}
\begin{flalign}
&\overset{1}{\beta}_{+m}(u)=\beta_0+\beta_0^{'}u+\mathcal{O}_{B_2}(v^2)+\mathcal{O}_{N_2}(v^2),\label{eq6.65}\\
&\overset{1}{\beta}_{+m}^{'}(u)=\beta_0^{'}+\mathcal{O}_{B_2}(v)+\mathcal{O}_{N_2}(v),\label{eq6.66}\\
&\overset{1}{\beta}_{+m}^{''}(u)=\mathcal{O}_{B_2}(1)+\mathcal{O}_{N_2}(1),\label{eq6.67}\\
&\overset{2}{\beta}_{+m}(v)=\beta_0+\beta_0^{'}v+\mathcal{O}_{B_2}(v^2)+\mathcal{O}_{N_2}(v^2),\label{eq6.68}\\
&\overset{2}{\beta}_{+m}^{'}(v)=\beta_0^{'}+\mathcal{O}_{B_2}(v)+\mathcal{O}_{N_2}(v),\label{eq6.69}\\
&\overset{2}{\beta}_{+m}^{''}(v)=\mathcal{O}_{B_2}(1)+\mathcal{O}_{N_2}(1).\label{eq6.70}
\end{flalign}
Then we can estimate $c_{in,m}(u,v)$ and $c_{out,m}(u,v)$ as well as their derivatives. By definition in \eqref{eq6.24},

\begin{equation*}
c_{in,m}(u,v)=c_{in}(\alpha_m(u,v),\beta_m(u,v)),\ \ \ c_{out,m}(u,v)=c_{out}(\alpha_m(u,v),\beta_m(u,v)).
\end{equation*}
Noticing that if $f(u)=\mathcal{O}_A(u^n)$, then as long as $\varepsilon$ is small enough, we have $f(u)=\mathcal{O}(u^{n-1})$, which means $\left|f(u)\right|\le Cu^{n-1}$, where $C$ is a constant independent of both $A$ and $u$. Using Taylor expansion, we can easily get

\begin{flalign}
&c_{in,m}(u,v)=-\eta_0+\mathcal{O}(v),\label{eq6.71}\\
&c_{out,m}(u,v)=\eta_0+\mathcal{O}(v),\label{eq6.72}\\
&\frac{\partial c_{in,m}}{\partial u}(u,v),\frac{\partial c_{in,m}}{\partial v}(u,v),\frac{\partial c_{out,m}}{\partial u}(u,v),\frac{\partial c_{out,m}}{\partial v}(u,v)=\mathcal{O}(1),\label{eq6.73}\\
&\frac{\partial^2 c_{in,m}}{\partial u^2}(u,v),\frac{\partial^2 c_{out,m}}{\partial u^2}(u,v)=\mathcal{O}_{N_2}(1)+\mathcal{O}_{B_1}(1),\label{eq6.74}\\
&\frac{\partial^2 c_{in,m}}{\partial u\partial v}(u,v),\frac{\partial^2 c_{out,m}}{\partial u\partial v}(u,v)=\mathcal{O}_{N_2}(1),\label{eq6.75}\\
&\frac{\partial^2 c_{in,m}}{\partial v^2}(u,v),\frac{\partial^2 c_{out,m}}{\partial v^2}(u,v)=\mathcal{O}_{N_2}(1)+\mathcal{O}_{B_2}(1).\label{eq6.76}
\end{flalign}
If we set $g_m(u,v)=\frac{1}{c_{in,m}(u,v)}$ and $h_m(u,v)=\frac{1}{c_{out,m}(u,v)}$, using the above conclusions, we have

\begin{flalign}
&g_m(u,v)=-\frac{1}{\eta_0}+\mathcal{O}(v),\label{eq6.77}\\
&h_m(u,v)=\frac{1}{\eta_0}+\mathcal{O}(v),\label{eq6.78}\\
&\frac{\partial g_m}{\partial u}(u,v),\frac{\partial g_m}{\partial v}(u,v),\frac{\partial h_m}{\partial u}(u,v),\frac{\partial h_m}{\partial v}(u,v)=\mathcal{O}(1),\label{eq6.79}\\
&\frac{\partial^2 g_m}{\partial u^2}(u,v),\frac{\partial^2 h_m}{\partial u^2}(u,v)=\mathcal{O}_{N_2}(1)+\mathcal{O}_{B_1}(1),\label{eq6.80}\\
&\frac{\partial^2 g_m}{\partial u\partial v}(u,v),\frac{\partial^2 h_m}{\partial u\partial v}(u,v)=\mathcal{O}_{N_2}(1),\label{eq6.81}\\
&\frac{\partial^2 g_m}{\partial v^2}(u,v),\frac{\partial^2 h_m}{\partial v^2}(u,v)=\mathcal{O}_{N_2}(1)+\mathcal{O}_{B_2}(1).\label{eq6.82}
\end{flalign}

According to the induction hypothesis for $r_m(u,v)$, we have

\begin{flalign}
&r_m(u,v)=r_0+\frac{u}{\Gamma_0}+v+\mathcal{O}_{N_1}(v^2),\label{eq6.83}\\
&\frac{\partial r_m}{\partial u}(u,v)=\frac{1}{\Gamma_0}+\mathcal{O}_{N_1}(v),\label{eq6.84}\\
&\frac{\partial r_m}{\partial v}(u,v)=1+\mathcal{O}_{N_1}(v),\label{eq6.85}\\
&\frac{\partial^2 r_m}{\partial u^2}(u,v),\frac{\partial^2 r_m}{\partial u\partial v}(u,v),\frac{\partial^2 r_m}{\partial v^2}(u,v)=\mathcal{O}_{N_1}(1).\label{eq6.86}
\end{flalign}
Next, we will obtain the properties of $t_m(u,v)$, for which we first estimate $\phi_m(u,v)$ and $\psi_m(u,v)$. As defined in \eqref{eq6.5},

\begin{flalign}
&\phi_m(u,v)=g_m(u,v)\frac{\partial r_m}{\partial u}(u,v)=-\frac{1}{\eta_0\Gamma_0}+\mathcal{O}_{N_1}(v),\label{eq6.87}\\
&\psi_m(u,v)=h_m(u,v)\frac{\partial r_m}{\partial v}(u,v)=\frac{1}{\eta_0}+\mathcal{O}_{N_1}(v),\label{eq6.88}
\end{flalign}
\begin{flalign}
&\frac{\partial\phi_m}{\partial u}(u,v)=\frac{\partial g_m}{\partial u}(u,v)\frac{\partial r_m}{\partial u}(u,v)+g_m(u,v)\frac{\partial^2 r_m}{\partial u^2}(u,v)=\mathcal{O}_{N_1}(1),\label{eq6.89}\\
&\frac{\partial\phi_m}{\partial v}(u,v)=\frac{\partial g_m}{\partial v}(u,v)\frac{\partial r_m}{\partial u}(u,v)+g_m(u,v)\frac{\partial^2 r_m}{\partial u\partial v}(u,v)=\mathcal{O}_{N_1}(1),\label{eq6.90}
\end{flalign}
\begin{flalign}
&\frac{\partial\psi_m}{\partial u}(u,v)=\frac{\partial h_m}{\partial u}(u,v)\frac{\partial r_m}{\partial v}(u,v)+h_m(u,v)\frac{\partial^2 r_m}{\partial u\partial v}(u,v)=\mathcal{O}_{N_1}(1),\label{eq6.91}\\
&\frac{\partial\psi_m}{\partial v}(u,v)=\frac{\partial h_m}{\partial v}(u,v)\frac{\partial r_m}{\partial v}(u,v)+h_m(u,v)\frac{\partial^2 r_m}{\partial v^2}(u,v)=\mathcal{O}_{N_1}(1).\label{eq6.92}
\end{flalign}

Then by the definition of $t_m(u,v)$ in \eqref{eq6.4}, we have

\begin{flalign}
&t_m(u,v)=\int_0^u\left(\phi_m+\psi_m\right)(u^{'},u^{'})du^{'}+\int_u^v\psi_m(u,v^{'})dv^{'}=\frac{1}{\eta_0}(v-\frac{u}{\Gamma_0})+\mathcal{O}_{N_1}(v^2),\label{eq6.93}\\
&\frac{\partial t_m}{\partial u}(u,v)=\phi_m(u,u)+\int_u^v\frac{\partial\psi_m}{\partial u}(u,v^{'})dv^{'}=-\frac{1}{\eta_0\Gamma_0}+\mathcal{O}_{N_1}(v),\label{eq6.94}\\
&\frac{\partial t_m}{\partial v}(u,v)=\psi_m(u,v)=\frac{1}{\eta_0}+\mathcal{O}_{N_1}(v),\label{eq6.95}\\
&\frac{\partial^2 t_m}{\partial u^2}(u,v)=\frac{\partial\phi_m}{\partial u}(u,u)+\frac{\partial\phi_m}{\partial v}(u,u)+\frac{\partial}{\partial u}\left(\int_u^v\frac{\partial\psi_m}{\partial u}(u,v^{'})dv^{'}\right)=\mathcal{O}_{N_1}(1),\label{eq6.96}\\
&\frac{\partial^2 t_m}{\partial u\partial v}(u,v)=\frac{\partial\psi_m}{\partial u}(u,v)=\mathcal{O}_{N_1}(1),\label{eq6.97}\\
&\frac{\partial^2 t_m}{\partial v^2}(u,v)=\frac{\partial\psi_m}{\partial v}(u,v)=\mathcal{O}_{N_1}(1).\label{eq6.98}
\end{flalign}
For the third term on the right hand side of the first equation in \eqref{eq6.96}, we use

\begin{equation*}
\begin{aligned}
\int_u^v\frac{\partial\psi_m}{\partial u}(u,v^{'})dv^{'}=&\int_u^v\frac{\partial}{\partial u}\left(h_m(u,v^{'})\frac{\partial r_m}{\partial v}(u,v^{'})\right)dv^{'}\\
=&\int_u^v\left(\frac{\partial h_m}{\partial u}(u,v^{'})\frac{\partial r_m}{\partial v}(u,v^{'})+h_m(u,v^{'})\frac{\partial^2 r_m}{\partial u\partial v}(u,v^{'})\right)dv^{'}\\
=&\int_u^v\frac{\partial h_m}{\partial u}(u,v^{'})\frac{\partial r_m}{\partial v}(u,v^{'})dv^{'}-\int_u^v\frac{\partial h_m}{\partial v}(u,v^{'})\frac{\partial r_m}{\partial u}(u,v^{'})dv^{'}\\
&+h_m(u,v)\frac{\partial r_m}{\partial u}(u,v)-h_m(u,u)\frac{\partial r_m}{\partial u}(u,u),
\end{aligned}
\end{equation*}
where we use integration by parts for the second term in the second line. Using
this, we can get

\begin{equation*}
\begin{aligned}
\frac{\partial}{\partial u}\left(\int_u^v\frac{\partial\psi_m}{\partial u}(u,v^{'})dv^{'}\right)=&\int_u^v\left(\frac{\partial^2 h_m}{\partial u^2}(u,v^{'})\frac{\partial r_m}{\partial v}(u,v^{'})+\frac{\partial h_m}{\partial u}(u,v^{'})\frac{\partial^2 r_m}{\partial u\partial v}(u,v^{'})\right)dv^{'}\\
&-\int_u^v\left(\frac{\partial^2 h_m}{\partial u\partial v}(u,v^{'})\frac{\partial r_m}{\partial u}(u,v^{'})+\frac{\partial h_m}{\partial v}(u,v^{'})\frac{\partial^2 r_m}{\partial u^2}(u,v^{'})\right)dv^{'}\\
&-\frac{\partial h_m}{\partial u}(u,u)\frac{\partial r_m}{\partial v}(u,u)+\frac{\partial h_m}{\partial v}(u,u)\frac{\partial r_m}{\partial u}(u,u)\\
&+\frac{\partial h_m}{\partial u}(u,v)\frac{\partial r_m}{\partial u}(u,v)+h_m(u,v)\frac{\partial^2 r_m}{\partial u^2}(u,v)\\
&-\left(\frac{\partial h_m}{\partial u}(u,u)+\frac{\partial h_m}{\partial v}(u,u)\right)\frac{\partial r_m}{\partial u}(u,u)\\
&-h_m(u,u)\left(\frac{\partial^2 r_m}{\partial u^2}(u,u)+\frac{\partial^2 r_m}{\partial u\partial v}(u,u)\right)\\
=&\mathcal{O}_{N_1}(1).
\end{aligned}
\end{equation*}

Then we can estimate $\overset{i}{V}_m$ and $\overset{i}{\Gamma}_m$ $(i=1,2)$ as well as their derivatives. According to the definition in \eqref{eq6.12},

\begin{equation*}
\overset{i}{V}_m=\frac{[\overset{i}{\rho}_m\overset{i}{w}_m]}{[\overset{i}{\rho}_m]}=\frac{\overset{i}{\rho}_{+m}\overset{i}{w}_{+m}-\overset{i}{\rho}_{-m}\overset{i}{w}_{-m}}{\overset{i}{\rho}_{+m}-\overset{i}{\rho}_{-m}}.
\end{equation*}
Using the previous estimates for $\overset{i}{\alpha}_{+m}$ and $\overset{i}{\beta}_{+m}$ $(i=1,2)$, we have

\begin{flalign}
&\overset{1}{\rho}_{+m}(u),\overset{2}{\rho}_{+m}(v)=\rho_0+\mathcal{O}(v),\label{eq6.99}\\
&\overset{1}{w}_{+m}(u),\overset{2}{w}_{+m}(v)=\mathcal{O}(v),\label{eq6.100}\\
&\overset{1}{\rho}_{+m}^{'}(u),\overset{2}{\rho}_{+m}^{'}(v),\overset{1}{w}_{+m}^{'}(u),\overset{2}{w}_{+m}^{'}(v)=\mathcal{O}(1),\label{eq6.101}
\end{flalign}
which can be used to obtain the properties of $\overset{1}{V}_m(u)$ and $\overset{2}{V}_m(v)$:

\begin{flalign}
&\overset{1}{V}_m(u)=\overset{1}{V}_0+\mathcal{O}(v),\label{eq6.102}\\
&\overset{2}{V}_m(v)=\overset{2}{V}_0+\mathcal{O}(v),\label{eq6.103}\\
&\overset{1}{V}_m^{'}(u),\overset{2}{V}_m^{'}(v)=\mathcal{O}(1).\label{eq6.104}
\end{flalign}
Here, we also used the estimates of $\overset{i}{\rho}_{-m}$, $\overset{i}{w}_{-m}$ and their derivatives:

\begin{flalign}
\overset{i}{\rho}_{-m}&=\rho(\overset{i}{\alpha}_{-m},\overset{i}{\beta}_{-m})=\overset{i}{\rho}_-+\mathcal{O}(v),\label{eq6.105}\\
\overset{i}{w}_{-m}&=w(\overset{i}{\alpha}_{-m},\overset{i}{\beta}_{-m})=\overset{i}{w}_-+\mathcal{O}(v),\label{eq6.106}
\end{flalign}
\begin{equation}\label{eq6.107}
\overset{i}{\rho}_{-m}^{'},\overset{i}{w}_{-m}^{'}=\mathcal{O}(1),
\end{equation}
where we used the fact that

\begin{flalign}
\overset{i}{\alpha}_{-m}&={\overset{i}{\alpha}}^*(\overset{i}{t}_{+m},\overset{i}{r}_{+m})=\overset{i}{\alpha}_{-0}+\mathcal{O}(v),\label{eq6.108}\\
\overset{i}{\beta}_{-m}&={\overset{i}{\beta}}^*(\overset{i}{t}_{+m},\overset{i}{r}_{+m})=\overset{i}{\beta}_{-0}+\mathcal{O}(v).\label{eq6.109}
\end{flalign}
\begin{equation}\label{eq6.110}
\overset{i}{\alpha}_{-m}^{'},\overset{i}{\beta}_{-m}^{'}=\mathcal{O}(1).
\end{equation}

Next we estimate $\overset{i}{\Gamma}_m$. By definition,

\begin{flalign}
&\overset{1}{\Gamma}_m(u)=\frac{\overset{1}{c_{out}}_{+,m}(u)}{\overset{1}{c_{in}}_{+,m}(u)}\frac{\overset{1}{V}_m(u)-\overset{1}{c_{in}}_{+,m}(u)}{\overset{1}{c_{out}}_{+,m}(u)-\overset{1}{V}_m(u)},\label{eq6.111}\\
&\overset{2}{\Gamma}_m(v)=a\frac{\overset{2}{c_{out}}_{+,m}(v)}{\overset{2}{c_{in}}_{+,m}(v)}\frac{\overset{2}{V}_m(v)-\overset{2}{c_{in}}_{+,m}(v)}{\overset{2}{c_{out}}_{+,m}(v)-\overset{2}{V}_m(v)}.\label{eq6.112}
\end{flalign}
Using the previous estimates of $c_{in,m}(u,v)$ and $c_{out,m}(u,v)$, we can get

\begin{flalign}
&\overset{1}{\Gamma}_m(u)=\frac{\overset{1}{V}_0+\eta_0}{\overset{1}{V}_0-\eta_0}+\mathcal{O}(u)=\overset{1}{\Gamma_0}+\mathcal{O}(u),\label{eq6.113}\\
&\overset{2}{\Gamma}_m(v)=a\frac{\overset{2}{V}_0+\eta_0}{\overset{2}{V}_0-\eta_0}+\mathcal{O}(v)=\overset{2}{\Gamma_0}+\mathcal{O}(v),\label{eq6.114}\\
&\overset{1}{\Gamma}_m^{'}(u),\overset{2}{\Gamma}_m^{'}(v)=\mathcal{O}(1).\label{eq6.115}
\end{flalign}
So as defined in \eqref{eq6.22}, we have

\begin{flalign}
&\overset{1}{\gamma}_m(u)=\frac{\overset{2}{\Gamma}_m(u)}{\overset{1}{\Gamma}_m(u)}=1+\mathcal{O}(u),\label{eq6.116}\\
&\overset{2}{\gamma}_m(v)=\frac{\overset{2}{\Gamma}_m(v)}{\overset{1}{\Gamma}_m(av)}=1+\mathcal{O}(v),\label{eq6.117}\\
&\overset{1}{\gamma}_m^{'}(u),\overset{2}{\gamma}_m^{'}(v)=\mathcal{O}(1),\label{eq6.118}
\end{flalign}
where we used $\overset{1}{\Gamma}_0=\overset{2}{\Gamma}_0=\Gamma_0$.

Finally, we also need to obtain the properties about $M_m(u,v)$. To estimate $M_m(u,v)$ and all its first order derivatives, we first estimate $\mu_m(u,v)$ and $\nu_m(u,v)$. Using the previous estimates of $c_{in,m}(u,v)$ and $c_{out,m}(u,v)$, we have

\begin{flalign}
&\mu_m(u,v)=\frac{1}{c_{out,m}-c_{in,m}}\frac{c_{out,m}}{c_{in,m}}\frac{\partial c_{in,m}}{\partial v}=\mathcal{O}(1),\label{eq6.119}\\
&\nu_m(u,v)=-\frac{1}{c_{out,m}-c_{in,m}}\frac{c_{in,m}}{c_{out,m}}\frac{\partial c_{out,m}}{\partial u}=\mathcal{O}(1),\label{eq6.120}\\
&\frac{\partial\mu_m}{\partial u}(u,v)=\mathcal{O}_{N_2}(1),\label{eq6.121}\\
&\frac{\partial\mu_m}{\partial v}(u,v)=\mathcal{O}_{N_2}(1)+\mathcal{O}_{B_2}(1),\label{eq6.122}\\
&\frac{\partial\nu_m}{\partial u}(u,v)=\mathcal{O}_{N_2}(1)+\mathcal{O}_{B_1}(1),\label{eq6.123}\\
&\frac{\partial\nu_m}{\partial v}(u,v)=\mathcal{O}_{N_2}(1).\label{eq6.124}
\end{flalign}

Therefore, by \eqref{eq6.21}, we have

\begin{flalign}
&M_m(u,v)=\mu_m(u,v)\frac{\partial r_m}{\partial u}(u,v)+\nu_m(u,v)\frac{\partial r_m}{\partial v}(u,v)=\mathcal{O}(1),\label{eq6.125}\\
&\frac{\partial M_m}{\partial u}(u,v)=\mathcal{O}_{B_1}(1)+\mathcal{O}_{N_1}(1)+\mathcal{O}_{N_2}(1),\label{eq6.126}\\
&\frac{\partial M_m}{\partial v}(u,v)=\mathcal{O}_{B_2}(1)+\mathcal{O}_{N_1}(1)+\mathcal{O}_{N_2}(1).\label{eq6.127}
\end{flalign}

In this way, estimates of all quantities in step $m$ of the iteration are obtained. Remember that we still need to prove that the shock curve in step $m$ is located in the future development corresponding to the initial value, which verifies the fourth item in Section \ref{6.2}.

In step $m$ of the iteration, we denote the left moving characteristic from the origin as $\overset{1}{r}_{0,m}^*(u)={\overset{1}{r}}_0^*(\overset{1}{t}_{+m}(u))$, and the left moving shock curve as $\overset{1}{r}_{+m}(u)=r_m(u,u)$. Note that when $\varepsilon$ is small enough, $\overset{1}{t}_{+m}(u)$ is also small enough. At this time, ${\overset{1}{r}}_0^*(\overset{1}{t}_{+m}(u))$ is well defined. If we define the characteristic speed of the left moving characteristic starting from the origin as ${\overset{1}{c_{in}}}_0^*(t)$, that is $\frac{d{\overset{1}{r}}_0^*}{dt}(t)={\overset{1}{c_{in}}}_0^*(t)$. Because ${\overset{1}{c_{in}}}_0^*(t)$ is a smooth function of $t$, we have ${\overset{1}{c_{in}}}_0^*(t)=({\overset{1}{c_{in}}}_0^*)_0+\mathcal{O}(t)$. Therefore,

\begin{equation}\label{eq6.128}
\begin{aligned}
\frac{d{\overset{1}{r}}_{0,m}^*}{du}(u)=\frac{d{\overset{1}{r}}_0^*}{dt}(\overset{1}{t}_{+m}(u))\frac{d\overset{1}{t}_{+m}}{du}(u)=\frac{2({\overset{1}{c_{in}}}_0^*)_0}{\eta_0+\overset{1}{V}_0}+\mathcal{O}_{N_1}(u),
\end{aligned}
\end{equation}
\begin{equation}\label{eq6.129}
\begin{aligned}
\frac{d\overset{1}{r}_{+m}}{du}(u)&=\frac{\partial r_m}{\partial u}(u,u)+\frac{\partial r_m}{\partial v}(u,u)=\frac{2\overset{1}{V}_0}{\eta_0+\overset{1}{V}_0}+\mathcal{O}_{N_1}(u).
\end{aligned}
\end{equation}
Subtracting the two equations above, we can get

\begin{equation*}
\frac{d{\overset{1}{r}}_{0,m}^*}{du}(u)-\frac{d\overset{1}{r}_{+m}}{du}(u)=\frac{2}{\eta_0+\overset{1}{V}_0}\left(({\overset{1}{c_{in}}}_0^*)_0-\overset{1}{V}_0\right)+\mathcal{O}_{N_1}(u).
\end{equation*}
Therefore,

\begin{equation}\label{eq6.130}
{\overset{1}{r}}_{0,m}^*(u)-\overset{1}{r}_{+m}(u)=\frac{2}{\eta_0+\overset{1}{V}_0}\left(({\overset{1}{c_{in}}}_0^*)_0-\overset{1}{V}_0\right)u+\mathcal{O}_{N_1}(u^2).
\end{equation}
According to the determinism condition, $({\overset{1}{c_{in}}}_0^*)_0-\overset{1}{V}_0>0$. So when $\varepsilon$ is small enough, we have ${\overset{1}{r}}_{0,m}^*(u)-\overset{1}{r}_{+m}(u)\ge 0$.

In the same way, we denote the right moving characteristic from the origin as $\overset{2}{r}_{0,m}^*(v)={\overset{2}{r}}_0^*(\overset{2}{t}_{+m}(v))$, and the right moving shock curve as $\overset{2}{r}_{+m}(v)=r_m(av,v)$. When $\varepsilon$ is small enough, ${\overset{2}{r}}_0^*(\overset{2}{t}_{+m}(v))$ is well defined. If we define the characteristic speed of the right moving characteristic starting from the origin as ${\overset{2}{c_{out}}}_0^*(t)$, that is $\frac{d{\overset{2}{r}}_0^*}{dt}(t)={\overset{2}{c_{out}}}_0^*(t)$. Because ${\overset{2}{c_{out}}}_0^*(t)$ is a smooth function of $t$, we have ${\overset{2}{c_{out}}}_0^*(t)=({\overset{2}{c_{out}}}_0^*)_0+\mathcal{O}(t)$. Therefore,

\begin{equation}\label{eq6.131}
\begin{aligned}
\frac{d{\overset{2}{r}}_{0,m}^*}{dv}(v)=\frac{d{\overset{2}{r}}_0^*}{dt}(\overset{2}{t}_{+m}(v))\frac{d\overset{2}{t}_{+m}}{dv}(v)=\frac{2({\overset{2}{c_{out}}}_0^*)_0}{\eta_0+\overset{2}{V}_0}+\mathcal{O}_{N_1}(v),
\end{aligned}
\end{equation}
\begin{equation}\label{eq6.132}
\begin{aligned}
\frac{d\overset{2}{r}_{+m}}{dv}(v)=a\frac{\partial r_m}{\partial u}(av,v)+\frac{\partial r_m}{\partial v}(av,v)=\frac{2\overset{2}{V}_0}{\eta_0+\overset{2}{V}_0}+\mathcal{O}_{N_1}(v).
\end{aligned}
\end{equation}
Subtracting the two equations above, we can get

\begin{equation*}
\frac{d{\overset{2}{r}}_{0,m}^*}{dv}(v)-\frac{d\overset{2}{r}_{+m}}{dv}(v)=\frac{2}{\eta_0+\overset{2}{V}_0}\left(({\overset{2}{c_{out}}}_0^*)_0-\overset{2}{V}_0\right)+\mathcal{O}_{N_1}(v),
\end{equation*}
\begin{equation}\label{eq6.133}
{\overset{2}{r}}_{0,m}^*(v)-\overset{2}{r}_{+m}(v)=\frac{2}{\eta_0+\overset{2}{V}_0}\left(({\overset{2}{c_{out}}}_0^*)_0-\overset{2}{V}_0\right)v+\mathcal{O}_{N_1}(v^2).
\end{equation}
According to the determinism condition, $({\overset{2}{c_{out}}}_0^*)_0-\overset{2}{V}_0<0$. So when $\varepsilon$ is small enough, we have ${\overset{2}{r}}_{0,m}^*(v)-\overset{2}{r}_{+m}(v)\le 0$. Therefore, both shock curves $u\longmapsto(\overset{1}{t}_{+m}(u),\overset{1}{r}_{+m}(u))$, $u\in [0,a\varepsilon]$ and $v\longmapsto(\overset{2}{t}_{+m}(v),\overset{2}{r}_{+m}(v))$, $v\in [0,\varepsilon]$ are in the future development of the corresponding initial value.

According to the expression for $r_{m+1}(u,v)$ in \eqref{eq6.16}, by calculating all of its second order derivatives, we can get

\begin{equation}\label{eq6.134}
\begin{aligned}
\frac{\partial^2 r_{m+1}}{\partial u^2}(u,v)=&\overset{1}{\Lambda}_m(u)-\frac{\overset{1}{\Gamma}_m^{'}(u)}{(\overset{1}{\Gamma}_m(u))^2}\int_{au}^u M_m(u^{'},u)du^{'}+\int_u^v\frac{\partial M_m}{\partial u}(u,v^{'})dv^{'}-M_m(u,u)\\
&+\frac{1}{\overset{1}{\Gamma}_m(u)}\left(\int_{au}^u\frac{\partial M_m}{\partial v}(u^{'},u)du^{'}+M_m(u,u)-aM_m(au,u)\right),
\end{aligned}
\end{equation}
\begin{equation}\label{eq6.135}
\frac{\partial^2 r_{m+1}}{\partial u\partial v}(u,v)=M_m(u,v),
\end{equation}
\begin{equation}\label{eq6.136}
\begin{aligned}
\frac{\partial^2 r_{m+1}}{\partial v^2}(u,v)=&\overset{2}{\Lambda}_m(v)+\overset{2}{\Gamma}_m^{'}(v)\int_{av}^v M_m(av,v^{'})dv^{'}+\int_{av}^u\frac{\partial M_m}{\partial v}(u^{'},v)du^{'}-aM_m(av,v)\\
&+\overset{2}{\Gamma}_m(v)\left(\int_{av}^v a\frac{\partial M_m}{\partial u}(av,v^{'})dv^{'}+M_m(av,v)-aM_m(av,av)\right),
\end{aligned}
\end{equation}
where

\begin{flalign*}
&\left|\overset{1}{\Lambda}_m(u)\right|=\left|\overset{1}{\gamma}_m^{'}(u)\frac{\partial r_m}{\partial u}(au,u)+\overset{1}{\gamma}_m(u)a\frac{\partial^2 r_m}{\partial u^2}(au,u)+\overset{1}{\gamma}_m(u)M_m(au,u)\right|\le a(1+Cu)N_1+C,\\
&\left|\overset{2}{\Lambda}_m(v)\right|=\left|\overset{2}{\gamma}_m^{'}(v)\frac{\partial r_m}{\partial v}(av,av)+\overset{2}{\gamma}_m(v)a\frac{\partial^2 r_m}{\partial v^2}(av,av)+\overset{2}{\gamma}_m(v)aM_m(av,v)\right|\le a(1+Cv)N_1+C,
\end{flalign*}
where $C$ is a constant that does not depend on $N_1$.

So using the previous estimates, when $\varepsilon$ is small enough, we can get

\begin{flalign*}
\left|\frac{\partial^2 r_{m+1}}{\partial u^2}(u,v)\right|\le aN_1+C,\ \ \ \left|\frac{\partial^2 r_{m+1}}{\partial u\partial v}(u,v)\right|\le C,\ \ \ \left|\frac{\partial^2 r_{m+1}}{\partial v^2}(u,v)\right|\le aN_1+C,
\end{flalign*}
$C$ here and $C$ above may differ by a constant that does not depend on other parameters. Therefore,

\begin{equation}\label{eq6.137}
\Vert r_{m+1}\Vert_0\le aN_1+C.
\end{equation}

By choosing a large enough constant $N_1$ such that $\frac{C}{1-a}\le N_1$, we can obtain

\begin{equation}\label{eq6.138}
\Vert r_{m+1}\Vert_0\le N_1.
\end{equation}

Next we estimate $\tilde{\alpha}_{m+1}(u,v)$ and $\tilde{\beta}_{m+1}(u,v)$ as well as their derivatives. As defined in \eqref{eq6.27} and \eqref{eq6.30},

\begin{equation}\label{eq6.139}
\tilde{\alpha}_{m+1}(u,v)=\int_u^v\left(A_m\frac{\partial t_m}{\partial v}(u,v^{'})\right)dv^{'}.
\end{equation}
If we take the first derivative with respect to $u$ and $v$ respectively for both sides of the equation above, by integrating by parts, we can get

\begin{equation*}
\begin{aligned}
\frac{\partial\tilde{\alpha}_{m+1}}{\partial u}(u,v)=&\int_u^v\left(\frac{\partial A_m}{\partial u}\frac{\partial t_m}{\partial v}\right)(u,v^{'})dv^{'}-\int_u^v\left(\frac{\partial A_m}{\partial v}\frac{\partial t_m}{\partial u}\right)(u,v^{'})dv^{'}\\
&+\left(A_m\frac{\partial t_m}{\partial u}\right)(u,v)-\left(A_m\frac{\partial t_m}{\partial u}\right)(u,u)-\left(A_m\frac{\partial t_m}{\partial v}\right)(u,u),
\end{aligned}
\end{equation*}
\begin{equation*}
\frac{\partial\tilde{\alpha}_{m+1}}{\partial v}(u,v)=\left(A_m\frac{\partial t_m}{\partial v}\right)(u,v).
\end{equation*}
If we take the derivative of the above two equations with respect to $u$ and $v$, respectively, we can obtain

\begin{equation}\label{eq6.140}
\begin{aligned}
\frac{\partial^2\tilde{\alpha}_{m+1}}{\partial u^2}(u,v)=&\int_u^v\left(\frac{\partial^2 A_m}{\partial u^2}\frac{\partial t_m}{\partial v}+\frac{\partial A_m}{\partial u}\frac{\partial^2 t_m}{\partial u\partial v}\right)(u,v^{'})dv^{'}-\left(\frac{\partial A_m}{\partial u}\frac{\partial t_m}{\partial v}\right)(u,u)\\
&-\int_u^v\left(\frac{\partial^2 A_m}{\partial u\partial v}\frac{\partial t_m}{\partial u}+\frac{\partial A_m}{\partial v}\frac{\partial^2 t_m}{\partial u^2}\right)(u,v^{'})dv^{'}-\left(\frac{\partial A_m}{\partial u}\frac{\partial t_m}{\partial u}\right)(u,u)\\
&+\left(\frac{\partial A_m}{\partial u}\frac{\partial t_m}{\partial u}\right)(u,v)+\left(A_m\frac{\partial^2 t_m}{\partial u^2}\right)(u,v)-A_m\left(\frac{\partial^2 t_m}{\partial u\partial v}+\frac{\partial^2 t_m}{\partial v^2}\right)(u,u)\\
&-A_m\left(\frac{\partial^2 t_m}{\partial u^2}+\frac{\partial^2 t_m}{\partial u\partial v}\right)(u,u)-\left(\frac{\partial A_m}{\partial u}+\frac{\partial A_m}{\partial v}\right)\frac{\partial t_m}{\partial v}(u,u),
\end{aligned}
\end{equation}

\begin{equation}\label{eq6.141}
\frac{\partial^2\tilde{\alpha}_{m+1}}{\partial u\partial v}(u,v)=\left(\frac{\partial A_m}{\partial u}\frac{\partial t_m}{\partial v}\right)(u,v)+\left( A_m \frac{\partial^2 t_m}{\partial u\partial v} \right) (u,v),
\end{equation}

\begin{equation}\label{eq6.142}
\frac{\partial^2\tilde{\alpha}_{m+1}}{\partial v^2}(u,v)=\left(\frac{\partial A_m}{\partial v}\frac{\partial t_m}{\partial v}\right)(u,v)+\left(A_m\frac{\partial^2 t_m}{\partial v^2}\right)(u,v).
\end{equation}

To obtain properties of $\tilde{\alpha}_{m+1}(u,v)$, we first estimate $A_m(u,v)$. By definition, we have

\begin{equation}\label{eq6.143}
\rho_m(u,v)=\rho(\alpha_m(u,v),\beta_m(u,v)),\ \ \ w_m(u,v)=w(\alpha_m(u,v),\beta_m(u,v)).
\end{equation}
So using our previous estimates for $\alpha_m(u,v)$ and $\beta_m(u,v)$, we have

\begin{flalign}
&\rho_m(u,v)=\rho_0+\mathcal{O}(v),\label{eq6.144}\\
&w_m(u,v)=\mathcal{O}(v),\label{eq6.145}
\end{flalign}
\begin{flalign}
&\frac{\partial\rho_m}{\partial u}(u,v),\frac{\partial\rho_m}{\partial v}(u,v),\frac{\partial w_m}{\partial u}(u,v),\frac{\partial w_m}{\partial v}(u,v)=\mathcal{O}(1),\label{eq6.146}\\
&\frac{\partial^2 \rho_m}{\partial u^2}(u,v),\frac{\partial^2 w_m}{\partial u^2}(u,v)=\mathcal{O}_{B_1}(1)+\mathcal{O}_{N_2}(1),\label{eq6.147}\\
&\frac{\partial^2 \rho_m}{\partial u\partial v}(u,v),\frac{\partial^2 w_m}{\partial u\partial v}(u,v)=\mathcal{O}_{N_2}(1),\label{eq6.148}\\
&\frac{\partial^2 \rho_m}{\partial v^2}(u,v),\frac{\partial^2 w_m}{\partial v^2}(u,v)=\mathcal{O}_{B_2}(1)+\mathcal{O}_{N_2}(1).\label{eq6.149}
\end{flalign}

In addition, as defined in \eqref{eq6.29}, we have $A_m(u,v)=-\frac{2\eta_m w_m}{r_m}(u,v)$, where $\eta_m(u,v)=\eta(\rho_m(u,v))$. Using \eqref{eq6.83}-\eqref{eq6.86} and \eqref{eq6.144}-\eqref{eq6.149}, we can get

\begin{flalign}
&A_m(u,v)=\mathcal{O}(v),\label{eq6.150}\\
&\frac{\partial A_m}{\partial u}(u,v),\frac{\partial A_m}{\partial v}(u,v)=\mathcal{O}(1),\label{eq6.151}\\
&\frac{\partial^2 A_m}{\partial u^2}(u,v)=\mathcal{O}_{B_1}(1)+\mathcal{O}_{N_2}(1),\label{eq6.152}\\
&\frac{\partial^2 A_m}{\partial u\partial v}(u,v)=\mathcal{O}_{N_2}(1),\label{eq6.153}\\
&\frac{\partial^2 A_m}{\partial v^2}(u,v)=\mathcal{O}_{B_2}(1)+\mathcal{O}_{N_2}(1).\label{eq6.154}
\end{flalign}
Using these estimates, we can obtain the estimates of all second order derivatives of $\tilde{\alpha}_{m+1}(u,v)$:
\begin{equation}\label{eq6.155}
\frac{\partial^2\tilde{\alpha}_{m+1}}{\partial u^2}(u,v),\frac{\partial^2\tilde{\alpha}_{m+1}}{\partial u\partial v}(u,v),\frac{\partial^2\tilde{\alpha}_{m+1}}{\partial v^2}(u,v)=\mathcal{O}(1).
\end{equation}

Similarly, as defined in \eqref{eq6.28} and \eqref{eq6.31},

\begin{equation}\label{eq6.156}
\tilde{\beta}_{m+1}(u,v)=\int_{av}^u\left(A_m\frac{\partial t_m}{\partial u}\right)(u^{'},v)du^{'}.
\end{equation}
We can also get the first partial derivatives of $\tilde{\beta}_{m+1}(u,v)$ with respect to $u$ and $v$, respectively:

\begin{equation*}
\frac{\partial\tilde{\beta}_{m+1}}{\partial u}(u,v)=A_m\frac{\partial t_m}{\partial u}(u,v),
\end{equation*}
\begin{equation*}
\begin{aligned}
\frac{\partial\tilde{\beta}_{m+1}}{\partial v}(u,v)=&\int_{av}^u\left(\frac{\partial A_m}{\partial v}\frac{\partial t_m}{\partial u}\right)(u^{'},v)du^{'}-\int_{av}^u\left(\frac{\partial A_m}{\partial u}\frac{\partial t_m}{\partial v}\right)(u^{'},v)du^{'}\\
&+\left(A_m\frac{\partial t_m}{\partial v}\right)(u,v)-\left(A_m\frac{\partial t_m}{\partial v}\right)(av,v)-a\left(A_m\frac{\partial t_m}{\partial u}\right)(av,v).
\end{aligned}
\end{equation*}
If we take the derivative of the above two equations with respect to $u$ and $v$, respectively, we can obtain

\begin{equation}\label{eq6.157}
\frac{\partial^2\tilde{\beta}_{m+1}}{\partial u^2}(u,v)=\left(\frac{\partial A_m}{\partial u}\frac{\partial t_m}{\partial u}+A_m\frac{\partial^2 t_m}{\partial u^2}\right)(u,v),
\end{equation}
\begin{equation}\label{eq6.158}
\frac{\partial^2\tilde{\beta}_{m+1}}{\partial u\partial v}(u,v)=\left(\frac{\partial A_m}{\partial v}\frac{\partial t_m}{\partial u}+A_m\frac{\partial^2 t_m}{\partial u\partial v}\right)(u,v),
\end{equation}
\begin{equation}\label{eq6.159}
\begin{aligned}
\frac{\partial^2\tilde{\beta}_{m+1}}{\partial v^2}(u,v)=&\int_{av}^u\left(\frac{\partial^2 A_m}{\partial v^2}\frac{\partial t_m}{\partial u}+\frac{\partial A_m}{\partial v}\frac{\partial^2 t_m}{\partial u\partial v}\right)(u^{'},v)du^{'}\\
&-\int_{av}^u\left(\frac{\partial^2 A_m}{\partial u\partial v}\frac{\partial t_m}{\partial v}+\frac{\partial A_m}{\partial u}\frac{\partial^2 t_m}{\partial v^2}\right)(u^{'},v)du^{'}\\
&-\left(\frac{\partial A_m}{\partial v}\frac{\partial t_m}{\partial v}+A_m\frac{\partial^2 t_m}{\partial v^2}\right)(av,v)-2a\left(\frac{\partial A_m}{\partial v}\frac{\partial t_m}{\partial u}+A_m\frac{\partial^2 t_m}{\partial u\partial v}\right)(av,v)\\
&-a^2\left(\frac{\partial A_m}{\partial u}\frac{\partial t_m}{\partial u}+A_m\frac{\partial^2 t_m}{\partial u^2}\right)(av,v)+\left(\frac{\partial A_m}{\partial v}\frac{\partial t_m}{\partial v}+A_m\frac{\partial^2 t_m}{\partial v^2}\right)(u,v).
\end{aligned}
\end{equation}
Using the previous estimates, we can obtain the estimates of all second order derivatives of $\tilde{\beta}_{m+1}(u,v)$:

\begin{equation}\label{eq6.160}
\frac{\partial^2 \tilde{\beta}_{m+1}}{\partial u^2}(u,v),\frac{\partial^2 \tilde{\beta}_{m+1}}{\partial u\partial v}(u,v),\frac{\partial^2\tilde{\beta}_{m+1}}{\partial v^2}(u,v)=\mathcal{O}(1).
\end{equation}
So, there exists a constant $C$ such that when $\varepsilon$ is small enough, we have

\begin{equation}\label{eq6.161}
\Vert\tilde{\alpha}_{m+1}\Vert_0,\Vert \tilde{\beta}_{m+1}\Vert_0 \le C.
\end{equation}

If we choose $N_2$ large enough such that $C\le N_2$, then we can get

\begin{equation}\label{eq6.162}
\Vert\tilde{\alpha}_{m+1}\Vert_0,\Vert \tilde{\beta}_{m+1}\Vert_0 \le N_2.
\end{equation}

Finally, we only need to estimate the second order derivatives of $\overset{1}{\tilde{\alpha}}_{+,m+1}(u)$ and $\overset{2}{\tilde{\beta}}_{+,m+1}(v)$. By \eqref{eq6.25} and \eqref{eq6.26}, we have

\begin{equation}\label{eq6.163}
\begin{aligned}
\overset{1}{\tilde{\alpha}}_{+,m+1}^{''}(u)=&\overset{1}{F}(\overset{1}{\beta}_{+m}(u),\overset{1}{\alpha}_{-m}(u),\overset{1}{\beta}_{-m}(u))\overset{1}{\beta}_{+m}^{''}(u)\\
&+\overset{1}{M}_1(\overset{1}{\beta}_{+m}(u),\overset{1}{\alpha}_{-m}(u),\overset{1}{\beta}_{-m}(u))\overset{1}{\alpha}_{-m}^{''}(u)\\
&+\overset{1}{M}_2(\overset{1}{\beta}_{+m}(u),\overset{1}{\alpha}_{-m}(u),\overset{1}{\beta}_{-m}(u))\overset{1}{\beta}_{-m}^{''}(u)\\
&+\overset{1}{G}(\overset{1}{\beta}_{+m}(u),\overset{1}{\alpha}_{-m}(u),\overset{1}{\beta}_{-m}(u),\overset{1}{\beta}_{+m}^{'}(u),\overset{1}{\alpha}_{-m}^{'}(u),\overset{1}{\beta}_{-m}^{'}(u)),
\end{aligned}
\end{equation}
\begin{equation}\label{eq6.164}
\begin{aligned}
\overset{2}{\tilde{\beta}}_{+,m+1}^{''}(v)=&\overset{2}{F}(\overset{2}{\alpha}_{+m}(v),\overset{2}{\alpha}_{-m}(v),\overset{2}{\beta}_{-m}(v))\overset{2}{\alpha}_{+m}^{''}(v)\\
&+\overset{2}{M}_1(\overset{2}{\alpha}_{+m}(v),\overset{2}{\alpha}_{-m}(v),\overset{2}{\beta}_{-m}(v))
\overset{2}{\alpha}_{-m}^{''}(v)\\
&+\overset{2}{M}_2(\overset{2}{\alpha}_{+m}(v),\overset{2}{\alpha}_{-m}(v),\overset{2}{\beta}_{-m}(v))\overset{2}{\beta}_{-m}^{''}(v)\\
&+\overset{2}{G}(\overset{2}{\alpha}_{+m}(v),\overset{2}{\alpha}_{-m}(v),\overset{2}{\beta}_{-m}(v),\overset{2}{\alpha}_{+m}^{'}(v),\overset{2}{\alpha}_{-m}^{'}(v),\overset{2}{\beta}_{-m}^{'}(v)),
\end{aligned}
\end{equation}
where

\begin{flalign}
&\overset{1}{\alpha}_{-m}(u)={\overset{1}{\alpha}}^*(\overset{1}{t}_{+m}(u),\overset{1}{r}_{+m}(u))=\overset{1}{\alpha}_{-0}+\mathcal{O}(v),\label{eq6.165}\\
&\overset{2}{\alpha}_{-m}(v)={\overset{2}{\alpha}}^*(\overset{2}{t}_{+m}(v),\overset{2}{r}_{+m}(v))=\overset{2}{\alpha}_{-0}+\mathcal{O}(v),\label{eq6.166}\\
&\overset{1}{\beta}_{-m}(u)={\overset{1}{\beta}}^*(\overset{1}{t}_{+m}(u),\overset{1}{r}_{+m}(u))=\overset{1}{\beta}_{-0}+\mathcal{O}(v),\label{eq6.167}\\
&\overset{2}{\beta}_{-m}(v)={\overset{2}{\beta}}^*(\overset{2}{t}_{+m}(v),\overset{2}{r}_{+m}(v))=\overset{2}{\beta}_{-0}+\mathcal{O}(v),\label{eq6.168}\\
&\overset{1}{\alpha}_{-m}^{'}(u),\overset{1}{\beta}_{-m}^{'}(u),\overset{2}{\alpha}_{-m}^{'}(v),\overset{2}{\beta}_{-m}^{'}(v)=\mathcal{O}(1),\label{eq6.169}\\
&\overset{1}{\alpha}_{-m}^{''}(u),\overset{1}{\beta}_{-m}^{''}(u),\overset{2}{\alpha}_{-m}^{''}(v),\overset{2}{\beta}_{-m}^{''}(v)=\mathcal{O}_{N_1}(1).\label{eq6.170}
\end{flalign}

In addition, we have

\begin{flalign}
&\left|\overset{1}{\beta}_{+m}^{''}(u)\right|=\left|\frac{\partial^2\beta_m}{\partial u^2}(u,u)+2\frac{\partial^2\beta_m}{\partial u\partial v}(u,u)+\frac{\partial^2\beta_m}{\partial v^2}(u,u)\right|\le 4N_2+B_2,\label{eq6.171}\\
&\left|\overset{2}{\alpha}_{+m}^{''}(v)\right|=\left| a^2\frac{\partial^2 \alpha_m}{\partial u^2}(av,v)+2a\frac{\partial^2\alpha_m}{\partial u\partial v}(av,v)+\frac{\partial^2\alpha_m}{\partial v^2}(av,v)\right|\le (a+1)^2N_2+a^2 B_1.\label{eq6.172}
\end{flalign}
And because $\overset{i}{F}$, $\overset{i}{M}_1$, $\overset{i}{M}_2$ and $\overset{i}{G}(i=1,2)$ are smooth functions of their components, we have

\begin{flalign}
&\overset{1}{F}(\overset{1}{\beta}_{+m}(u),\overset{1}{\alpha}_{-m}(u),\overset{1}{\beta}_{-m}(u))=\overset{1}{F}_0+\mathcal{O}(v),\label{eq6.173}\\
&\overset{2}{F}(\overset{2}{\alpha}_{+m}(v),\overset{2}{\alpha}_{-m}(v),\overset{2}{\beta}_{-m}(v))=\overset{2}{F}_0+\mathcal{O}(v),\label{eq6.174}\\
&\overset{i}{M}_1,\overset{i}{M}_2=\mathcal{O}(1),\label{eq6.175}\\
&\overset{i}{G}=\mathcal{O}(1).\label{eq6.176}
\end{flalign}
Then we have $\left|\overset{1}{\tilde{\alpha}}_{+,m+1}^{''}(u)\right|\le \overset{1}{F}_0B_2+CN_1+CN_2$, $\left|\overset{2}{\tilde{\beta}}_{+,m+1}^{''}(v)\right|\le a^2\overset{2}{F}_0 B_1+CN_1+CN_2$.

Because $N_1$ and $N_2$ have already been fixed, there exists a constant $C$ such that when $\varepsilon$ is sufficiently small,

\begin{flalign}
&\Vert \overset{1}{\tilde{\alpha}}_{+,m+1}\Vert_1\le \overset{1}{F}_0B_2+C,\label{eq6.177}\\
&\Vert \overset{2}{\tilde{\beta}}_{+,m+1}\Vert_2\le a^2\overset{2}{F}_0B_1+C.\label{eq6.178}
\end{flalign}
If we set $B_1=\overset{1}{F}_0B_2+C$, then

\begin{equation}\label{eq6.179}
\Vert \overset{1}{\tilde{\alpha}}_{+,m+1}\Vert_1\le B_1
\end{equation}
is already satisfied. And now we have $\Vert \overset{2}{\tilde{\beta}}_{+,m+1}\Vert_2\le a^2\overset{1}{F}_0\overset{2}{F}_0B_2+C=a^4B_2+C$.

Finally, by taking a sufficiently large $B_2$ such that $\frac{C}{1-a^4}\le B_2$, we can obtain

\begin{equation}\label{eq6.180}
\Vert \overset{2}{\tilde{\beta}}_{+,m+1}\Vert_2\le B_2.
\end{equation}

We thus complete the proof of Proposition \ref{prop1}.

\end{proof}

\subsection{Convergence}\label{6.4}

\begin{proposition}\label{prop2}
For $\varepsilon$ sufficiently small, the sequence

\begin{equation*}
(\overset{1}{\tilde{\alpha}}_{+m},\overset{2}{\tilde{\beta}}_{+m},r_m,\tilde{\alpha}_m,\tilde{\beta}_m);m=0,1,2,\cdot\cdot\cdot
\end{equation*}
converges in $\overset{1}{B}_B\times \overset{2}{B}_B\times\overset{1}{B}_N\times\overset{2}{B}_N\times \overset{2}{B}_N$.
\end{proposition}

\begin{proof}
We use the notation $\Delta_m f:=f_m-f_{m-1}$. We first consider $\tilde{\alpha}_m(u,v)$ and $\tilde{\beta}_m(u,v)$. According to the properties previously obtained for $\tilde{\alpha}_m(u,v)$ and $\tilde{\beta}_m(u,v)$ as given in \eqref{eq6.41}-\eqref{eq6.43}, we have

\begin{flalign}
&\left|\Delta_m \tilde{\alpha}\right|\le Cv^2\Vert\Delta_m\tilde{\alpha}\Vert_0,\label{eq6.181}\\
&\left|\Delta_m \tilde{\beta}\right|\le Cv^2\Vert\Delta_m\tilde{\beta}\Vert_0,\label{eq6.182}\\
&\left|\Delta_m\frac{\partial\tilde{\alpha}}{\partial u}\right|,\left|\Delta_m\frac{\partial\tilde{\alpha}}{\partial v}\right|\le Cv\Vert\Delta_m\tilde{\alpha}\Vert_0,\label{eq6.183}\\
&\left|\Delta_m\frac{\partial\tilde{\beta}}{\partial u}\right|,\left|\Delta_m\frac{\partial\tilde{\beta}}{\partial v}\right|\le Cv\Vert\Delta_m\tilde{\beta}\Vert_0,\label{eq6.184}\\
&\left|\Delta_m\frac{\partial^2\tilde{\alpha}}{\partial u^2}\right|,\left|\Delta_m\frac{\partial^2\tilde{\alpha}}{\partial u\partial v}\right|,\left|\Delta_m\frac{\partial^2\tilde{\alpha}}{\partial v^2}\right|\le \Vert\Delta_m\tilde{\alpha}\Vert_0,\label{eq6.185}\\
&\left|\Delta_m\frac{\partial^2\tilde{\beta}}{\partial u^2}\right|,\left|\Delta_m\frac{\partial^2\tilde{\beta}}{\partial u\partial v}\right|,\left|\Delta_m\frac{\partial^2\tilde{\beta}}{\partial v^2}\right|\le \Vert\Delta_m\tilde{\beta}\Vert_0.\label{eq6.186}
\end{flalign}
Using the properties previously obtained for $\overset{1}{\tilde{\alpha}}_{+m}(u)$ and $\overset{2}{\tilde{\beta}}_{+m}(v)$ in \eqref{eq6.44}-\eqref{eq6.46}, we can obtain

\begin{flalign}
&\left|\Delta_m\overset{1}{\tilde{\alpha}}_+\right|\le \frac{v^2}{2}\Vert\Delta_m\overset{1}{\tilde{\alpha}}_+\Vert_1,\ \ \ \left|\Delta_m\overset{2}{\tilde{\beta}}_+\right|\le \frac{v^2}{2}\Vert\Delta_m\overset{2}{\tilde{\beta}}_+\Vert_2,\label{eq6.187}\\
&\left|\Delta_m\overset{1}{\tilde{\alpha}}_+^{'}\right|\le v\Vert\Delta_m\overset{1}{\tilde{\alpha}}_+\Vert_1,\ \ \ \left|\Delta_m\overset{2}{\tilde{\beta}}_+^{'}\right|\le v\Vert\Delta_m\overset{2}{\tilde{\beta}}_+\Vert_2,\label{eq6.188}\\
&\left|\Delta_m\overset{1}{\tilde{\alpha}}_+^{''}\right|\le \Vert\Delta_m\overset{1}{\tilde{\alpha}}_+\Vert_1,\ \ \ \left|\Delta_m\overset{2}{\tilde{\beta}}_+^{''}\right|\le \Vert\Delta_m\overset{2}{\tilde{\beta}}_+\Vert_2.\label{eq6.189}
\end{flalign}
Therefore, we have

\begin{flalign}
&\left|\Delta_m\alpha\right|\le \left|\Delta_m\tilde{\alpha}\right|+\left|\Delta_m\overset{1}{\tilde{\alpha}}_+\right|\le Cv^2\Vert\Delta_m\tilde{\alpha}\Vert_0+\frac{v^2}{2}\Vert\Delta_m\overset{1}{\tilde{\alpha}}_+\Vert_1,\label{eq6.190}\\
&\left|\Delta_m\beta\right|\le \left|\Delta_m\tilde{\beta}\right|+\left|\Delta_m\overset{2}{\tilde{\beta}}_+\right|\le Cv^2\Vert\Delta_m\tilde{\beta}\Vert_0+\frac{v^2}{2}\Vert\Delta_m\overset{2}{\tilde{\beta}}_+\Vert_2,\label{eq6.191}\\
&\left|\Delta_m\frac{\partial\alpha}{\partial u}\right|\le \left|\Delta_m\frac{\partial\tilde{\alpha}}{\partial u}\right|+\left|\Delta_m\overset{1}{\tilde{\alpha}}_+^{'}\right|\le Cv\Vert\Delta_m\tilde{\alpha}\Vert_0+v\Vert\Delta_m\overset{1}{\tilde{\alpha}}_+\Vert_1,\label{eq6.192}\\
&\left|\Delta_m\frac{\partial\alpha}{\partial v}\right|\le \left|\Delta_m\frac{\partial\tilde{\alpha}}{\partial v}\right|\le Cv\Vert\Delta_m\tilde{\alpha}\Vert_0,\label{eq6.193}\\
&\left|\Delta_m\frac{\partial\beta}{\partial u}\right|\le \left|\Delta_m\frac{\partial\tilde{\beta}}{\partial u}\right|\le Cv\Vert\Delta_m\tilde{\beta}\Vert_0,\label{eq6.194}\\
&\left|\Delta_m\frac{\partial\beta}{\partial v}\right|\le \left|\Delta_m\frac{\partial\tilde{\beta}}{\partial v}\right|+\left|\Delta_m\overset{2}{\tilde{\beta}}_+^{'}\right|\le Cv\Vert\Delta_m\tilde{\beta}\Vert_0+v\Vert\Delta_m\overset{2}{\tilde{\beta}}_+\Vert_2,\label{eq6.195}
\end{flalign}

\begin{flalign}
&\left|\Delta_m\frac{\partial^2\alpha}{\partial u^2}\right|\le\left|\Delta_m\frac{\partial^2\tilde{\alpha}}{\partial u^2}\right|+\left|\Delta_m\overset{1}{\tilde{\alpha}}_+^{''}\right|\le \Vert\Delta_m\tilde{\alpha}\Vert_0+\Vert\Delta_m\overset{1}{\tilde{\alpha}}_+\Vert_1,\label{eq6.196}\\
&\left|\Delta_m\frac{\partial^2\alpha}{\partial u\partial v}\right|,\left|\Delta_m\frac{\partial^2\alpha}{\partial v^2}\right|\le \Vert\Delta_m\tilde{\alpha}\Vert_0,\label{eq6.197}\\
&\left|\Delta_m\frac{\partial^2\beta}{\partial u^2}\right|,\left|\Delta_m\frac{\partial^2\beta}{\partial u\partial v}\right|\le \Vert\Delta_m\tilde{\beta}\Vert_0,\label{eq6.198}\\
&\left|\Delta_m\frac{\partial^2\beta}{\partial v^2}\right|\le \left|\Delta_m\frac{\partial^2 \tilde{\beta}}{\partial v^2}\right|+\left|\Delta_m\overset{2}{\tilde{\beta}}_+^{''}\right|\le \Vert\Delta_m\tilde{\beta}\Vert_0+\Vert\Delta_m\overset{2}{\tilde{\beta}}_+\Vert_2.\label{eq6.199}
\end{flalign}

Furthermore, by the definitions of $\overset{i}{\alpha}_+$ and $\overset{i}{\beta}_+$ in \eqref{eq6.10} and \eqref{eq6.11}, we naturally have

\begin{flalign}
&\left|\Delta_m\overset{1}{\alpha}_+\right|\le Cv^2\Vert\Delta_m\tilde{\alpha}\Vert_0+\frac{v^2}{2}\Vert\Delta_m\overset{1}{\tilde{\alpha}}_+\Vert_1,\label{eq6.200}\\
&\left|\Delta_m\overset{1}{\alpha}_+^{'}\right|\le Cv\Vert\Delta_m\tilde{\alpha}\Vert_0+v\Vert\Delta_m\overset{1}{\tilde{\alpha}}_+\Vert_1,\label{eq6.201}\\
&\left|\Delta_m\overset{1}{\alpha}_+^{''}\right|\le C\Vert\Delta_m\tilde{\alpha}\Vert_0+\Vert\Delta_m\overset{1}{\tilde{\alpha}}_+\Vert_1,\label{eq6.202}
\end{flalign}
\begin{flalign}
&\left|\Delta_m\overset{1}{\beta}_+\right|\le Cv^2\Vert\Delta_m\tilde{\beta}\Vert_0+\frac{v^2}{2}\Vert\Delta_m\overset{2}{\tilde{\beta}}_+\Vert_2,\label{eq6.203}\\
&\left|\Delta_m\overset{1}{\beta}_+^{'}\right|\le Cv\Vert\Delta_m\tilde{\beta}\Vert_0+v\Vert\Delta_m\overset{2}{\tilde{\beta}}_+\Vert_2,\label{eq6.204}\\
&\left|\Delta_m\overset{1}{\beta}_+^{''}\right|\le C\Vert\Delta_m\tilde{\beta}\Vert_0+\Vert\Delta_m\overset{2}{\tilde{\beta}}_+\Vert_2,\label{eq6.205}
\end{flalign}
\begin{flalign}
&\left|\Delta_m\overset{2}{\alpha}_+\right|\le Cv^2\Vert\Delta_m\tilde{\alpha}\Vert_0+\frac{v^2}{2}\Vert\Delta_m\overset{1}{\tilde{\alpha}}_+\Vert_1,\label{eq6.206}\\
&\left|\Delta_m\overset{2}{\alpha}_+^{'}\right|\le Cv\Vert\Delta_m\tilde{\alpha}\Vert_0+av\Vert\Delta_m\overset{1}{\tilde{\alpha}}_+\Vert_1,\label{eq6.207}\\
&\left|\Delta_m\overset{2}{\alpha}_+^{''}\right|\le (a+1)^2\Vert\Delta_m\tilde{\alpha}\Vert_0+a^2\Vert\Delta_m\overset{1}{\tilde{\alpha}}_+\Vert_1,\label{eq6.208}
\end{flalign}
\begin{flalign}
&\left|\Delta_m\overset{2}{\beta}_+\right|\le Cv^2\Vert\Delta_m\tilde{\beta}\Vert_0+\frac{v^2}{2}\Vert\Delta_m\overset{2}{\tilde{\beta}}_+\Vert_2,\label{eq6.209}\\
&\left|\Delta_m\overset{2}{\beta}_+^{'}\right|\le Cv\Vert\Delta_m\tilde{\beta}\Vert_0+v\Vert\Delta_m\overset{2}{\tilde{\beta}}_+\Vert_2,\label{eq6.210}\\
&\left|\Delta_m\overset{2}{\beta}_+^{''}\right|\le (a+1)^2\Vert\Delta_m\tilde{\beta}\Vert_0+\Vert\Delta_m\overset{2}{\tilde{\beta}}_+\Vert_2.\label{eq6.211}
\end{flalign}

Next we estimate $\Delta_m c_{in}(u,v)$ and $\Delta_m c_{out}(u,v)$. We recall that $c_{in}$ and $c_{out}$ are both smooth functions of $\alpha$ and $\beta$. Consider \eqref{eq6.49} and \eqref{eq6.50}, notice that the points $(\alpha_{m-1}(u,v),\beta_{m-1}(u,v))$ and $(\alpha_m(u,v),\beta_m(u,v))$ are both located inside a ball centered at $(\beta_0,\beta_0)$ in $\mathbb{R}^2$. Therefore, the line segment connecting them is also within this ball. So we have

\begin{equation}\label{eq6.212}
\begin{aligned}
\left|\Delta_m c_{in}\right|,\left|\Delta_m c_{out}\right|\le Cv^2\left(\Vert\Delta_m\tilde{\alpha}\Vert_0+\Vert\Delta_m\overset{1}{\tilde{\alpha}}_+\Vert_1+\Vert\Delta_m\tilde{\beta}\Vert_0+\Vert\Delta_m\overset{2}{\tilde{\beta}}_+\Vert_2\right),
\end{aligned}
\end{equation}
Taking the partial derivatives of $c_{in}$ and $c_{out}$ with respect to $u$ and $v$, respectively, we can obtain

\begin{equation}\label{eq6.213}
\begin{aligned}
&\left|\Delta_m\frac{\partial c_{in}}{\partial u}\right|,\left|\Delta_m\frac{\partial c_{out}}{\partial u}\right|\le Cv\left(\Vert\Delta_m\tilde{\alpha}\Vert_0+\Vert\Delta_m\overset{1}{\tilde{\alpha}}_+\Vert_1+ \Vert\Delta_m\tilde{\beta}\Vert_0\right)+Cv^2\Vert\Delta_m\overset{2}{\tilde{\beta}}_+\Vert_2,
\end{aligned}
\end{equation}
\begin{equation}\label{eq6.214}
\begin{aligned}
&\left|\Delta_m\frac{\partial c_{in}}{\partial v}\right|,\left|\Delta_m\frac{\partial c_{out}}{\partial v}\right|\le Cv\left(\Vert\Delta_m\tilde{\alpha}\Vert_0+ \Vert\Delta_m\tilde{\beta}\Vert_0+\Vert\Delta_m\overset{2}{\tilde{\beta}}_+\Vert_2\right)+Cv^2\Vert\Delta_m\overset{1}{\tilde{\alpha}}_+\Vert_1.
\end{aligned}
\end{equation}
If we take the second order partial derivatives of $c_{in}$ and $c_{out}$ with respect to $u$ and $v$, respectively, we can obtain

\begin{equation}\label{eq6.215}
\begin{aligned}
&\left|\Delta_m\frac{\partial^2 c_{in}}{\partial u^2}\right|,\left|\Delta_m\frac{\partial^2 c_{out}}{\partial u^2}\right|\le  C\left(\Vert\Delta_m\tilde{\alpha}\Vert_0+\Vert\Delta_m\overset{1}{\tilde{\alpha}}_+\Vert_1+ \Vert\Delta_m\tilde{\beta}\Vert_0\right)+Cv^2\Vert\Delta_m\overset{2}{\tilde{\beta}}_+\Vert_2,
\end{aligned}
\end{equation}
\begin{equation}\label{eq6.216}
\begin{aligned}
&\left|\Delta_m\frac{\partial^2 c_{in}}{\partial u\partial v}\right|,\left|\Delta_m\frac{\partial^2 c_{out}}{\partial u\partial v}\right|\le C\left(\Vert\Delta_m\tilde{\alpha}\Vert_0+\Vert\Delta_m\tilde{\beta}\Vert_0\right)+Cv\left(\Vert\Delta_m\overset{1}{\tilde{\alpha}}_+\Vert_1+\Vert\Delta_m\overset{2}{\tilde{\beta}}_+\Vert_2\right),
\end{aligned}
\end{equation}
\begin{equation}\label{eq6.217}
\begin{aligned}
&\left|\Delta_m\frac{\partial^2 c_{in}}{\partial v^2}\right|,\left|\Delta_m\frac{\partial^2 c_{out}}{\partial v^2}\right|\le C\left(\Vert\Delta_m\tilde{\alpha}\Vert_0+\Vert\Delta_m\tilde{\beta}\Vert_0+\Vert\Delta_m\overset{2}{\tilde{\beta}}_+\Vert_2\right)+Cv^2\Vert\Delta_m\overset{1}{\tilde{\alpha}}_+\Vert_1.
\end{aligned}
\end{equation}
Then by the definitions of $g_m(u,v)$ and $h_m(u,v)$, that is, $g_m(u,v)=\frac{1}{c_{in,m}(u,v)}$, $h_m(u,v)=\frac{1}{c_{out,m}(u,v)}$, we have

\begin{equation}\label{eq6.218}
\begin{aligned}
\left|\Delta_m g\right|,\left|\Delta_m h\right|\le Cv^2\left(\Vert\Delta_m\tilde{\alpha}\Vert_0+\Vert\Delta_m\overset{1}{\tilde{\alpha}}_+\Vert_1+\Vert\Delta_m\tilde{\beta}\Vert_0+\Vert\Delta_m\overset{2}{\tilde{\beta}}_+\Vert_2\right),
\end{aligned}
\end{equation}
\begin{equation}\label{eq6.219}
\begin{aligned}
\left|\Delta_m\frac{\partial g}{\partial u}\right|,\left|\Delta_m\frac{\partial h}{\partial u}\right| \le Cv\left(\Vert\Delta_m\tilde{\alpha}\Vert_0+\Vert\Delta_m\overset{1}{\tilde{\alpha}}_+\Vert_1+ \Vert\Delta_m\tilde{\beta}\Vert_0\right)+Cv^2\Vert\Delta_m\overset{2}{\tilde{\beta}}_+\Vert_2,
\end{aligned}
\end{equation}
\begin{equation}\label{eq6.220}
\begin{aligned}
\left|\Delta_m\frac{\partial g}{\partial v}\right|,\left|\Delta_m\frac{\partial h}{\partial v}\right| \le Cv\left(\Vert\Delta_m\tilde{\alpha}\Vert_0+ \Vert\Delta_m\tilde{\beta}\Vert_0+\Vert\Delta_m\overset{2}{\tilde{\beta}}_+\Vert_2\right)+Cv^2\Vert\Delta_m\overset{1}{\tilde{\alpha}}_+\Vert_1,
\end{aligned}
\end{equation}
\begin{equation}\label{eq6.221}
\begin{aligned}
\left|\Delta_m\frac{\partial^2 g}{\partial u^2}\right|,\left|\Delta_m\frac{\partial^2 h}{\partial u^2}\right| \le C\left(\Vert\Delta_m\tilde{\alpha}\Vert_0+\Vert\Delta_m\overset{1}{\tilde{\alpha}}_+\Vert_1+ \Vert\Delta_m\tilde{\beta}\Vert_0\right)+Cv^2\Vert\Delta_m\overset{2}{\tilde{\beta}}_+\Vert_2,
\end{aligned}
\end{equation}
\begin{equation}\label{eq6.222}
\begin{aligned}
\left|\Delta_m\frac{\partial^2 g}{\partial u\partial v}\right|,\left|\Delta_m\frac{\partial^2 h}{\partial u\partial v}\right| \le C\left(\Vert\Delta_m\tilde{\alpha}\Vert_0+\Vert\Delta_m\tilde{\beta}\Vert_0\right)+Cv\left(\Vert\Delta_m\overset{1}{\tilde{\alpha}}_+\Vert_1+\Vert\Delta_m\overset{2}{\tilde{\beta}}_+\Vert_2\right),
\end{aligned}
\end{equation}
\begin{equation}\label{eq6.223}
\begin{aligned}
\left|\Delta_m\frac{\partial^2 g}{\partial v^2}\right|,\left|\Delta_m\frac{\partial^2 h}{\partial v^2}\right|\le C\left(\Vert\Delta_m\tilde{\alpha}\Vert_0+\Vert\Delta_m\tilde{\beta}\Vert_0+\Vert\Delta_m\overset{2}{\tilde{\beta}}_+\Vert_2\right)+Cv^2\Vert\Delta_m\overset{1}{\tilde{\alpha}}_+\Vert_1.
\end{aligned}
\end{equation}
In addition, obviously $\Delta_m r(u,v)$ satisfies

\begin{flalign}
&\left|\Delta_m r\right|\le Cv^2\Vert\Delta_m r\Vert_0,\label{eq6.224}\\
&\left|\Delta_m\frac{\partial r}{\partial u}\right|,\left|\Delta_m\frac{\partial r}{\partial v}\right|\le Cv\Vert\Delta_m r\Vert_0,\label{eq6.225}\\
&\left|\Delta_m\frac{\partial^2 r}{\partial u^2}\right|,\left|\Delta_m\frac{\partial^2 r}{\partial u\partial v}\right|,\left|\Delta_m\frac{\partial^2 r}{\partial v^2}\right|\le \Vert \Delta_m r\Vert_0.\label{eq6.226}
\end{flalign}

Next, to estimate $\Delta_m t(u,v)$, we first estimate $\Delta_m\phi(u,v)$ and $\Delta_m\psi(u,v)$. By \eqref{eq6.87}-\eqref{eq6.92}, we have

\begin{equation}\label{eq6.227}
\begin{aligned}
\left|\Delta_m\phi\right|,\left|\Delta_m\psi\right|\le Cv^2\left(\Vert\Delta_m\tilde{\alpha}\Vert_0+\Vert\Delta_m\overset{1}{\tilde{\alpha}}_+\Vert_1+\Vert\Delta_m\tilde{\beta}\Vert_0+\Vert\Delta_m\overset{2}{\tilde{\beta}}_+\Vert_2\right)+Cv\Vert\Delta_m r\Vert_0,
\end{aligned}
\end{equation}
\begin{flalign}
&\left|\Delta_m\frac{\partial\phi}{\partial u}\right|,\left|\Delta_m\frac{\partial\psi}{\partial u}\right|\le Cv\left(\Vert\Delta_m\tilde{\alpha}\Vert_0+\Vert\Delta_m\overset{1}{\tilde{\alpha}}_+\Vert_1+ \Vert\Delta_m\tilde{\beta}\Vert_0\right)+Cv^2\Vert\Delta_m\overset{2}{\tilde{\beta}}_+\Vert_2+C\Vert\Delta_m r\Vert_0,\label{eq6.228}\\
&\left|\Delta_m\frac{\partial\phi}{\partial v}\right|,\left|\Delta_m\frac{\partial\psi}{\partial v}\right|\le Cv\left(\Vert\Delta_m\tilde{\alpha}\Vert_0+ \Vert\Delta_m\tilde{\beta}\Vert_0+\Vert\Delta_m\overset{2}{\tilde{\beta}}_+\Vert_2\right)+Cv^2\Vert\Delta_m\overset{1}{\tilde{\alpha}}_+\Vert_1+C\Vert\Delta_m r\Vert_0.\label{eq6.229}
\end{flalign}
From the calculations in \eqref{eq6.93}-\eqref{eq6.98}, using the estimates above, we have

\begin{equation}\label{eq6.230}
\begin{aligned}
\left|\Delta_m t\right|\le Cv^3\left(\Vert\Delta_m\tilde{\alpha}\Vert_0+\Vert\Delta_m\overset{1}{\tilde{\alpha}}_+\Vert_1+\Vert\Delta_m\tilde{\beta}\Vert_0+\Vert\Delta_m\overset{2}{\tilde{\beta}}_+\Vert_2\right)+Cv^2\Vert\Delta_m r\Vert_0,
\end{aligned}
\end{equation}
\begin{equation}\label{eq6.231}
\begin{aligned}
\left|\Delta_m\frac{\partial t}{\partial u}\right|,\left|\Delta_m\frac{\partial t}{\partial v}\right|\le Cv^2\left(\Vert\Delta_m\tilde{\alpha}\Vert_0+\Vert\Delta_m\overset{1}{\tilde{\alpha}}_+\Vert_1+\Vert\Delta_m\tilde{\beta}\Vert_0+\Vert\Delta_m\overset{2}{\tilde{\beta}}_+\Vert_2\right)+Cv\Vert\Delta_m r\Vert_0,
\end{aligned}
\end{equation}
To estimate $\Delta_m\frac{\partial^2 t}{\partial u^2}(u,v)$, we first consider $\Delta_m\left(\frac{\partial}{\partial u}(\int_u^v\frac{\partial\psi}{\partial u}(u,v^{'})dv^{'})\right)$. Based on the previous calculation, we can get

\begin{equation*}
\begin{aligned}
&\Delta_m\left(\frac{\partial}{\partial u}(\int_u^v\frac{\partial\psi}{\partial u}(u,v^{'})dv^{'})\right)\\
\le & C\mathop{sup}\limits_{T_\varepsilon}\left(\left|\Delta_m h\right|+\left|\Delta_m\frac{\partial h}{\partial v}\right|+\left|\Delta_m\frac{\partial h}{\partial u}\right|+\left|\Delta_m\frac{\partial r}{\partial u}\right|+\left|\Delta_m\frac{\partial r}{\partial v}\right|+\left|\Delta_m\frac{\partial^2 r}{\partial u^2}\right|+\left|\Delta_m\frac{\partial^2 r}{\partial u\partial v}\right|\right)\\
&+Cv\mathop{sup}\limits_{T_\varepsilon}\left(\left|\Delta_m\frac{\partial^2 h}{\partial u^2}\right|+\left|\Delta_m\frac{\partial^2 h}{\partial u\partial v}\right|\right)\\
\le & Cv\left(\Vert\Delta_m\tilde{\alpha}\Vert_0+\Vert\Delta_m\overset{1}{\tilde{\alpha}}_+\Vert_1+\Vert\Delta_m\tilde{\beta}\Vert_0+\Vert\Delta_m\overset{2}{\tilde{\beta}}_+\Vert_2\right)+C\Vert\Delta_m r\Vert_0,
\end{aligned}
\end{equation*}
So we have

\begin{equation}\label{eq6.232}
\begin{aligned}
\left|\Delta_m\frac{\partial^2 t}{\partial u^2}\right|&\le \left|\Delta_m\frac{\partial\phi}{\partial u}\right|+\left|\Delta_m\frac{\partial\phi}{\partial v}\right|+\Delta_m\left(\frac{\partial}{\partial u}(\int_u^v\frac{\partial\psi}{\partial u}(u,v^{'})dv^{'})\right)\\
&\le Cv\left(\Vert\Delta_m\tilde{\alpha}\Vert_0+\Vert\Delta_m\overset{1}{\tilde{\alpha}}_+\Vert_1+\Vert\Delta_m\tilde{\beta}\Vert_0+\Vert\Delta_m\overset{2}{\tilde{\beta}}_+\Vert_2\right)+C\Vert\Delta_m r\Vert_0,
\end{aligned}
\end{equation}
\begin{equation}\label{eq6.233}
\begin{aligned}
\left|\Delta_m\frac{\partial^2 t}{\partial u\partial v}\right|&=\left|\Delta_m\frac{\partial\psi}{\partial u}\right|\le Cv\left(\Vert\Delta_m\tilde{\alpha}\Vert_0+\Vert\Delta_m\overset{1}{\tilde{\alpha}}_+\Vert_1+ \Vert\Delta_m\tilde{\beta}\Vert_0\right)+Cv^2\Vert\Delta_m\overset{2}{\tilde{\beta}}_+\Vert_2+C\Vert\Delta_m r\Vert_0,
\end{aligned}
\end{equation}
\begin{equation}\label{eq6.234}
\begin{aligned}
\left|\Delta_m\frac{\partial^2 t}{\partial v^2}\right|&=\left|\Delta_m\frac{\partial\psi}{\partial v}\right|\le Cv\left(\Vert\Delta_m\tilde{\alpha}\Vert_0+ \Vert\Delta_m\tilde{\beta}\Vert_0+\Vert\Delta_m\overset{2}{\tilde{\beta}}_+\Vert_2\right)+Cv^2\Vert\Delta_m\overset{1}{\tilde{\alpha}}_+\Vert_1+C\Vert\Delta_m r\Vert_0.
\end{aligned}
\end{equation}

Next we can estimate $\Delta_{m+1}r(u,v)$, for which we first estimate $\Delta_m\rho(u,v)$ and $\Delta_m w(u,v)$. Similarly to $\Delta_m c_{in}(u,v)$ and $\Delta_m c_{out}(u.v)$, we have

\begin{equation}\label{eq6.235}
\begin{aligned}
&\left|\Delta_m\rho\right|,\left|\Delta_m w\right|\le Cv^2\left(\Vert\Delta_m\tilde{\alpha}\Vert_0+\Vert\Delta_m\overset{1}{\tilde{\alpha}}_+\Vert_1+\Vert\Delta_m\tilde{\beta}\Vert_0+\Vert\Delta_m\overset{2}{\tilde{\beta}}_+\Vert_2\right).
\end{aligned}
\end{equation}

\begin{equation}\label{eq6.236}
\begin{aligned}
&\left|\Delta_m\frac{\partial \rho}{\partial u}\right|,\left|\Delta_m\frac{\partial w}{\partial u}\right|\le Cv\left(\Vert\Delta_m\tilde{\alpha}\Vert_0+\Vert\Delta_m\overset{1}{\tilde{\alpha}}_+\Vert_1+ \Vert\Delta_m\tilde{\beta}\Vert_0\right)+Cv^2\Vert\Delta_m\overset{2}{\tilde{\beta}}_+\Vert_2,
\end{aligned}
\end{equation}

\begin{equation}\label{eq6.237}
\begin{aligned}
&\left|\Delta_m\frac{\partial \rho}{\partial v}\right|,\left|\Delta_m\frac{\partial w}{\partial v}\right|\le Cv\left(\Vert\Delta_m\tilde{\alpha}\Vert_0+ \Vert\Delta_m\tilde{\beta}\Vert_0+\Vert\Delta_m\overset{2}{\tilde{\beta}}_+\Vert_2\right)+Cv^2\Vert\Delta_m\overset{1}{\tilde{\alpha}}_+\Vert_1.
\end{aligned}
\end{equation}

\begin{equation}\label{eq6.238}
\begin{aligned}
&\left|\Delta_m\frac{\partial^2 \rho}{\partial u^2}\right|,\left|\Delta_m\frac{\partial^2 w}{\partial u^2}\right|\le C\left(\Vert\Delta_m\tilde{\alpha}\Vert_0+\Vert\Delta_m\overset{1}{\tilde{\alpha}}_+\Vert_1+ \Vert\Delta_m\tilde{\beta}\Vert_0\right)+Cv^2\Vert\Delta_m\overset{2}{\tilde{\beta}}_+\Vert_2,
\end{aligned}
\end{equation}

\begin{equation}\label{eq6.239}
\begin{aligned}
&\left|\Delta_m\frac{\partial^2 \rho}{\partial u\partial v}\right|,\left|\Delta_m\frac{\partial^2 w}{\partial u\partial v}\right|\le C\left(\Vert\Delta_m\tilde{\alpha}\Vert_0+\Vert\Delta_m\tilde{\beta}\Vert_0\right)+Cv\left(\Vert\Delta_m\overset{1}{\tilde{\alpha}}_+\Vert_1+\Vert\Delta_m\overset{2}{\tilde{\beta}}_+\Vert_2\right),
\end{aligned}
\end{equation}

\begin{equation}\label{eq6.240}
\begin{aligned}
&\left|\Delta_m\frac{\partial^2 \rho}{\partial v^2}\right|,\left|\Delta_m\frac{\partial^2 w}{\partial v^2}\right|\le C\left(\Vert\Delta_m\tilde{\alpha}\Vert_0+\Vert\Delta_m\tilde{\beta}\Vert_0+\Vert\Delta_m\overset{2}{\tilde{\beta}}_+\Vert_2\right)+Cv^2\Vert\Delta_m\overset{1}{\tilde{\alpha}}_+\Vert_1.
\end{aligned}
\end{equation}
Therefore, according to the definitions of $\overset{i}{\rho}_{+m}$ and $\overset{i}{w}_{+m}$ in \eqref{eq6.13} and \eqref{eq6.14}, we have

\begin{flalign}
&\left|\Delta_m\overset{i}{\rho}_+\right|,\left|\Delta_m\overset{i}{w}_+\right|\le Cv^2\left(\Vert\Delta_m\tilde{\alpha}\Vert_0+\Vert\Delta_m\overset{1}{\tilde{\alpha}}_+\Vert_1+\Vert\Delta_m\tilde{\beta}\Vert_0+\Vert\Delta_m\overset{2}{\tilde{\beta}}_+\Vert_2\right),\label{eq6.241}\\
&\left|\Delta_m\overset{i}{\rho}_+^{'}\right|,\left|\Delta_m\overset{i}{w}_+^{'}\right|\le Cv\left(\Vert\Delta_m\tilde{\alpha}\Vert_0+\Vert\Delta_m\overset{1}{\tilde{\alpha}}_+\Vert_1+\Vert\Delta_m\tilde{\beta}\Vert_0+\Vert\Delta_m\overset{2}{\tilde{\beta}}_+\Vert_2\right),\label{eq6.242}\\
&\left|\Delta_m\overset{i}{\rho}_+^{''}\right|,\left|\Delta_m\overset{i}{w}_+^{''}\right|\le C\left(\Vert\Delta_m\tilde{\alpha}\Vert_0+\Vert\Delta_m\overset{1}{\tilde{\alpha}}_+\Vert_1+\Vert\Delta_m\tilde{\beta}\Vert_0+\Vert\Delta_m\overset{2}{\tilde{\beta}}_+\Vert_2\right).\label{eq6.243}
\end{flalign}

In addition, for the Riemann invariants in the state ahead, using the estimates for $\Delta_m\overset{i}{t}_+$, $\Delta_m\overset{i}{r}_+$ and their derivatives which can be obtained by \eqref{eq6.230}-\eqref{eq6.234} and \eqref{eq6.224}-\eqref{eq6.226}, we can conclude

\begin{equation}\label{eq6.244}
\begin{aligned}
\left|\Delta_m \overset{i}{\alpha}_-\right|,\left|\Delta_m\overset{i}{\beta}_-\right|\le Cv^3\left(\Vert\Delta_m\tilde{\alpha}\Vert_0+\Vert\Delta_m\overset{1}{\tilde{\alpha}}_+\Vert_1+\Vert\Delta_m\tilde{\beta}\Vert_0+\Vert\Delta_m\overset{2}{\tilde{\beta}}_+\Vert_2\right)+Cv^2\Vert\Delta_m r\Vert_0,
\end{aligned}
\end{equation}
\begin{equation}\label{eq6.245}
\begin{aligned}
\left|\Delta_m\overset{i}{\alpha}_-^{'}\right|,\left|\Delta_m\overset{i}{\beta}_-^{'}\right|\le Cv^2\left(\Vert\Delta_m\tilde{\alpha}\Vert_0+\Vert\Delta_m\overset{1}{\tilde{\alpha}}_+\Vert_1+\Vert\Delta_m\tilde{\beta}\Vert_0+\Vert\Delta_m\overset{2}{\tilde{\beta}}_+\Vert_2\right)+Cv\Vert\Delta_m r\Vert_0,
\end{aligned}
\end{equation}
\begin{equation}\label{eq6.246}
\begin{aligned}
\left|\Delta_m \overset{i}{\alpha}_-^{''}\right|,\left|\Delta_m\overset{i}{\beta}_-^{''}\right|\le Cv\left(\Vert\Delta_m\tilde{\alpha}\Vert_0+\Vert\Delta_m\overset{1}{\tilde{\alpha}}_+\Vert_1+\Vert\Delta_m\tilde{\beta}\Vert_0+\Vert\Delta_m\overset{2}{\tilde{\beta}}_+\Vert_2\right)+C\Vert\Delta_m r\Vert_0.
\end{aligned}
\end{equation}
Therefore, according to the definitions of $\overset{i}{\rho}_{-m}$ and $\overset{i}{w}_{-m}$ in \eqref{eq6.13} and \eqref{eq6.14}, we have

\begin{equation}\label{eq6.247}
\begin{aligned}
\left|\Delta_m\overset{i}{\rho}_-\right|,\left|\Delta_m\overset{i}{w}_-\right|\le Cv^3\left(\Vert\Delta_m\tilde{\alpha}\Vert_0+\Vert\Delta_m\overset{1}{\tilde{\alpha}}_+\Vert_1+\Vert\Delta_m\tilde{\beta}\Vert_0+\Vert\Delta_m\overset{2}{\tilde{\beta}}_+\Vert_2\right)+Cv^2\Vert\Delta_m r\Vert_0,
\end{aligned}
\end{equation}
\begin{equation}\label{eq6.248}
\begin{aligned}
\left|\Delta_m\overset{i}{\rho}_-^{'}\right|,\left|\Delta_m\overset{i}{w}_-^{'}\right|\le Cv^2\left(\Vert\Delta_m\tilde{\alpha}\Vert_0+\Vert\Delta_m\overset{1}{\tilde{\alpha}}_+\Vert_1+\Vert\Delta_m\tilde{\beta}\Vert_0+\Vert\Delta_m\overset{2}{\tilde{\beta}}_+\Vert_2\right)+Cv\Vert\Delta_m r\Vert_0.
\end{aligned}
\end{equation}
According to the expressions for $\overset{i}{V}_m$ in \eqref{eq6.12}, using the estimates obtained above, we have

\begin{equation}\label{eq6.249}
\begin{aligned}
\left|\Delta_m\overset{i}{V}\right|\le Cv^2\left(\Vert\Delta_m\tilde{\alpha}\Vert_0+\Vert\Delta_m\overset{1}{\tilde{\alpha}}_+\Vert_1+\Vert\Delta_m\tilde{\beta}\Vert_0+\Vert\Delta_m\overset{2}{\tilde{\beta}}_+\Vert_2+\Vert\Delta_m r\Vert_0\right),
\end{aligned}
\end{equation}
\begin{equation}\label{eq6.250}
\begin{aligned}
\left|\Delta_m\overset{i}{V}^{'}\right|\le Cv\left(\Vert\Delta_m\tilde{\alpha}\Vert_0+\Vert\Delta_m\overset{1}{\tilde{\alpha}}_+\Vert_1+\Vert\Delta_m\tilde{\beta}\Vert_0+\Vert\Delta_m\overset{2}{\tilde{\beta}}_+\Vert_2+\Vert\Delta_m r\Vert_0\right).
\end{aligned}
\end{equation}
According to the expressions for $\overset{i}{\Gamma}_m$ in \eqref{eq6.111} and \eqref{eq6.112}, we have

\begin{equation}\label{eq6.251}
\begin{aligned}
\left|\Delta_m\overset{i}{\Gamma}\right|\le Cv^2\left(\Vert\Delta_m\tilde{\alpha}\Vert_0+\Vert\Delta_m\overset{1}{\tilde{\alpha}}_+\Vert_1+\Vert\Delta_m\tilde{\beta}\Vert_0+\Vert\Delta_m\overset{2}{\tilde{\beta}}_+\Vert_2+\Vert\Delta_m r\Vert_0\right),
\end{aligned}
\end{equation}
\begin{equation}\label{eq6.252}
\begin{aligned}
\left|\Delta_m\overset{i}{\Gamma}^{'}\right|\le Cv\left(\Vert\Delta_m\tilde{\alpha}\Vert_0+\Vert\Delta_m\overset{1}{\tilde{\alpha}}_+\Vert_1+\Vert\Delta_m\tilde{\beta}\Vert_0+\Vert\Delta_m\overset{2}{\tilde{\beta}}_+\Vert_2+\Vert\Delta_m r\Vert_0\right).
\end{aligned}
\end{equation}
Therefore, as defined in \eqref{eq6.22}, we can obtain the estimates for $\overset{i}{\gamma}_m$ and their first order derivatives:

\begin{equation}\label{eq6.253}
\begin{aligned}
\left|\Delta_m\overset{i}{\gamma}\right|\le Cv^2\left(\Vert\Delta_m\tilde{\alpha}\Vert_0+\Vert\Delta_m\overset{1}{\tilde{\alpha}}_+\Vert_1+\Vert\Delta_m\tilde{\beta}\Vert_0+\Vert\Delta_m\overset{2}{\tilde{\beta}}_+\Vert_2+\Vert\Delta_m r\Vert_0\right),
\end{aligned}
\end{equation}
\begin{equation}\label{eq6.254}
\begin{aligned}
\left|\Delta_m\overset{i}{\gamma}^{'}\right|\le Cv\left(\Vert\Delta_m\tilde{\alpha}\Vert_0+\Vert\Delta_m\overset{1}{\tilde{\alpha}}_+\Vert_1+\Vert\Delta_m\tilde{\beta}\Vert_0+\Vert\Delta_m\overset{2}{\tilde{\beta}}_+\Vert_2+\Vert\Delta_m r\Vert_0\right).
\end{aligned}
\end{equation}

To estimate $\Delta_m M(u,v)$, we first estimate $\Delta_m\mu(u,v)$ and $\Delta_m\nu(u,v)$. According to the definitions for $\mu_m(u,v)$ and $\nu_m(u,v)$ in \eqref{eq6.23} and the properties of $\Delta_m c_{in}(u,v)$ and $\Delta_m c_{out}(u,v)$ obtained previously in \eqref{eq6.212}-\eqref{eq6.217}, we can easily get

\begin{equation}\label{eq6.255}
\begin{aligned}
\left|\Delta_m\mu\right|\le Cv\left(\Vert\Delta_m\tilde{\alpha}\Vert_0+ \Vert\Delta_m\tilde{\beta}\Vert_0+\Vert\Delta_m\overset{2}{\tilde{\beta}}_+\Vert_2\right)+Cv^2\Vert\Delta_m\overset{1}{\tilde{\alpha}}_+\Vert_1,
\end{aligned}
\end{equation}
\begin{equation}\label{eq6.256}
\begin{aligned}
\left|\Delta_m\nu\right|\le Cv\left(\Vert\Delta_m\tilde{\alpha}\Vert_0+\Vert\Delta_m\overset{1}{\tilde{\alpha}}_+\Vert_1+ \Vert\Delta_m\tilde{\beta}\Vert_0\right)+Cv^2\Vert\Delta_m\overset{2}{\tilde{\beta}}_+\Vert_2,
\end{aligned}
\end{equation}
\begin{equation}\label{eq6.257}
\begin{aligned}
\left|\Delta_m\frac{\partial\mu}{\partial u}\right|\le C\left(\Vert\Delta_m\tilde{\alpha}\Vert_0+\Vert\Delta_m\tilde{\beta}\Vert_0\right)+Cv\left(\Vert\Delta_m\overset{1}{\tilde{\alpha}}_+\Vert_1+\Vert\Delta_m\overset{2}{\tilde{\beta}}_+\Vert_2\right),
\end{aligned}
\end{equation}
\begin{equation}\label{eq6.258}
\begin{aligned}
\left|\Delta_m\frac{\partial\mu}{\partial v}\right|\le C\left(\Vert\Delta_m\tilde{\alpha}\Vert_0+\Vert\Delta_m\tilde{\beta}\Vert_0+\Vert\Delta_m\overset{2}{\tilde{\beta}}_+\Vert_2\right)+Cv^2\Vert\Delta_m\overset{1}{\tilde{\alpha}}_+\Vert_1,
\end{aligned}
\end{equation}
\begin{equation}\label{eq6.259}
\begin{aligned}
\left|\Delta_m\frac{\partial\nu}{\partial u}\right|\le C\left(\Vert\Delta_m\tilde{\alpha}\Vert_0+\Vert\Delta_m\overset{1}{\tilde{\alpha}}_+\Vert_1+ \Vert\Delta_m\tilde{\beta}\Vert_0\right)+Cv^2\Vert\Delta_m\overset{2}{\tilde{\beta}}_+\Vert_2,
\end{aligned}
\end{equation}
\begin{equation}\label{eq6.260}
\begin{aligned}
\left|\Delta_m\frac{\partial\nu}{\partial v}\right|\le C\left(\Vert\Delta_m\tilde{\alpha}\Vert_0+\Vert\Delta_m\tilde{\beta}\Vert_0\right)+Cv\left(\Vert\Delta_m\overset{1}{\tilde{\alpha}}_+\Vert_1+\Vert\Delta_m\overset{2}{\tilde{\beta}}_+\Vert_2\right).
\end{aligned}
\end{equation}
Therefore, combining the definition of $M_m(u,v)$ in \eqref{eq6.21} with the estimates for $\Delta_m r(u,v)$ and its derivatives obtained in \eqref{eq6.224}-\eqref{eq6.226} earlier, we have

\begin{equation}\label{eq6.261}
\begin{aligned}
\left|\Delta_m M\right|\le Cv\left(\Vert\Delta_m\tilde{\alpha}\Vert_0+ \Vert\Delta_m\tilde{\beta}\Vert_0+\Vert\Delta_m\overset{2}{\tilde{\beta}}_+\Vert_2+\Vert\Delta_m\overset{1}{\tilde{\alpha}}_+\Vert_1+\Vert\Delta_m r\Vert_0\right),
\end{aligned}
\end{equation}
\begin{equation}\label{eq6.262}
\begin{aligned}
\left|\Delta_m\frac{\partial M}{\partial u}\right|\le C\left(\Vert\Delta_m\tilde{\alpha}\Vert_0+ \Vert\Delta_m\tilde{\beta}\Vert_0+\Vert\Delta_m\overset{1}{\tilde{\alpha}}_+\Vert_1+\Vert\Delta_m r\Vert_0\right)+Cv\Vert\Delta_m\overset{2}{\tilde{\beta}}_+\Vert_2,
\end{aligned}
\end{equation}
\begin{equation}\label{eq6.263}
\begin{aligned}
\left|\Delta_m\frac{\partial M}{\partial v}\right|\le C\left(\Vert\Delta_m\tilde{\alpha}\Vert_0+ \Vert\Delta_m\tilde{\beta}\Vert_0+\Vert\Delta_m\overset{2}{\tilde{\beta}}_+\Vert_2+\Vert\Delta_m r\Vert_0\right)+Cv\Vert\Delta_m\overset{1}{\tilde{\alpha}}_+\Vert_1.
\end{aligned}
\end{equation}
Using these, according to the definition of $\overset{i}{\Lambda}_m$ $(i=1,2)$ in \eqref{eq6.19} and \eqref{eq6.20}, when $\varepsilon$ is small enough, we have

\begin{equation}\label{eq6.264}
\begin{aligned}
\left|\Delta_m\overset{i}{\Lambda}\right|\le Cv\left(\Vert\Delta_m\tilde{\alpha}\Vert_0+ \Vert\Delta_m\tilde{\beta}\Vert_0+\Vert\Delta_m\overset{2}{\tilde{\beta}}_+\Vert_2+\Vert\Delta_m\overset{1}{\tilde{\alpha}}_+\Vert_1\right)+k_1\Vert\Delta_m r\Vert_0,
\end{aligned}
\end{equation}
where $0<k_1<1$. So based on the expressions for the second order derivatives of $r_{m+1}(u,v)$ in \eqref{eq6.134}-\eqref{eq6.136}, when $\varepsilon$ is small enough, we can get

\begin{equation}\label{eq6.265}
\begin{aligned}
\left|\Delta_{m+1}\frac{\partial^2 r}{\partial u^2}\right|\le Cv\left(\Vert\Delta_m\tilde{\alpha}\Vert_0+ \Vert\Delta_m\tilde{\beta}\Vert_0+\Vert\Delta_m\overset{2}{\tilde{\beta}}_+\Vert_2+\Vert\Delta_m\overset{1}{\tilde{\alpha}}_+\Vert_1\right)+k_1\Vert\Delta_m r\Vert_0,
\end{aligned}
\end{equation}
\begin{equation}\label{eq6.266}
\begin{aligned}
\left|\Delta_{m+1}\frac{\partial^2 r}{\partial u\partial v}\right|\le Cv\left(\Vert\Delta_m\tilde{\alpha}\Vert_0+ \Vert\Delta_m\tilde{\beta}\Vert_0+\Vert\Delta_m\overset{2}{\tilde{\beta}}_+\Vert_2+\Vert\Delta_m\overset{1}{\tilde{\alpha}}_+\Vert_1+\Vert\Delta_m r\Vert_0\right),
\end{aligned}
\end{equation}
\begin{equation}\label{eq6.267}
\begin{aligned}
\left|\Delta_{m+1}\frac{\partial^2 r}{\partial v^2}\right|\le Cv\left(\Vert\Delta_m\tilde{\alpha}\Vert_0+ \Vert\Delta_m\tilde{\beta}\Vert_0+\Vert\Delta_m\overset{2}{\tilde{\beta}}_+\Vert_2+\Vert\Delta_m\overset{1}{\tilde{\alpha}}_+\Vert_1\right)+k_1\Vert\Delta_m r\Vert_0.
\end{aligned}
\end{equation}
So in summary, when $\varepsilon$ is small enough, we have our estimate of $\Vert\Delta_{m+1} r\Vert_0$:

\begin{equation}\label{eq6.268}
\Vert\Delta_{m+1} r\Vert_0\le C\varepsilon\left(\Vert\Delta_m\tilde{\alpha}\Vert_0+ \Vert\Delta_m\tilde{\beta}\Vert_0+\Vert\Delta_m\overset{2}{\tilde{\beta}}_+\Vert_2+\Vert\Delta_m\overset{1}{\tilde{\alpha}}_+\Vert_1\right)+k_1\Vert\Delta_m r\Vert_0.
\end{equation}

Next, to estimate $\Delta_{m+1}\tilde{\alpha}(u,v)$ and $\Delta_{m+1}\tilde{\beta}(u,v)$, we first estimate $\Delta_m A(u,v)$. Based on the definition of $A_m(u,v)$ in \eqref{eq6.29} and our previous estimates of $\Delta_m \rho(u,v)$, $\Delta_m w(u,v)$, $\Delta_m r(u,v)$ and their derivatives, we can easily obtain

\begin{equation}\label{eq6.269}
\begin{aligned}
\left|\Delta_m A\right|\le Cv^2\left(\Vert\Delta_m\tilde{\alpha}\Vert_0+\Vert\Delta_m\overset{1}{\tilde{\alpha}}_+\Vert_1+\Vert\Delta_m\tilde{\beta}\Vert_0+\Vert\Delta_m\overset{2}{\tilde{\beta}}_+\Vert_2\right)+Cv^3\Vert\Delta_m r\Vert_0,
\end{aligned}
\end{equation}
\begin{equation}\label{eq6.270}
\begin{aligned}
\left|\Delta_m \frac{\partial A}{\partial u}\right|\le Cv\left(\Vert\Delta_m\tilde{\alpha}\Vert_0+\Vert\Delta_m\overset{1}{\tilde{\alpha}}_+\Vert_1+ \Vert\Delta_m\tilde{\beta}\Vert_0\right)+Cv^2\left(\Vert\Delta_m\overset{2}{\tilde{\beta}}_+\Vert_2+\Vert\Delta_m r\Vert_0\right),
\end{aligned}
\end{equation}
\begin{equation}\label{eq6.271}
\begin{aligned}
\left|\Delta_m\frac{\partial A}{\partial v}\right|\le Cv\left(\Vert\Delta_m\tilde{\alpha}\Vert_0+ \Vert\Delta_m\tilde{\beta}\Vert_0+\Vert\Delta_m\overset{2}{\tilde{\beta}}_+\Vert_2\right)+Cv^2\left(\Vert\Delta_m\overset{1}{\tilde{\alpha}}_+\Vert_1+\Vert\Delta_m r\Vert_0\right),
\end{aligned}
\end{equation}
\begin{equation}\label{eq6.272}
\begin{aligned}
\left|\Delta_m\frac{\partial^2 A}{\partial u^2}\right|\le C\left(\Vert\Delta_m\tilde{\alpha}\Vert_0+\Vert\Delta_m\overset{1}{\tilde{\alpha}}_+\Vert_1+ \Vert\Delta_m\tilde{\beta}\Vert_0\right)+Cv^2\Vert\Delta_m\overset{2}{\tilde{\beta}}_+\Vert_2+Cv\Vert\Delta_m r\Vert_0,
\end{aligned}
\end{equation}
\begin{equation}\label{eq6.273}
\begin{aligned}
\left|\Delta_m\frac{\partial^2 A}{\partial u\partial v}\right|\le C\left(\Vert\Delta_m\tilde{\alpha}\Vert_0+\Vert\Delta_m\tilde{\beta}\Vert_0\right)+Cv\left(\Vert\Delta_m\overset{1}{\tilde{\alpha}}_+\Vert_1+\Vert\Delta_m\overset{2}{\tilde{\beta}}_+\Vert_2+\Vert\Delta_m r\Vert_0\right),
\end{aligned}
\end{equation}
\begin{equation}\label{eq6.274}
\begin{aligned}
\left|\Delta_m\frac{\partial^2 A}{\partial v^2}\right|\le C\left(\Vert\Delta_m\tilde{\alpha}\Vert_0+\Vert\Delta_m\tilde{\beta}\Vert_0+\Vert\Delta_m\overset{2}{\tilde{\beta}}_+\Vert_2\right)+Cv^2\Vert\Delta_m\overset{1}{\tilde{\alpha}}_+\Vert_1+Cv\Vert\Delta_m r\Vert_0.
\end{aligned}
\end{equation}
Therefore, according to the previously calculated results of all second order derivatives of $\tilde{\alpha}_{m+1}(u,v)$ in \eqref{eq6.140} - \eqref{eq6.142}, we have

\begin{equation}\label{eq6.275}
\begin{aligned}
\left|\Delta_{m+1}\frac{\partial^2\tilde{\alpha}}{\partial u^2}\right|\le Cv\left(\Vert\Delta_m\tilde{\alpha}\Vert_0+\Vert\Delta_m\tilde{\beta}\Vert_0+\Vert\Delta_m\overset{2}{\tilde{\beta}}_+\Vert_2+\Vert\Delta_m\overset{1}{\tilde{\alpha}}_+\Vert_1+\Vert\Delta_m r\Vert_0\right),
\end{aligned}
\end{equation}
\begin{equation}\label{eq6.276}
\begin{aligned}
\left|\Delta_{m+1}\frac{\partial^2\tilde{\alpha}}{\partial u\partial v}\right|\le Cv\left(\Vert\Delta_m\tilde{\alpha}\Vert_0+\Vert\Delta_m\tilde{\beta}\Vert_0+\Vert\Delta_m\overset{1}{\tilde{\alpha}}_+\Vert_1+\Vert\Delta_m r\Vert_0\right)+Cv^2\Vert\Delta_m\overset{2}{\tilde{\beta}}_+\Vert_2,
\end{aligned}
\end{equation}
\begin{equation}\label{eq6.277}
\begin{aligned}
\left|\Delta_{m+1}\frac{\partial^2\tilde{\alpha}}{\partial v^2}\right|\le Cv\left(\Vert\Delta_m\tilde{\alpha}\Vert_0+\Vert\Delta_m\tilde{\beta}\Vert_0+\Vert\Delta_m\overset{2}{\tilde{\beta}}_+\Vert_2+\Vert\Delta_m r\Vert_0\right)+Cv^2\Vert\Delta_m\overset{1}{\tilde{\alpha}}_+\Vert_1.
\end{aligned}
\end{equation}
So in summary, by the definition in \eqref{eq6.38}, when $\varepsilon$ is small enough, we have

\begin{equation}\label{eq6.278}
\Vert\Delta_{m+1}\tilde{\alpha}\Vert_0\le C\varepsilon\left(\Vert\Delta_m\tilde{\alpha}\Vert_0+\Vert\Delta_m\tilde{\beta}\Vert_0+\Vert\Delta_m\overset{2}{\tilde{\beta}}_+\Vert_2+\Vert\Delta_m\overset{1}{\tilde{\alpha}}_+\Vert_1+\Vert\Delta_m r\Vert_0\right).
\end{equation}

Similarly, based on the calculation results of all second order derivatives of $\tilde{\beta}_{m+1}(u,v)$ in \eqref{eq6.157} - \eqref{eq6.159}, we can easily obtain

\begin{equation}\label{eq6.279}
\begin{aligned}
\left|\Delta_{m+1}\frac{\partial^2\tilde{\beta}}{\partial u^2}\right|\le & Cv\left(\Vert\Delta_m\tilde{\alpha}\Vert_0+\Vert\Delta_m\tilde{\beta}\Vert_0+\Vert\Delta_m\overset{1}{\tilde{\alpha}}_+\Vert_1+\Vert\Delta_m r\Vert_0\right)+Cv^2\Vert\Delta_m\overset{2}{\tilde{\beta}}_+\Vert_2,
\end{aligned}
\end{equation}
\begin{equation}\label{eq6.280}
\begin{aligned}
\left|\Delta_{m+1}\frac{\partial^2\tilde{\beta}}{\partial u\partial v}\right|\le & Cv\left(\Vert\Delta_m\tilde{\alpha}\Vert_0+\Vert\Delta_m\tilde{\beta}\Vert_0+\Vert\Delta_m\overset{2}{\tilde{\beta}}_+\Vert_2+\Vert\Delta_m r\Vert_0\right)+Cv^2\Vert\Delta_m\overset{1}{\tilde{\alpha}}_+\Vert_1,
\end{aligned}
\end{equation}
\begin{equation}\label{eq6.281}
\begin{aligned}
\left|\Delta_{m+1}\frac{\partial^2\tilde{\beta}}{\partial v^2}\right|\le & Cv\left(\Vert\Delta_m\tilde{\alpha}\Vert_0+\Vert\Delta_m\tilde{\beta}\Vert_0+\Vert\Delta_m\overset{2}{\tilde{\beta}}_+\Vert_2+\Vert\Delta_m\overset{1}{\tilde{\alpha}}_+\Vert_1+\Vert\Delta_m r\Vert_0\right).
\end{aligned}
\end{equation}
Therefore, when $\varepsilon$ is small enough, we have

\begin{equation}\label{eq6.282}
\begin{aligned}
\Vert\Delta_{m+1}\tilde{\beta}\Vert_0\le C\varepsilon\left(\Vert\Delta_m\tilde{\alpha}\Vert_0+\Vert\Delta_m\tilde{\beta}\Vert_0+\Vert\Delta_m\overset{2}{\tilde{\beta}}_+\Vert_2+\Vert\Delta_m\overset{1}{\tilde{\alpha}}_+\Vert_1+\Vert\Delta_m r\Vert_0\right).
\end{aligned}
\end{equation}

Finally, we only need to estimate $\Delta_{m+1}\overset{1}{\tilde{\alpha}}_+^{''}(u)$ and $\Delta_{m+1}\overset{2}{\tilde{\beta}}_+^{''}(v)$. For $\Delta_{m+1}\overset{1}{\tilde{\alpha}}_+^{''}(u)$, we have

\begin{equation}\label{eq6.283}
\begin{aligned}
\left|\Delta_{m+1}\overset{1}{\tilde{\alpha}}_+^{''}\right|\le & C\left(\left|\Delta_m\overset{1}{F}\right|+\left|\Delta_m\overset{1}{M}_1\right|+\left|\Delta_m\overset{1}{\alpha}_-^{''}\right|+\left|\Delta_m\overset{1}{M}_2\right|+\left|\Delta_m\overset{1}{\beta}_-^{''}\right|+\left|\Delta_m\overset{1}{G}\right|\right)\\
&+\left|\overset{1}{F}(\overset{1}{\beta}_{+,m-1},\overset{1}{\alpha}_{-,m-1},\overset{1}{\beta}_{-,m-1})\right|\left|\Delta_m\overset{1}{\beta}_+^{''}\right|\\
\le & Cv\left(\Vert\Delta_m\tilde{\alpha}\Vert_0+\Vert\Delta_m\overset{1}{\tilde{\alpha}}_+\Vert_1\right)+C\left(\Vert\Delta_m\tilde{\beta}\Vert_0+\Vert\Delta_m r\Vert_0\right)+(\overset{1}{F}_0+Cv)\Vert\Delta_m\overset{2}{\tilde{\beta}}_+\Vert_2.
\end{aligned}
\end{equation}
by using \eqref{eq6.163} to take the difference of $\overset{1}{\tilde{\alpha}}_{+,m+1}^{''}(u)$ and $\overset{1}{\tilde{\alpha}}_{+m}^{''}(u)$.

Similarly, for $\Delta_{m+1}\overset{2}{\tilde{\beta}}_+^{''}(v)$, by utilizing \eqref{eq6.164} to take the difference between $\overset{2}{\tilde{\beta}}_{m+1}^{''}(v)$ and $\overset{2}{\tilde{\beta}}_m^{''}(v)$, we can obtain

\begin{equation}\label{eq6.284}
\begin{aligned}
\left|\Delta_{m+1}\overset{2}{\tilde{\beta}}_+^{''}\right|\le & C\left(\left|\Delta_m\overset{2}{F}\right|+\left|\Delta_m\overset{2}{M}_1\right|+\left|\Delta_m\overset{2}{\alpha}_-^{''}\right|+\left|\Delta_m\overset{2}{M}_2\right|+\left|\Delta_m\overset{2}{\beta}_-^{''}\right|+\left|\Delta_m\overset{2}{G}\right|\right)\\
&+\left|\overset{2}{F}(\overset{2}{\alpha}_{+,m-1},\overset{2}{\alpha}_{-,m-1},\overset{2}{\beta}_{-,m-1})\right|\left|\Delta_m\overset{2}{\alpha}_+^{''}\right|\\
\le & Cv\left(\Vert\Delta_m\tilde{\beta}\Vert_0+\Vert\Delta_m\overset{2}{\tilde{\beta}}_+\Vert_2\right)+C\left(\Vert\Delta_m\tilde{\alpha}\Vert_0+\Vert\Delta_m r\Vert_0\right)+a^2(\overset{2}{F}_0+Cv)\Vert\Delta_m\overset{1}{\tilde{\alpha}}_+\Vert_1.
\end{aligned}
\end{equation}

According to the previous calculation, we already have $0<\overset{1}{F}_0,\overset{2}{F}_0<1$. So when $\varepsilon$ is small enough, we have

\begin{equation}\label{eq6.285}
\Vert\Delta_{m+1}\overset{1}{\tilde{\alpha}}_+\Vert_1\le C\varepsilon\left(\Vert\Delta_m\tilde{\alpha}\Vert_0+\Vert\Delta_m\overset{1}{\tilde{\alpha}}_+\Vert_1\right)+C\left(\Vert\Delta_m\tilde{\beta}\Vert_0+\Vert\Delta_m r\Vert_0\right)+k_2\Vert\Delta_m\overset{2}{\tilde{\beta}}_+\Vert_2,
\end{equation}
\begin{equation}\label{eq6.286}
\Vert\Delta_{m+1}\overset{2}{\tilde{\beta}}_+\Vert_2\le C\varepsilon\left(\Vert\Delta_m\tilde{\beta}\Vert_0+\Vert\Delta_m\overset{2}{\tilde{\beta}}_+\Vert_2\right)+C\left(\Vert\Delta_m\tilde{\alpha}\Vert_0+\Vert\Delta_m r\Vert_0\right)+k_2\Vert\Delta_m\overset{1}{\tilde{\alpha}}_+\Vert_1,
\end{equation}
where $0<k_2<1$. We take the maximum of $k_1$ and $k_2$, which is denoted as $k$. Thus, we have completed the estimates for all quantities. Finally, we will use the previous estimates to complete the proof of convergence.

We begin by defining parameters as follows:

\begin{equation*}
\begin{cases}
&a_m:=\Vert\Delta_m\tilde{\alpha}\Vert_0+\Vert\Delta_m\tilde{\beta}\Vert_0,\\
&b_m:=\Vert\Delta_m\overset{1}{\tilde{\alpha}}_+\Vert_1+\Vert\Delta_m\overset{2}{\tilde{\beta}}_+\Vert_2,\\
&c_m:=\Vert\Delta_m r\Vert_0,\\
&a:=C\varepsilon.
\end{cases}
\end{equation*}
We can use the parameters defined above to represent the estimates of the differences obtained earlier in matrix form as

\begin{equation}\label{eq6.287}
\begin{aligned}
\begin{pmatrix}
a_{m+1}\\
b_{m+1}\\
c_{m+1}\\
\end{pmatrix}\le \begin{pmatrix}
a & a & a\\
C & k & C\\
a & a & k\\
\end{pmatrix} \begin{pmatrix}
a_m\\
b_m\\
c_m\\
\end{pmatrix}.
\end{aligned}
\end{equation}
$k$ here is slightly larger than $k$ before, but it's still strictly less than $1$. Let the $3 \times 3$ coefficient matrix above be $M$. We know that $\varepsilon$ can be arbitrarily small, hence $a$ can also be small enough. And because the eigenvalues of a matrix are continuously dependent on the elements in the matrix, the eigenvalues of matrix $M$ are close to those of matrix 

\begin{equation*}
B=\begin{pmatrix}
0 & 0 & 0\\
C & k & C\\
0 & 0 & k\\
\end{pmatrix}.
\end{equation*}
By direct calculation, we know that the eigenvalues of the matrix $B$ are $0$, $k$, $k$. On the other hand, we can calculate the characteristic polynomial of the matrix $M$ as

\begin{equation*}
f_M(\lambda)=(\lambda-a)(\lambda-k)^2-a^2(\lambda-k+2C)-aC(2\lambda-a-k).
\end{equation*}
When $a$ is sufficiently small, we can obtain(We can assume that $C$ is greater than 1)

$$f_M(0)=a(ak+Ck-aC-k^2)>0,$$
$$f_M(k)=-aC(k+a)<0.$$
Therefore, if the three eigenvalues of the matrix $M$ are respectively $\lambda_1$, $\lambda_2$ and $\lambda_3$ from small to large, then $-1<\lambda_1<0<\lambda_2<k<\lambda_3<1$, that is, the matrix $M$ has three different eigenvalues. So $M$ can be diagonalized, that is, there is an invertible matrix $P$ such that

\begin{equation*}
M=P^{-1}\begin{pmatrix}
\lambda_1 & 0 & 0\\
0 & \lambda_2 & 0\\
0 & 0 & \lambda_3\\
\end{pmatrix}P.
\end{equation*}
If we set $x_m=\begin{pmatrix}
a_m\\
b_m\\
c_m\\
\end{pmatrix}$, then we have

\begin{equation*}
x_{m+1}\le P^{-1}\begin{pmatrix}
\lambda_1 & 0 & 0\\
0 & \lambda_2 & 0\\
0 & 0 & \lambda_3\\
\end{pmatrix}Px_m\le \cdot\cdot\cdot\le P^{-1}\begin{pmatrix}
\lambda_1^m & 0 & 0\\
0 & \lambda_2^m & 0\\
0 & 0 & \lambda_3^m\\
\end{pmatrix}Px_1.
\end{equation*}
Summing over the sequence $x_m$, we can obtain

\begin{equation*}
x_1+x_2+\cdot\cdot\cdot+x_{m+1}\le P^{-1}\begin{pmatrix}
\mathop{\sum}\limits_{k=0}^{m}\lambda_1^k & 0 & 0\\
0 & \mathop{\sum}\limits_{k=0}^{m}\lambda_2^k & 0\\
0 & 0 & \mathop{\sum}\limits_{k=0}^{m}\lambda_3^k\\
\end{pmatrix}Px_1.
\end{equation*}
Therefore,

\begin{equation}\label{eq6.288}
\mathop{\sum}\limits_{m=1}^{\infty}x_m\le P^{-1}\begin{pmatrix}
\frac{1}{1-\lambda_1} & 0 & 0\\
0 & \frac{1}{1-\lambda_2} & 0\\
0 & 0 & \frac{1}{1-\lambda_3}\\
\end{pmatrix}Px_1\le Mx_1.
\end{equation}
So for sufficiently small $\varepsilon$, the sequence $(\overset{1}{\tilde{\alpha}}_{+m},\overset{2}{\tilde{\beta}}_{+m},r_m,\tilde{\alpha}_m,\tilde{\beta}_m);m=0,1,2,\cdot\cdot\cdot$ is Cauchy columns in the function space $\overset{1}{B}_B\times \overset{2}{B}_B\times\overset{1}{B}_N\times\overset{2}{B}_N\times \overset{2}{B}_N$, and so converges. We assume that the above sequence converges to $(\overset{1}{\tilde{\alpha}}_+,\overset{2}{\tilde{\beta}}_+,r,\tilde{\alpha},\tilde{\beta})\in \overset{1}{B}_B\times \overset{2}{B}_B\times \overset{1}{B}_N\times \overset{2}{B}_N\times \overset{2}{B}_N$ in the corresponding function space under the corresponding norm.

\end{proof}

\subsection{Proof of the existence theorem}\label{6.5}

The two properties mentioned above imply that the sequence $((\overset{1}{\tilde{\alpha}}_{+m},\overset{2}{\tilde{\beta}}_{+m},r_m,\tilde{\alpha}_m,\tilde{\beta}_m);m=0,1,2,\cdot\cdot\cdot)$ uniformly converges to $(\overset{1}{\tilde{\alpha}}_+,\overset{2}{\tilde{\beta}}_+,r,\tilde{\alpha},\tilde{\beta})\in \overset{1}{B}_B\times \overset{2}{B}_B\times \overset{1}{B}_N\times \overset{2}{B}_N\times \overset{2}{B}_N$ within the range $[0,a\varepsilon]\times [0,\varepsilon]\times T_\varepsilon\times T_\varepsilon\times T_\varepsilon$. We denote this convergence as $\rightrightarrows$ in the corresponding function space. Taking $\overset{1}{\tilde{\alpha}}_{+m}(u)$ as an example, since both $\overset{1}{\tilde{\alpha}}_{+m}(u)$ and $\overset{1}{\tilde{\alpha}}_+(u)$ belong to the function space $\overset{1}{B}_B$, we can apply the definition of $\overset{1}{B}_B$ in \eqref{eq6.34} to establish the following relation:

\begin{equation*}
\begin{aligned}
\mathop{lim}\limits_{m\rightarrow \infty}\mathop{sup}\limits_{[0,a\varepsilon]}\left|\overset{1}{\tilde{\alpha}}_{+m}(u)-\overset{1}{\tilde{\alpha}}_+(u)\right|&\le C\varepsilon^2\mathop{lim}\limits_{m\rightarrow\infty}\mathop{sup}\limits_{[0,a\varepsilon]}\left|\overset{1}{\tilde{\alpha}}_{+m}^{''}(u)-\overset{1}{\tilde{\alpha}}_+^{''}(u)\right|=C\varepsilon^2\mathop{lim}\limits_{m\rightarrow\infty}\Vert\overset{1}{\tilde{\alpha}}_{+m}(u)-\overset{1}{\tilde{\alpha}}_+(u)\Vert_1=0.
\end{aligned}
\end{equation*}

Therefore, we have $\overset{1}{\tilde{\alpha}}_{+m}(u)\rightrightarrows \overset{1}{\tilde{\alpha}}_+(u)$. Similarly, we can also observe $\overset{2}{\tilde{\beta}}_{+m}(v)\rightrightarrows \overset{2}{\tilde{\beta}}_+(v)$, $\tilde{\alpha}_m(u,v)\rightrightarrows \tilde{\alpha}(u,v)$ and $\tilde{\beta}_m(u,v)\rightrightarrows \tilde{\beta}(u,v)$. Hence, we establish

\begin{equation}\label{eq6.289}
\alpha_m(u,v)=\tilde{\alpha}_m(u,v)+\overset{1}{\tilde{\alpha}}_{+m}(u)\rightrightarrows\tilde{\alpha}(u,v)+\overset{1}{\tilde{\alpha}}_+(u)=:\alpha(u,v),
\end{equation}
\begin{equation}\label{eq6.290}
\beta_m(u,v)=\tilde{\beta}_m(u,v)+\overset{2}{\tilde{\beta}}_{+m}(v)\rightrightarrows\tilde{\beta}(u,v)+\overset{2}{\tilde{\beta}}_+(v)=:\beta(u,v).
\end{equation}

Likewise, by applying the same method and utilizing the definitions of function spaces described in Section \ref{6.3}, we obtain 

\begin{equation}\label{eq6.291}
\frac{\partial\alpha_m}{\partial u}(u,v)=\frac{\partial\tilde{\alpha}_m}{\partial u}(u,v)+\overset{1}{\tilde{\alpha}}_{+m}^{'}(u)\rightrightarrows\frac{\partial\tilde{\alpha}}{\partial u}(u,v)+\overset{1}{\tilde{\alpha}}_+^{'}(u)=\frac{\partial\alpha}{\partial u}(u,v),
\end{equation}
\begin{equation}\label{eq6.292}
\frac{\partial\alpha_m}{\partial v}(u,v)=\frac{\partial\tilde{\alpha}_m}{\partial v}(u,v)\rightrightarrows\frac{\partial\tilde{\alpha}}{\partial v}(u,v)=\frac{\partial\alpha}{\partial v}(u,v),
\end{equation}
\begin{equation}\label{eq6.293}
\frac{\partial\beta_m}{\partial u}(u,v)=\frac{\partial\tilde{\beta}_m}{\partial u}(u,v)\rightrightarrows\frac{\partial\tilde{\beta}}{\partial u}(u,v)=\frac{\partial\beta}{\partial u}(u,v),
\end{equation}
\begin{equation}\label{eq6.294}
\frac{\partial\beta_m}{\partial v}(u,v)=\frac{\partial\tilde{\beta}_m}{\partial v}(u,v)+\overset{2}{\tilde{\beta}}_{+m}^{'}(v)\rightrightarrows \frac{\partial\tilde{\beta}}{\partial v}(u,v)+\overset{2}{\tilde{\beta}}_+^{'}(v)=\frac{\partial\beta}{\partial v}(u,v).
\end{equation}
Similarly, for $r_m(u,v)$, we can observe the following convergence:

\begin{equation}\label{eq6.295}
r_m(u,v)\rightrightarrows r(u,v),\ \ \ \frac{\partial r_m}{\partial u}(u,v)\rightrightarrows \frac{\partial r}{\partial u}(u,v),\ \ \ \frac{\partial r_m}{\partial v}(u,v)\rightrightarrows \frac{\partial r}{\partial v}(u,v).
\end{equation}
Clearly, all second order derivatives of the aforementioned iterative sequence uniformly converge.

As $c_{in}$ and $c_{out}$ are smooth functions of $\alpha$ and $\beta$, we establish 

\begin{equation}\label{eq6.296}
c_{in,m}(u,v)=c_{in}(\alpha_m(u,v),\beta_m(u,v))\rightrightarrows c_{in}(\alpha(u,v),\beta(u,v))=:c_{in}(u,v).
\end{equation}
In a similar manner, we have

\begin{equation}\label{eq6.297}
\frac{\partial c_{in,m}}{\partial u}(u,v)\rightrightarrows \frac{\partial c_{in}}{\partial\alpha}(\alpha(u,v),\beta(u,v))\frac{\partial\alpha}{\partial u}(u,v)+\frac{\partial c_{in}}{\partial\beta}(\alpha(u,v),\beta(u,v))\frac{\partial\beta}{\partial u}(u,v)=\frac{\partial c_{in}}{\partial u}(u,v),
\end{equation}
\begin{equation}\label{eq6.298}
\frac{\partial c_{in,m}}{\partial v}(u,v)\rightrightarrows \frac{\partial c_{in}}{\partial\alpha}(\alpha(u,v),\beta(u,v))\frac{\partial\alpha}{\partial v}(u,v)+\frac{\partial c_{in}}{\partial\beta}(\alpha(u,v),\beta(u,v))\frac{\partial\beta}{\partial v}(u,v)=\frac{\partial c_{in}}{\partial v}(u,v).
\end{equation}

When replacing $c_{in}$ with $c_{out}$, similar estimations can be obtained, resulting in the convergence of $\mu_m(u,v)$ and $\nu_m(u,v)$. As defined in \eqref{eq6.23}, we deduce

\begin{equation}\label{eq6.299}
\mu_m(u,v)=\frac{1}{c_{out,m}-c_{in,m}}\frac{c_{out,m}}{c_{in,m}}\frac{\partial c_{in,m}}{\partial v}\rightrightarrows \frac{1}{c_{out}-c_{in}}\frac{c_{out}}{c_{in}}\frac{\partial c_{in}}{\partial v}=:\mu(u,v),
\end{equation}
\begin{equation}\label{eq6.300}
\nu_m(u,v)=-\frac{1}{c_{out,m}-c_{in,m}}\frac{c_{in,m}}{c_{out,m}}\frac{\partial c_{out,m}}{\partial u}\rightrightarrows -\frac{1}{c_{out}-c_{in}}\frac{c_{in}}{c_{out}}\frac{\partial c_{out}}{\partial u}=:\nu(u,v).
\end{equation}
Hence, based on the definition of $M_m(u,v)$ in \eqref{eq6.21}, we obtain 

\begin{equation}\label{eq6.301}
M_m(u,v)=\mu_m\frac{\partial r_m}{\partial u}+\nu_m\frac{\partial r_m}{\partial v}\rightrightarrows \mu\frac{\partial r}{\partial u}+\nu\frac{\partial r}{\partial v}=:M(u,v).
\end{equation}
In addition, by taking the limit of equation \eqref{eq6.135}, we yield 

\begin{equation}\label{eq6.302}
M(u,v)=\frac{\partial^2 r}{\partial u\partial v}(u,v).
\end{equation}

Considering the definitions of $\phi_m(u,v)$ and $\psi_m(u,v)$ in \eqref{eq6.5}, we deduce

\begin{equation}\label{eq6.303}
\phi_m(u,v)=\frac{1}{c_{in,m}}\frac{\partial r_m}{\partial u}\rightrightarrows \frac{1}{c_{in}}\frac{\partial r}{\partial u}=:\phi(u,v),\ \ \ \psi_m(u,v)=\frac{1}{c_{out,m}}\frac{\partial r_m}{\partial v}\rightrightarrows \frac{1}{c_{out}}\frac{\partial r}{\partial v}=:\psi(u,v).
\end{equation}
Using the same method as before, it is easy to know that all first order derivatives of $\phi_m(u,v)$ and $\psi_m(u,v)$ also converge uniformly to the derivatives of their limit functions. Thus, by \eqref{eq6.4}, we establish the convergence of $t_m(u,v)$ and all of its first derivatives:

\begin{equation}\label{eq6.304}
\begin{aligned}
t_m(u,v)=&\int_0^u(\phi_m+\psi_m)(u^{'},u^{'})du^{'}+\int_u^v\psi_m(u,v^{'})dv^{'}\\\rightrightarrows& \int_0^u(\phi+\psi)(u^{'},u^{'})du^{'}+\int_u^v\psi(u,v^{'})dv^{'}=:t(u,v),
\end{aligned}
\end{equation}
\begin{equation}\label{eq6.305}
\frac{\partial t_m}{\partial v}(u,v)=\psi_m(u,v)\rightrightarrows \psi(u,v)=\frac{\partial t}{\partial v}(u,v),
\end{equation}
\begin{equation}\label{eq6.306}
\frac{\partial t_m}{\partial u}(u,v)=\phi_m(u,u)+\int_u^v\frac{\partial\psi_m}{\partial u}(u,v^{'})dv^{'}\rightrightarrows \phi(u,u)+\int_u^v\frac{\partial\psi}{\partial u}(u,v^{'})dv^{'}=\frac{\partial t}{\partial u}(u,v).
\end{equation}
Taking the derivative of $\phi(u,v)$ with respect to $v$ and $\psi(u,v)$ with respect to $u$, based on the definition above, we obtain $\frac{\partial\phi}{\partial v}=-\frac{1}{c_{in}^2}\frac{\partial c_{in}}{\partial v}\frac{\partial r}{\partial u}+\frac{1}{c_{in}}\frac{\partial^2 r}{\partial u\partial v}$, $\frac{\partial\psi}{\partial u}=-\frac{1}{c_{out}^2}\frac{\partial c_{out}}{\partial u}\frac{\partial r}{\partial v}+\frac{1}{c_{out}}\frac{\partial^2 r}{\partial u\partial v}$. Furthermore, combining the properties of the obtained limit function $M(u,v)$, expressed as

\begin{equation}\label{eq6.307}
M=\frac{\partial^2 r}{\partial u\partial v}=\mu\frac{\partial r}{\partial u}+\nu\frac{\partial r}{\partial v},
\end{equation}
we deduce $\frac{\partial\phi}{\partial v}-\frac{\partial\psi}{\partial u}=0$. Hence, by\eqref{eq6.306},\eqref{eq6.303} and \eqref{eq6.305}, we establish

\begin{equation}\label{eq6.308}
\frac{\partial t}{\partial u}(u,v)=\phi(u,v)=\frac{1}{c_{in}(u,v)}\frac{\partial r}{\partial u}(u,v),
\end{equation}
\begin{equation}\label{eq6.309}
\frac{\partial t}{\partial v}(u,v)=\psi(u,v)=\frac{1}{c_{out}(u,v)}\frac{\partial r}{\partial v}(u,v).
\end{equation}
Consequently, we have demonstrated that the limit functions comply with the characteristic equations \eqref{eq5.56} and \eqref{eq5.57}.

Subsequently, our aim is to prove that the limit functions also satisfy the remaining two differential equations, \eqref{eq5.58} and \eqref{eq5.59}. Considering the definition of $A_m(u,v)$ in \eqref{eq6.29}, we naturally obtain 

\begin{equation}\label{eq6.310}
A_m(u,v)=A(\alpha_m(u,v),\beta_m(u,v))\rightrightarrows A(\alpha(u,v),\beta(u,v))=:A(u,v).
\end{equation}
Accordingly, based on the definitions of $\alpha_{m+1}(u,v)$ and $\beta_{m+1}(u,v)$ presented in equations \eqref{eq6.27} and \eqref{eq6.28}, we deduce 

\begin{equation}\label{eq6.311}
\frac{\partial\alpha_{m+1}}{\partial v}(u,v)\rightrightarrows A(u,v)\frac{\partial t}{\partial v}(u,v)=\frac{\partial\alpha}{\partial v}(u,v),
\end{equation}
\begin{equation}\label{eq6.312}
\frac{\partial\beta_{m+1}}{\partial u}(u,v)\rightrightarrows A(u,v)\frac{\partial t}{\partial u}(u,v)=\frac{\partial\beta}{\partial u}(u,v).
\end{equation}
In the above derivation, we also used \eqref{eq6.292} and \eqref{eq6.293}.

Subsequently, our task is to demonstrate that the obtained limit functions satisfy the boundary conditions, jump conditions, and determinism condition outlined in Section \ref{5.5}. By taking the limit of the iterative process, it becomes apparent that the remaining conditions are trivial, except for equations \eqref{eq5.60}, \eqref{eq5.61}, \eqref{eq5.63}, \eqref{eq5.66}, and \eqref{eq5.67}. Below we first prove the limit functions of $\overset{i}{\Gamma}_m$, which are $\overset{1}{\Gamma}(u)=\frac{\overset{1}{c_{out}}_+}{\overset{1}{c_{in}}_+}\frac{\overset{1}{V}-\overset{1}{c_{in}}_+}{\overset{1}{c_{out}}_+-\overset{1}{V}}$, $\overset{2}{\Gamma}(v)=a\frac{\overset{2}{c_{out}}_+}{\overset{2}{c_{in}}_+}\frac{\overset{2}{V}-\overset{2}{c_{in}}_+}{\overset{2}{c_{out}}_+-\overset{2}{V}}$, satisfy

\begin{equation*}
\frac{\partial r}{\partial v}(u,u)=\frac{\partial r}{\partial u}(u,u)\overset{1}{\Gamma}(u),\ \ \ \frac{\partial r}{\partial v}(av,v)=\frac{\partial r}{\partial u}(av,v)\overset{2}{\Gamma}(v).
\end{equation*}
Given that we have already established the uniform convergence of $\alpha_m(u,v)$ and $\beta_m(u,v)$, we can deduce that $\overset{i}{\alpha}_{+m}$ and $\overset{i}{\beta}_{+m}$ $(i=1,2)$, as defined in \eqref{eq6.10} and \eqref{eq6.11}, also uniformly converge. Similarly, we can prove that $\overset{i}{t}_{+m}$ and $\overset{i}{r}_{+m}$ uniformly converge. Consequently, based on the definitions of $\overset{i}{\alpha}_{-m}$ and $\overset{i}{\beta}_{-m}$ in \eqref{eq6.7} and \eqref{eq6.8}, it is evident that they also converge uniformly. As a result, by \eqref{eq6.13} and \eqref{eq6.14}, we observe that $\overset{i}{\rho}_{\pm m}$ and $\overset{i}{w}_{\pm m}$ converge uniformly. Furthermore, \eqref{eq6.12} implies that $\overset{i}{V}_m$ also uniformly converges, denoted as $\overset{i}{V}_m \rightrightarrows \overset{i}{V}$. Considering the definitions of $\overset{i}{c_{in}}_{+,m}(u)$ and $\overset{i}{c_{out}}_{+,m}(u)$ in \eqref{eq6.15}, we can deduce their uniform convergence as well. Therefore, $\overset{i}{\Gamma}_m$ also uniformly converges, denoted as $\overset{i}{\Gamma}_m \rightrightarrows \overset{i}{\Gamma}$, where $\overset{i}{\Gamma}$ is defined as mentioned above. Lastly, based on the definition of $\overset{i}{\gamma}_m$ in \eqref{eq6.22}, we conclude that

\begin{equation}\label{eq6.313}
\overset{1}{\gamma}_m(u)\rightrightarrows \frac{\overset{2}{\Gamma}(u)}{\overset{1}{\Gamma}(u)}=:\overset{1}{\gamma}(u),\ \ \ \overset{2}{\gamma}_m(v)\rightrightarrows \frac{\overset{2}{\Gamma}(v)}{\overset{1}{\Gamma}(av)}=:\overset{2}{\gamma}(v).
\end{equation}

By utilizing the definitions of $\overset{1}{\Phi}_m(u)$, $\overset{1}{\Lambda}_m(u)$, and $M_m(u,v)$ in \eqref{eq6.17}, \eqref{eq6.19}, and \eqref{eq6.21}, and incorporating the properties of the earlier obtained limit function $M(u,v)$ in \eqref{eq6.307}, we establish

\begin{equation}\label{eq6.314}
\begin{aligned}
\overset{1}{\Phi}_m(u)=&\overset{1}{\gamma}_m(u)\frac{\partial r_m}{\partial u}(au,u)-\frac{1}{\Gamma_0}+\frac{1}{\overset{1}{\Gamma}_m(u)}\int_{au}^u M_m(u^{'},u)du^{'}\\
\rightrightarrows &\overset{1}{\gamma}(u)\frac{\partial r}{\partial u}(au,u)-\frac{1}{\Gamma_0}+\frac{1}{\overset{1}{\Gamma}(u)}\left(\frac{\partial r}{\partial v}(u,u)-\frac{\partial r}{\partial v}(au,u)\right).
\end{aligned}
\end{equation}

To obtain 

\begin{equation*}
\frac{\partial r_{m+1}}{\partial u}(u,v)=\frac{1}{\Gamma_0}+\overset{1}{\Phi}_m(u)+\int_u^v M_m(u,v^{'})dv^{'},
\end{equation*}
we differentiate the expression of $r_{m+1}(u,v)$ given in \eqref{eq6.16} with respect to $u$. Upon taking the limit of the aforementioned equation, we obtain the limit function $r(u,v)$ that satisfies

\begin{equation*}
\frac{\partial r}{\partial u}(u,v)=\overset{1}{\gamma}(u)\frac{\partial r}{\partial u}(au,u)+\frac{1}{\overset{1}{\Gamma}(u)}\left(\frac{\partial r}{\partial v}(u,u)-\frac{\partial r}{\partial v}(au,u)\right)+\left(\frac{\partial r}{\partial u}(u,v)-\frac{\partial r}{\partial u}(u,u)\right),
\end{equation*}
i.e.,

\begin{equation}\label{eq6.315}
\overset{1}{\Gamma}(u)\frac{\partial r}{\partial u}(u,u)+\frac{\partial r}{\partial v}(au,u)=\overset{2}{\Gamma}(u)\frac{\partial r}{\partial u}(au,u)+\frac{\partial r}{\partial v}(u,u).
\end{equation}

Similarly, by applying the definitions of $\overset{2}{\Phi}_m(v)$, $\overset{2}{\Lambda}_m(v)$, and $M_m(u,v)$ provided in \eqref{eq6.18}, \eqref{eq6.20}, and \eqref{eq6.21}, and considering the properties of the previously derived limit function $M(u,v)$ in \eqref{eq6.307}, we deduce

\begin{equation}\label{eq6.316}
\begin{aligned}
\overset{2}{\Phi}_m(v)=&\overset{2}{\gamma}_m(v)\frac{\partial r_m}{\partial v}(av,av)-1+\overset{2}{\Gamma}_m(v)\int_{av}^v M_m(av,v^{'})dv^{'}\\
\rightrightarrows &\overset{2}{\gamma}(v)\frac{\partial r}{\partial v}(av,av)-1+\overset{2}{\Gamma}(v)\left(\frac{\partial r}{\partial u}(av,v)-\frac{\partial r}{\partial u}(av,av)\right).
\end{aligned}
\end{equation}
Taking the partial derivative of $r_{m+1}(u,v)$ with respect to $v$, as presented in \eqref{eq6.16}, yields:

\begin{equation*}
\frac{\partial r_{m+1}}{\partial v}(u,v)=1+\overset{2}{\Phi}_m(v)+\int_{av}^u M_m(u^{'},v)du^{'}.
\end{equation*}
Upon taking the limit of the above equation, we obtain

\begin{equation*}
\frac{\partial r}{\partial v}(u,v)=\overset{2}{\gamma}(v)\frac{\partial r}{\partial v}(av,av)+\overset{2}{\Gamma}(v)\left(\frac{\partial r}{\partial u}(av,v)-\frac{\partial r}{\partial u}(av,av)\right)+\left(\frac{\partial r}{\partial v}(u,v)-\frac{\partial r}{\partial v}(av,v)\right),
\end{equation*}
i.e.,

\begin{equation}\label{eq6.317}
\overset{1}{\Gamma}(av)\frac{\partial r}{\partial v}(av,v)+\overset{1}{\Gamma}(av)\overset{2}{\Gamma}(v)\frac{\partial r}{\partial u}(av,av)=\overset{2}{\Gamma}(v)\frac{\partial r}{\partial v}(av,av)+\overset{1}{\Gamma}(av)\overset{2}{\Gamma}(v)\frac{\partial r}{\partial u}(av,v).
\end{equation}

We define $f(u)$ and $g(v)$ as

\begin{equation}\label{eq6.318}
f(u):=\frac{\partial r}{\partial u}(u,u)\overset{1}{\Gamma}(u)-\frac{\partial r}{\partial v}(u,u),\ \ \ g(v):=\frac{\partial r}{\partial u}(av,v)\overset{2}{\Gamma}(v)-\frac{\partial r}{\partial v}(av,v).
\end{equation}
It is evident that $f$ is a $C^1$ function defined on the domain, and it satisfies $f(0)=0$. $g$ obviously has the same properties. Equation \eqref{eq6.315} is equivalent to $f(u)=g(u)$. Therefore, \eqref{eq6.317} simplifies to $\overset{1}{\Gamma}(av)f(v)=\overset{2}{\Gamma}(v)f(av)$, which can be expressed as $f(av)=\frac{1}{\overset{2}{\gamma}(v)}f(v)$.

Considering $\frac{1}{\overset{2}{\gamma}(v)}=1+\mathcal{O}(v)$, for sufficiently small $\varepsilon$, there exists a constant $C$ such that

\begin{equation*}
\left|f(av)-f(v)\right|\le Cv\left|f(v)\right|,
\end{equation*}
\begin{equation*}
\left|f(av)\right|\le \left(1+Cv\right)\left|f(v)\right|.
\end{equation*}
Hence, we obtain

\begin{equation*}
\left|f(a^m v)\right|\le \left(1+Ca^{m-1}v\right)\left|f(a^{m-1}v)\right|\le \cdot\cdot\cdot\le (1+Ca^{m-1}v)\cdot\cdot\cdot (1+Cv)\left|f(v)\right|.
\end{equation*}
Introducing the definition: $h_m=(1+Ca^{m-1}v)\cdot\cdot\cdot (1+Cv)$, we can derive

\begin{equation*}
\begin{aligned}
log h_m=&log(1+Cv)+log(1+Cav)+\cdot\cdot\cdot+log(1+Ca^{m-1}v)\\
\le & Cv+Cav+\cdot\cdot\cdot+Ca^{m-1}v\\
=&Cv(1+a+\cdot\cdot\cdot+a^{m-1})\\
\le & \frac{C\varepsilon}{1-a}.
\end{aligned}
\end{equation*}
Consequently, $h_m$ is uniformly bounded, and there exists a constant $M$ such that $\left| h_m \right|\le M$. Thus,

\begin{equation*}
\begin{aligned}
\left|f(v)\right|\le & \left|f(v)-f(av)\right|+\left|f(av)-f(a^2 v)\right|+\cdot\cdot\cdot+\left|f(a^m v)-f(a^{m+1}v)\right|+\left|f(a^{m+1}v)\right|\\
\le & Cv\left|f(v)\right|+Cav\left|f(av)\right|+\cdot\cdot\cdot+Ca^m v\left|f(a^m v)\right|+\left|f(a^{m+1}v)\right|\\
\le & CMv(1+a+\cdot\cdot\cdot+a^m)\left| f(v) \right|+\left|f(a^{m+1}v)\right|\\.
\end{aligned}
\end{equation*}
By taking the limit on both sides of the above equation with respect to $m$, we obtain $\left|f(v)\right|\le \frac{CM}{1-a}v\left|f(v)\right|$. For sufficiently small $\varepsilon$, the coefficient on the right hand side satisfies $\frac{CM}{1 − a}v<1$. Thus, when $\varepsilon$ is sufficiently small, we have $f(v)\equiv 0$, which implies

\begin{equation}\label{eq6.319}
f(v)=g(v)\equiv 0.
\end{equation}

In the following, we prove \eqref{eq5.60} and \eqref{eq5.61}. By \eqref{eq6.111} and \eqref{eq6.112}, we derive the following expression for $\overset{i}{V}$:

\begin{equation*}
\overset{1}{V}(u)=\overset{1}{c_{in}}_+\overset{1}{c_{out}}_+\frac{1+\overset{1}{\Gamma}(u)}{\overset{1}{c_{in}}_+\overset{1}{\Gamma}(u)+\overset{1}{c_{out}}_+},\ \ \ \overset{2}{V}(v)=\overset{2}{c_{in}}_+\overset{2}{c_{out}}_+\frac{a+\overset{2}{\Gamma}(v)}{\overset{2}{c_{in}}_+\overset{2}{\Gamma}(v)+a\overset{2}{c_{out}}_+}.
\end{equation*}
Employing the characteristic equations \eqref{eq6.308} and \eqref{eq6.309} and the above equation \eqref{eq6.319}, we derive

\begin{equation}\label{eq6.320}
\begin{aligned}
\frac{{d\overset{1}{r}_+}/{du}(u)}{{d\overset{1}{t}_+}/{du}(u)}=\frac{\overset{1}{c_{in}}_+\overset{1}{c_{out}}_+}{\overset{1}{c_{in}}_+\overset{1}{\Gamma}+\overset{1}{c_{out}}_+}+\frac{\overset{1}{\Gamma}\overset{1}{c_{in}}_+\overset{1}{c_{out}}_+}{\overset{1}{c_{in}}_+\overset{1}{\Gamma}+\overset{1}{c_{out}}_+}= \overset{1}{V}(u),
\end{aligned}
\end{equation}
\begin{equation}\label{eq6.321}
\begin{aligned}
\frac{{d\overset{2}{r}_+}/ {dv}(v)}{{d\overset{2}{t}_+}/{dv}(v)}=\frac{a\overset{2}{c_{in}}_+ \overset{2}{c_{out}}_+}{a\overset{2}{c_{out}}_+ +\overset{2}{\Gamma}\overset{2}{c_{in}}_+}+\frac{\overset{2}{\Gamma}\overset{2}{c_{in}}_+ \overset{2}{c_{out}}_+}{a\overset{2}{c_{out}}_+ +\overset{2}{\Gamma}\overset{2}{c_{in}}_+}=\overset{2}{V}(v).
\end{aligned}
\end{equation}
This completes the proof of \eqref{eq5.60} and \eqref{eq5.61}.

To prove \eqref{eq5.63}, we only need to show that

\begin{equation*}
\overset{1}{\alpha}_+(u)=\overset{1}{H}\left(\overset{1}{\beta}_+(u),\overset{1}{\alpha}_-(u),\overset{1}{\beta}_-(u)\right),\ \ \ \overset{2}{\beta}_+(v)=\overset{2}{H}\left(\overset{2}{\alpha}_+(v),\overset{2}{\alpha}_-(v),\overset{2}{\beta}_-(v)\right),
\end{equation*}
where we used

\begin{equation*}
\frac{\partial J}{\partial\alpha_+}(\overset{1}{\alpha}_+(0),\overset{1}{\beta}_+(0),\overset{1}{\alpha}_-(0),\overset{1}{\beta}_-(0))=\frac{\partial J}{\partial\alpha_+}(\beta_0,\beta_0,\overset{1}{\alpha}_{-0},\overset{1}{\beta}_{-0})\neq 0,
\end{equation*}
\begin{equation*}
\frac{\partial J}{\partial\beta_+}(\overset{2}{\alpha}_+(0),\overset{2}{\beta}_+(0),\overset{2}{\alpha}_-(0),\overset{2}{\beta}_-(0))=\frac{\partial J}{\partial\beta_+}(\beta_0,\beta_0,\overset{2}{\alpha}_{-0},\overset{2}{\beta}_{-0})\neq 0.
\end{equation*}
By taking the limit of \eqref{eq6.27} at $u=v$ and \eqref{eq6.28} at $u=av$, we obtain

\begin{equation}\label{eq6.322}
\overset{1}{\alpha}_+(u)=\overset{1}{\tilde{\alpha}}_+(u),\ \ \ \overset{2}{\beta}_+(v)=\overset{2}{\tilde{\beta}}_+(v).
\end{equation}
Subsequently, we can take the limit of \eqref{eq6.25} and \eqref{eq6.26} to obtain the desired result.

Finally, we need to demonstrate that the limit functions satisfy \eqref{eq5.66} and \eqref{eq5.67}. We recall that the determinism condition is satisfied at the interaction point. By \eqref{eq4.4} and \eqref{eq4.6}, we have

\begin{equation*}
\overset{1}{c_{in}}_+(0)=-\eta_0<\overset{1}{V}_0=\overset{1}{V}(0),\ \ \ \overset{2}{V}(0)=\overset{2}{V}_0<\eta_0=\overset{2}{c_{out}}_+(0),
\end{equation*}
and

\begin{equation*}
\overset{1}{V}(0)=\overset{1}{V}_0<(\overset{1}{c_{in}}_0^*)_0=\overset{1}{c_{in}}_-(0),\ \ \ \overset{2}{c_{out}}_-(0)=(\overset{2}{c_{out}}_0^*)_0<\overset{2}{V}_0=\overset{2}{V}(0).
\end{equation*}
According to the continuity of the limit functions, it can be inferred that \eqref{eq5.66} and \eqref{eq5.67} hold when $\varepsilon$ is small enough. Thus, we have proved that the obtained limit function $(\alpha,\beta,t,r)$ is a twice continuously differentiable solution to the shock interaction problem mentioned in Section \ref{5.5}. By taking the limit of \eqref{eq6.94}, \eqref{eq6.95}, \eqref{eq6.84} and \eqref{eq6.85}, we can directly compute the determinant of the Jacobian matrix $\frac{\partial(t,r)}{\partial(u,v)}$, as shown below:

\begin{equation}\label{eq6.323}
\begin{aligned}
\begin{vmatrix}
\frac{\partial t}{\partial u} & \frac{\partial t}{\partial v}\\
\frac{\partial r}{\partial u} & \frac{\partial r}{\partial v}
\end{vmatrix}=\left(-\frac{1}{\eta_0\Gamma_0}+\mathcal{O}(v)\right)\left(1+\mathcal{O}(v)\right)-\left(\frac{1}{\eta_0}+\mathcal{O}(v)\right)\left(\frac{1}{\Gamma_0}+\mathcal{O}(v)\right)=-\frac{2}{\eta_0\Gamma_0}+\mathcal{O}(v).
\end{aligned}
\end{equation}

We observe that for small enough $\varepsilon$, the Jacobian matrix remains non-degenerate. This allows us to choose a sufficiently small $\varepsilon$ such that $\alpha$ and $\beta$ are twice differentiable functions of $(t,r)$ on the image of $T_\varepsilon$ by the map $(u,v)\longmapsto (t(u,v),r(u,v))$.  Consequently, we have proven Theorem 1.

\subsection{Asymptotic form}\label{6.6}

Using the initial values calculated earlier, we can expand the solution $(\alpha,\beta,t,r)$ of the shock wave interaction problem at the interaction point in the state behind. Any twice differentiable solution that satisfies this problem can be represented in the following asymptotic form:

\begin{flalign*}
\begin{cases}
\alpha(u,v)&=\beta_0+\alpha_0^{'}u+\mathcal{O}(v^2),\\ 
\beta(u,v)&=\beta_0+\beta_0^{'}v+\mathcal{O}(v^2),\\ 
t(u,v)&=\frac{1}{\eta_0}\left(v-\frac{u}{\Gamma_0}\right)+\mathcal{O}(v^2),\\ 
r(u,v)&=r_0+\frac{u}{\Gamma_0}+v+\mathcal{O}(v^2). 
\end{cases}
\end{flalign*}

\subsection{Uniqueness}\label{6.7}

\begin{proof}

Next we prove Theorem \ref{uniqueness}:

Firstly, if we replace all subscripts $m$ and $m+1$ in the iteration scheme in Section \ref{6.1} with $k$, where $k=1,2$, then all equations still hold. At this stage, only the functions $\alpha_k(u,v)$, $\beta_k(u,v)$, $t_k(u,v)$ and $r_k(u,v)$ can be known. We define

\begin{equation}\label{eq6.324}
\phi_k(u,v):=\frac{1}{c_{in,k}(u,v)}\frac{\partial r_k}{\partial u}(u,v),\ \ \ \psi_k(u,v):=\frac{1}{c_{out,k}(u,v)}\frac{\partial r_k}{\partial v}(u,v),
\end{equation}
where

\begin{equation}\label{eq6.325}
c_{in,k}(u,v):=c_{in}(\alpha_k(u,v),\beta_k(u,v)),\ \ \ c_{out,k}(u,v):=c_{out}(\alpha_k(u,v),\beta_k(u,v)).
\end{equation}
Since $(\alpha_k,\beta_k,t_k,r_k)$ is a solution to the shock wave interaction problem, it satisfies the characteristic equations given by

\begin{equation}\label{eq6.326}
\frac{1}{c_{in,k}(u,v)}\frac{\partial r_k}{\partial u}(u,v)=\frac{\partial t_k}{\partial u}(u,v),\ \ \ \frac{1}{c_{out,k}(u,v)}\frac{\partial r_k}{\partial v}(u,v)=\frac{\partial t_k}{\partial v}(u,v).
\end{equation}
Thus, at this point, we have

\begin{equation}\label{eq6.327}
\phi_k(u,v)=\frac{\partial t_k}{\partial u}(u,v),\ \ \ \psi_k(u,v)=\frac{\partial t_k}{\partial v}(u,v).
\end{equation}

A direct calculation shows that \eqref{eq6.4}-\eqref{eq6.6} are satisfied after replacing the subscripts. To define $\overset{i}{\alpha}_{\pm,k}$, $\overset{i}{\beta}_{\pm,k}$, $\overset{i}{t}_{+,k}$, $\overset{i}{r}_{+,k}$, $\overset{i}{\rho}_{\pm,k}$, $\overset{i}{w}_{\pm,k}$, $\overset{i}{V}_k$, $\overset{i}{c_{in}}_{+,k}$ and $\overset{i}{c_{out}}_{+,k}$ in \eqref{eq6.7}-\eqref{eq6.15}, we replace all the subscripts $m$ with $k$. Following this, we can obtain the definition of $c_{in,k}$, $c_{out,k}$, $\mu_k$, $\nu_k$, $M_k$, $\overset{i}{\gamma}_k$, $\overset{i}{\Gamma}_k$， $\overset{i}{\Lambda}_k$ and $\overset{i}{\Phi}_k$, as shown in \eqref{eq6.17}-\eqref{eq6.24}.

Next, we begin our proof by showing that

\begin{equation}\label{eq6.328}
\frac{\partial r_k}{\partial v}(u,u)=\frac{\partial r_k}{\partial u}(u,u)\overset{1}{\Gamma}_k(u),\ \ \ \frac{\partial r_k}{\partial v}(av,v)=\frac{\partial r_k}{\partial u}(av,v)\overset{2}{\Gamma}_k(v).
\end{equation}
To establish this, we prove the following equation:

\begin{equation}\label{eq6.329}
M_k(u,v)=\mu_k(u,v)\frac{\partial r_k}{\partial u}(u,v)+\nu_k(u,v)\frac{\partial r_k}{\partial v}(u,v)=\frac{\partial^2 r_k}{\partial u\partial v}(u,v).
\end{equation}
To derive this equation, we simply take the derivative of the first equation in \eqref{eq6.326} with respect to $v$ and the derivative of the second equation in \eqref{eq6.326} with respect to $u$, and eliminate $t_k$ to obtain the desired result.

 Since $\overset{i}{V}_k$, $\overset{i}{t}_k$, and $\overset{i}{r}_k$ must satisfy the boundary conditions in \eqref{eq5.60} and \eqref{eq5.61}, as well as \eqref{eq6.326}, by a equation similar to \eqref{eq5.9}, we have

\begin{equation*}
\begin{aligned}
\overset{1}{\Gamma}_k(u)=\frac{\overset{1}{c_{out}}_{+,k}(u)}{\overset{1}{c_{in}}_{+,k}(u)}\frac{\frac{{d\overset{1}{r}_{+,k}}/{du}}{{d\overset{1}{t}_{+,k}}/{du}}(u)-\overset{1}{c_{in}}_{+,k}(u)}{\overset{1}{c_{out}}_{+,k}(u)-\frac{{d\overset{1}{r}_{+,k}}/{du}}{{d\overset{1}{t}_{+,k}}/{du}}(u)}=\frac{\overset{1}{c_{out}}_{+,k}(u)}{\overset{1}{c_{in}}_{+,k}(u)}\frac{\frac{\partial t_k}{\partial v}(u,u)}{\frac{\partial t_k}{\partial u}(u,u)}=\frac{\frac{\partial r_k}{\partial v}(u,u)}{\frac{\partial r_k}{\partial u}(u,u)}.
\end{aligned}
\end{equation*}
Using the same method, we have $\overset{2}{\Gamma}_k(v)=\frac{\frac{\partial r_k}{\partial v}(av,v)}{\frac{\partial r_k}{\partial u}(av,v)}$. Thus we have proven \eqref{eq6.328}, and can now write

\begin{equation*}
\begin{aligned}
\overset{1}{\Phi}_k(u)=&\int_0^u\overset{1}{\Lambda}_k(u^{'})du^{'}+\frac{1}{\overset{1}{\Gamma}_k(u)}\int_{au}^u M_k(u^{'},u)du^{'}\\
=&\overset{1}{\gamma}_k(u)\frac{\partial r_k}{\partial u}(au,u)-\overset{1}{\gamma}_k(0)\frac{\partial r_k}{\partial u}(0,0)+\frac{1}{\overset{1}{\Gamma}_k(u)}\left(\frac{\partial r_k}{\partial v}(u,u)-\frac{\partial r_k}{\partial v}(au,u)\right)\\
=&\frac{1}{\overset{1}{\Gamma}_k(u)}\frac{\partial r_k}{\partial v}(au,u)-\frac{1}{\Gamma_0}+\frac{1}{\overset{1}{\Gamma}_k(u)}\left(\frac{\partial r_k}{\partial v}(u,u)-\frac{\partial r_k}{\partial v}(au,u)\right)\\
=&\frac{\partial r_k}{\partial u}(u,u)-\frac{1}{\Gamma_0},
\end{aligned}
\end{equation*}
\begin{equation*}
\begin{aligned}
\overset{2}{\Phi}_k(v)=&\int_0^v\overset{2}{\Lambda}_k(v^{'})dv^{'}+\overset{2}{\Gamma}_k(v)\int_{av}^v M_k(av,v^{'})dv^{'}\\
=&\overset{2}{\gamma}_k(v)\frac{\partial r_k}{\partial v}(av,av)-\overset{2}{\gamma}_k(0)\frac{\partial r_k}{\partial v}(0,0)+\overset{2}{\Gamma}_k(v)\left(\frac{\partial r_k}{\partial u}(av,v)-\frac{\partial r_k}{\partial u}(av,av)\right)\\
=&\overset{2}{\Gamma}_k(v)\frac{\partial r_k}{\partial u}(av,av)-1+\overset{2}{\Gamma}_k(v)\left(\frac{\partial r_k}{\partial u}(av,v)-\frac{\partial r_k}{\partial u}(av,av)\right)\\
=&\frac{\partial r_k}{\partial v}(av,v)-1.
\end{aligned}
\end{equation*}
Then an equation similar to \eqref{eq6.16} can be obtained by a simple calculation:

\begin{equation*}
\begin{aligned}
&r_0+\frac{u}{\Gamma_0}+v+\int_0^u \overset{1}{\Phi}_k(u^{'})du^{'}+\int_0^v \overset{2}{\Phi}_k(v^{'})dv^{'}+\int_0^u\int_u^{v^{'}} M_k(u^{'},v^{'})du^{'}dv^{'}+\int_0^v \int_{av^{'}}^u M_k(u^{'},v^{'})du^{'}dv^{'}\\
=&r_0+\frac{u}{\Gamma_0}+v+\int_0^u\left(\frac{\partial r_k}{\partial u}(u^{'},u^{'})-\frac{1}{\Gamma_0}\right)du^{'}+\int_0^v\left(\frac{\partial r_k}{\partial v}(av^{'},v^{'})-1\right)dv^{'}\\
&+\int_0^u\int_u^{v^{'}} \frac{\partial^2 r_k}{\partial u\partial v} (u^{'},v^{'})du^{'}dv^{'}+\int_0^v\int_{av^{'}}^u\frac{\partial^2 r_k}{\partial u\partial v}(u^{'},v^{'})du^{'}dv^{'}\\
=&r_k(u,v).
\end{aligned}
\end{equation*}

To complete the proof, we define the functions $\overset{1}{\tilde{\alpha}}_{+,k}(u)$ and $\overset{2}{\tilde{\beta}}_{+,k}(v)$ as follows:

\begin{equation}\label{eq6.330}
\overset{1}{\tilde{\alpha}}_{+,k}(u)=\overset{1}{H}(\overset{1}{\beta}_{+,k}(u),\overset{1}{\alpha}_{-,k}(u),\overset{1}{\beta}_{-,k}(u)),
\end{equation}
\begin{equation}\label{eq6.331}
\overset{2}{\tilde{\beta}}_{+,k}(v)=\overset{2}{H}(\overset{2}{\alpha}_{+,k}(v),\overset{2}{\alpha}_{-,k}(v),\overset{2}{\beta}_{-,k}(v)).
\end{equation}
Since the solution $(\alpha_k,\beta_k,t_k,r_k)$ satisfies the jump condition \eqref{eq5.63} and the partial derivatives of $J$ with respect to $\alpha_+$ and $\beta_+$ at the interaction point are not zero, i.e.,

\begin{equation*}
\frac{\partial J}{\partial\alpha_+}(\overset{1}{\alpha}_{+,k}(0),\overset{1}{\beta}_{+,k}(0),\overset{1}{\alpha}_{-,k}(0),\overset{1}{\beta}_{-,k}(0))=\frac{\partial J}{\partial\alpha_+}(\beta_0,\beta_0,\overset{1}{\alpha}_{-0},\overset{1}{\beta}_{-0})\neq 0,
\end{equation*}
\begin{equation*}
\frac{\partial J}{\partial\beta_+}(\overset{2}{\alpha}_{+,k}(0),\overset{2}{\beta}_{+,k}(0),\overset{2}{\alpha}_{-,k}(0),\overset{2}{\beta}_{-,k}(0))=\frac{\partial J}{\partial\beta_+}(\beta_0,\beta_0,\overset{2}{\alpha}_{-0},\overset{2}{\beta}_{-0})\neq 0,
\end{equation*}
we can write

\begin{equation}\label{eq6.332}
\overset{1}{\alpha}_{+,k}(u)=\overset{1}{\tilde{\alpha}}_{+,k}(u),\ \ \ \overset{2}{\beta}_{+,k}(v)=\overset{2}{\tilde{\beta}}_{+,k}(v).
\end{equation}

If we define $A_k(u,v)=A(\alpha_k(u,v),\beta_k(u,v))$, we can then obtain

\begin{equation*}
\alpha_k(u,v)=\overset{1}{\alpha}_{+,k}(u)+\int_u^v\left(A_k\frac{\partial t_k}{\partial v}\right)(u,v^{'})dv^{'}=\overset{1}{\tilde{\alpha}}_{+,k}(u)+\int_u^v\left(A_k\frac{\partial t_k}{\partial v}\right)(u,v^{'})dv^{'},
\end{equation*}
\begin{equation*}
\beta_k(u,v)=\overset{2}{\beta}_{+,k}(v)+\int_{av}^u\left(A_k\frac{\partial t_k}{\partial u}\right)(u^{'},v)du^{'}=\overset{2}{\tilde{\beta}}_{+,k}(v)+\int_{av}^u\left(A_k\frac{\partial t_k}{\partial u}\right)(u^{'},v)du^{'}.
\end{equation*}
So \eqref{eq6.27} and \eqref{eq6.28} hold after replacing the subscript. We only need to consider \eqref{eq6.30} and \eqref{eq6.31} after the related subscript replacement to define $\tilde{\alpha}_k(u,v)$ and $\tilde{\beta}_k(u,v)$.

By following this method, we find that $(\overset{1}{\tilde{\alpha}}_{+,k},\overset{2}{\tilde{\beta}}_{+,k},r_k,\tilde{\alpha}_k,\tilde{\beta}_k)\in \overset{1}{B}_B\times \overset{2}{B}_B\times \overset{1}{B}_N\times \overset{2}{B}_N\times\overset{2}{B}_N$, where $B_1$, $B_2$, $N_1$, and $N_2$ are all definite constants. We can then prove the convergence similarly to Section \ref{6.4}.

If we set $\Delta f:=f_2-f_1$, then results similar to those in section \ref{6.4} hold as long as $\varepsilon$ is sufficiently small. Moreover, we can have the estimate

\begin{equation}\label{eq6.333}
\begin{pmatrix}
a\\
b\\
c\\
\end{pmatrix}\le \begin{pmatrix}
a_0 & a_0 & a_0\\
C & k & C\\
a_0 & a_0 & k\\
\end{pmatrix}\begin{pmatrix}
a\\
b\\
c\\
\end{pmatrix},
\end{equation}
where

\begin{flalign*}
\begin{cases}
a=&\Vert\Delta\tilde{\alpha}\Vert_0+\Vert\Delta\tilde{\beta}\Vert_0,\\
b=&\Vert\Delta\overset{1}{\tilde{\alpha}}_+\Vert_1+\Vert\Delta_m\overset{2}{\tilde{\beta}}_+\Vert_2,\\
c=&\Vert\Delta r\Vert_0,
\end{cases}
\end{flalign*}
and $a_0=C\varepsilon$. Using the analogous notations, we obtain

\begin{equation*}
x\le P^{-1}\begin{pmatrix}
\lambda_1 & 0 & 0\\
0 & \lambda_2 & 0\\
0 & 0 & \lambda_3\\
\end{pmatrix}Px\le \cdot\cdot\cdot\le P^{-1}\begin{pmatrix}
\lambda_1^k & 0 & 0\\
0 & \lambda_2^k & 0\\
0 & 0 & \lambda_3^k\\
\end{pmatrix}Px,
\end{equation*}
where $x=(a\ b\ c)^{T}$, matrix $P$ and $\lambda_i(i=1,2,3)$ agree with \eqref{eq6.288} and $k$ is any positive integer. As $k$ tends to infinity, we find that $a$, $b$ and $c$ are zero. 

Consequently, we have $\Vert\Delta\tilde{\alpha}\Vert_0=\Vert\Delta\tilde{\beta}\Vert_0=\Vert\Delta\overset{1}{\tilde{\alpha}}_+\Vert_1=\Vert\Delta\overset{2}{\tilde{\beta}}_+\Vert_2=\Vert\Delta r\Vert_0=0$. Since the initial values of the zeroth and first order derivatives of each group $\tilde{\alpha}_k$, $\tilde{\beta}_k$, $\overset{1}{\tilde{\alpha}}_{+,k}$, $\overset{2}{\tilde{\beta}}_{+,k}$ and $r_k$ at the interaction point are equal, we have 

\begin{equation}\label{eq6.334}
\Delta\tilde{\alpha}=\Delta\tilde{\beta}=\Delta\overset{1}{\tilde{\alpha}}_+=\Delta\overset{2}{\tilde{\beta}}_+=\Delta r=0.
\end{equation}
Therefore, we have

\begin{equation}\label{eq6.335}
\Delta\alpha=\Delta\beta=0,
\end{equation}
Hence,

\begin{equation*}
c_{in,1}(u,v)=c_{in}(\alpha_1(u,v),\beta_1(u,v))=c_{in}(\alpha_2(u,v),\beta_2(u,v))=c_{in,2}(u,v),
\end{equation*}
\begin{equation*}
c_{out,1}(u,v)=c_{out}(\alpha_1(u,v),\beta_1(u,v))=c_{out}(\alpha_2(u,v),\beta_2(u,v))=c_{out,2}(u,v).
\end{equation*}
As a result, $\Delta\phi=\Delta\psi=0$, and thus $\Delta t=0$. We have thus proven Theorem \ref{uniqueness}.

\end{proof}

\section{Higher regularity}\label{7} 

\begin{proof}

To establish the boundedness of all order derivatives of $\alpha$, $\beta$, $t$, and $r$, we employ induction. Based on earlier findings, we know that for sufficiently small $\varepsilon$, the solution exhibits twice differentiability in $T_{\varepsilon}$. Consequently, derivatives up to order two for $\alpha$, $\beta$, $t$, and $r$ are bounded. Assuming that the $(m-1)$th order and lower derivatives are bounded, where $m\ge 3$, demonstrating the boundedness of the $m$th order derivatives would complete the induction.

Before continuing, we would like to address the estimation related constants used in the proof: the variable $C$ represents a constant associated with $m$, wherein its value alters throughout the steps. In contrast, we employ $\overline{C}$ (with a subscript when required) to denote a constant that remains independent of $m$, meaning its value remains constant at each step. Furthermore, the values of $C$ and $\overline{C}$ (which may have an index) may change from line to line in each equation. We take the maximum value of the different constants used for estimation, but use the previous notation for simplicity.

Appreciatively, we do not need to estimate the $m$th order mixed derivatives of $\alpha$, $\beta$, $t$, and $r$, as their boundedness is straightforward.

To start, we observe that $\frac{\partial^{2} r}{\partial u \partial v}(u,v)$ can be expressed as a linear combination of $\frac{\partial r}{\partial u}(u,v)$ and $\frac{\partial r}{\partial v}(u,v)$. This relationship can be represented as:

\begin{equation}\label{eq7.1}
\frac{\partial^2 r}{\partial u\partial v}(u,v)=\mu\frac{\partial r}{\partial u}+\nu\frac{\partial r}{\partial v}=\frac{1}{c_{out}-c_{in}}\frac{c_{out}}{c_{in}}\frac{\partial c_{in}}{\partial v}\frac{\partial r}{\partial u}(u,v)-\frac{1}{c_{out}-c_{in}}\frac{c_{in}}{c_{out}}\frac{\partial c_{out}}{\partial u}\frac{\partial r}{\partial v}(u,v).
\end{equation}
We can see that it only contains derivatives of at most first order. As a consequence, mixed derivatives of $r$ of order $m$ will have at most derivatives of order $m-1$ on the right hand side, which is bounded by the induction hypothesis.

Next, let's examine the mixed derivatives of $t(u,v)$. According to the characteristic equations $\frac{\partial r}{\partial u}=c_{in}\frac{\partial t}{\partial u}$, $\frac{\partial r}{\partial v}=c_{out}\frac{\partial t}{\partial v}$, we can derive

\begin{equation}\label{eq7.2}
\frac{\partial^2 t}{\partial u\partial v}=\frac{1}{c_{out}-c_{in}}\frac{\partial c_{in}}{\partial v}\frac{\partial t}{\partial u}-\frac{1}{c_{out}-c_{in}}\frac{\partial c_{out}}{\partial u}\frac{\partial t}{\partial v}.
\end{equation}
Following a similar approach as with $r(u,v)$, if we take the $m$th order mixed derivative of $t(u,v)$, the order of the right hand side will be at most $m-1$. Based on the inductive hypothesis, we can establish that these derivatives are bounded.

Lastly, let's consider the mixed derivatives of $\alpha$ and $\beta$. It's worth noting that $\alpha$ and $\beta$ satisfy

\begin{equation*}
\frac{\partial\alpha}{\partial v}(u,v)=\left(A\frac{\partial t}{\partial v}\right)(u,v),\ \ \ \frac{\partial\beta}{\partial u}(u,v)=\left(A\frac{\partial t}{\partial u}\right)(u,v).
\end{equation*}
By taking the partial derivative of the first equation with respect to $u$, and the second equation with respect to $v$, we obtain

\begin{equation*}
\frac{\partial^2\alpha}{\partial u\partial v}(u,v)=\left(\frac{\partial A}{\partial\alpha}\frac{\partial\alpha}{\partial u}+\frac{\partial A}{\partial \beta}\frac{\partial\beta}{\partial u}\right)\frac{\partial t}{\partial v}(u,v)+A\frac{\partial^2 t}{\partial u\partial v}(u,v),
\end{equation*}
\begin{equation*}
\frac{\partial^2\beta}{\partial u\partial v}(u,v)=\left(\frac{\partial A}{\partial\alpha}\frac{\partial \alpha}{\partial v}+\frac{\partial A}{\partial\beta}\frac{\partial\beta}{\partial v}\right)\frac{\partial t}{\partial u}(u,v)+A\frac{\partial^2 t}{\partial u\partial v}(u,v).
\end{equation*}

Consequently, the $m$th order mixed derivatives of $\alpha$ and $\beta$ will only contain derivatives of order no more than $m-1$ with respect to $u$ or $v$ of $\alpha$, $\beta$, $t$, and $r$. In summary, we only need to assess whether the $m$th order derivatives of $\alpha$, $\beta$, $t$, and $r$ with respect to $u$ or $v$ are bounded.

To begin, let's estimate the $m$th order derivatives of $\alpha$ and $\beta$. Utilizing our previous calculations, we can represent $\alpha(u,v)$ as

\begin{equation*}
\alpha(u,v)=\overset{1}{\alpha}_+(u)+\tilde{\alpha}(u,v)=\alpha(u,u)+\int_u^v\left(A\frac{\partial t}{\partial v}\right)(u,v^{'})dv^{'}.
\end{equation*}
Next, we compute the first derivative of $\tilde{\alpha}$ with respect to $u$:

\begin{equation*}
\frac{\partial\tilde{\alpha}}{\partial u}(u,v)=\int_u^v\left[\left(\frac{\partial A}{\partial \alpha}\frac{\partial\alpha}{\partial u}+\frac{\partial A}{\partial \beta}\frac{\partial\beta}{\partial u}\right)\frac{\partial t}{\partial v}+A\frac{\partial^2 t}{\partial u\partial v}\right](u,v^{'})dv^{'}-\left(A\frac{\partial t}{\partial v}\right)(u,u).
\end{equation*}
If we differentiate both sides of the above equation $m-1$ times with respect to $u$, we find that only $\frac{\partial \alpha}{\partial u}$, $\frac{\partial\beta}{\partial u}$, $\frac{\partial^2 t}{\partial u\partial v}$ and $\frac{\partial t}{\partial v}$ have $m$th order derivatives on the right hand side. Specifically, we obtain

\begin{equation*}
\frac{\partial^m\tilde{\alpha}}{\partial u^m}(u,v)=\int_u^v\left[\left(\frac{\partial A}{\partial\alpha}\frac{\partial^m\alpha}{\partial u^m}+\frac{\partial A}{\partial\beta}\frac{\partial^m\beta}{\partial u^m}\right)\frac{\partial t}{\partial v}+A\frac{\partial^{m+1} t}{\partial u^m\partial v}\right](u,v^{'})dv^{'}-\left(A\frac{\partial^ m t}{\partial v^m}\right)(u,u)+LOT,
\end{equation*}
where $LOT$ represents all the remaining terms with derivatives of order at most $m-1$ with respect to $u$ and $v$, and its size is clearly related to the number of iteration steps $m$. We can choose a constant $C$ such that $\left|LOT\right|\le C$.

Note that, according to our previous calculation, $A(u,v)=\mathcal{O}(v)$. This implies the existence of a constant $\overline{C}$ such that $\left|A\right|\le \overline{C}$. Now, using the equation

\begin{equation*}
\begin{aligned}
&\frac{\partial^{m+1} t}{\partial u^m\partial v}(u,v)=\frac{\partial^{m-1}}{\partial u^{m-1}}\left(\frac{\partial^2 t}{\partial u\partial v}\right)(u,v)\\
=&\frac{1}{c_{out}-c_{in}}\left(\frac{\partial c_{in}}{\partial\alpha}\frac{\partial\alpha}{\partial v}+\frac{\partial c_{in}}{\partial \beta}\frac{\partial\beta}{\partial v}\right)\frac{\partial^m t}{\partial u^m}-\frac{1}{c_{out}-c_{in}}\left(\frac{\partial c_{out}}{\partial\alpha}\frac{\partial^m\alpha}{\partial u^m}+\frac{\partial c_{out}}{\partial\beta}\frac{\partial^m\beta}{\partial u^m}\right)\frac{\partial t}{\partial v}+LOT,
\end{aligned}
\end{equation*}
we can derive

\begin{equation}\label{eq7.3}
\left|\frac{\partial^{m+1} t}{\partial u^m\partial v}(u,v)\right|\le \overline{C}\left(\mathop{sup}\limits_{T_\varepsilon}\left|\frac{\partial^m t}{\partial u^m}\right|+\mathop{sup}\limits_{T_\varepsilon}\left|\frac{\partial^m\alpha}{\partial u^m}\right|+\mathop{sup}\limits_{T_\varepsilon}\left|\frac{\partial^m\beta}{\partial u^m}\right|\right)+C.
\end{equation}
Hence, we have

\begin{equation*}
\mathop{sup}\limits_{T_\varepsilon}\left|\frac{\partial^m\alpha}{\partial u^m}(u,v)\right|\le \mathop{sup}\limits_{[0,a\varepsilon]}\left|\frac{d^m\overset{1}{\alpha}_+}{du^m}(u)\right|+\overline{C}\varepsilon\left(\mathop{sup}\limits_{T_\varepsilon}\left|\frac{\partial^m\alpha}{\partial u^m}\right|+\mathop{sup}\limits_{T_\varepsilon}\left|\frac{\partial^m\beta}{\partial u^m}\right|+\varepsilon\mathop{sup}\limits_{T_\varepsilon}\left|\frac{\partial^m t}{\partial u^m}\right|+\mathop{sup}\limits_{T_\varepsilon}\left|\frac{\partial^m t}{\partial v^m}\right|\right)+C.
\end{equation*}
When $\varepsilon$ is sufficiently small, we obtain

\begin{equation}\label{eq7.4}
\mathop{sup}\limits_{T_\varepsilon}\left|\frac{\partial^m\alpha}{\partial u^m}(u,v)\right|\le \frac{1}{1-\overline{C}\varepsilon}\mathop{sup}\limits_{[0,a\varepsilon]}\left|\frac{d^m\overset{1}{\alpha}_+}{du^m}(u)\right|+\overline{C}_1\varepsilon\left(\mathop{sup}\limits_{T_\varepsilon}\left|\frac{\partial^m\beta}{\partial u^m}\right|+\varepsilon\mathop{sup}\limits_{T_\varepsilon}\left|\frac{\partial^m t}{\partial u^m}\right|+\mathop{sup}\limits_{T_\varepsilon}\left|\frac{\partial^m t}{\partial v^m}\right|\right)+C.
\end{equation}
Following that, let's compute the first derivative of $\tilde{\alpha}(u,v)$ with respect to $v$. This can be expressed as $\frac{\partial\tilde{\alpha}}{\partial v}(u,v)=\left(A\frac{\partial t}{\partial v}\right)(u,v)$. When we differentiate both sides of this equation $m-1$ times with respect to $v$, we obtain

\begin{equation*}
\frac{\partial^m\tilde{\alpha}}{\partial v^m}(u,v)=\left(A\frac{\partial^m t}{\partial v^m}\right)(u,v)+LOT,
\end{equation*}
where the term $LOT$ represents lower order terms. For sufficiently small $\varepsilon$, we have

\begin{equation}\label{eq7.5}
\mathop{sup}\limits_{T_\varepsilon}\left|\frac{\partial^m\alpha}{\partial v^m}(u,v)\right|=\mathop{sup}\limits_{T_\varepsilon}\left|\frac{\partial^m\tilde{\alpha}}{\partial v^m}(u,v)\right|\le \overline{C}\varepsilon\mathop{sup}\limits_{T_\varepsilon}\left|\frac{\partial^m t}{\partial v^m}\right|+C,
\end{equation}
where $\overline{C}$ and $C$ are constants.

Similarly, for $\beta(u,v)$, we have

\begin{equation*}
\beta(u,v)=\overset{2}{\beta}_+(v)+\tilde{\beta}(u,v)=\beta(av,v)+\int_{av}^u\left(A\frac{\partial t}{\partial u}\right)(u^{'},v)du^{'}.
\end{equation*}
Differentiating it $m$ times with respect to $u$, we arrive at

\begin{equation*}
\frac{\partial^m \beta}{\partial u^m}(u,v)=\left(A\frac{\partial^m t}{\partial u^m}\right)(u,v)+LOT.
\end{equation*}
When $\varepsilon$ is sufficiently small, we have

\begin{equation}\label{eq7.6}
\mathop{sup}\limits_{T_\varepsilon}\left|\frac{\partial^m\beta}{\partial u^m}(u,v)\right|\le \overline{C}\varepsilon\mathop{sup}\limits_{T_\varepsilon}\left|\frac{\partial^m t}{\partial u^m}\right|+C.
\end{equation}

Now let's consider the $m$th order derivative of $\tilde{\beta}(u,v)$ with respect to $v$, which can be expressed as

\begin{equation*}
\begin{aligned}
\frac{\partial^m\tilde{\beta}}{\partial v^m}(u,v)=&\int_{av}^u\left[\left(\frac{\partial A}{\partial\alpha}\frac{\partial^m\alpha}{\partial v^m}+\frac{\partial A}{\partial\beta}\frac{\partial^m \beta}{\partial v^m}\right)\frac{\partial t}{\partial u}+A\frac{\partial^{m+1} t}{\partial u\partial v^m}\right](u^{'},v)du^{'}\\
&-a^m\left(A\frac{\partial^m t}{\partial u^m}\right)(av,v)+LOT.
\end{aligned}
\end{equation*}
Using the expression for $\frac{\partial^{m+1}t}{\partial u\partial v^m}(u,v)$:

\begin{equation*}
\begin{aligned}
&\frac{\partial^{m+1}t}{\partial u\partial v^m}(u,v)=\frac{\partial^{m-1}}{\partial v^{m-1}}\left(\frac{\partial^2 t}{\partial u\partial v}\right)\\
=&\frac{1}{c_{out}-c_{in}}\left(\frac{\partial c_{in}}{\partial\alpha}\frac{\partial^m\alpha}{\partial v^m}+\frac{\partial c_{in}}{\partial \beta}\frac{\partial^m\beta}{\partial v^m}\right)\frac{\partial t}{\partial u}-\frac{1}{c_{out}-c_{in}}\left(\frac{\partial c_{out}}{\partial\alpha}\frac{\partial\alpha}{\partial u}+\frac{\partial c_{out}}{\partial\beta}\frac{\partial\beta}{\partial u}\right)\frac{\partial^m t}{\partial v^m}+LOT,
\end{aligned}
\end{equation*}
we arrive at

\begin{equation}\label{eq7.7}
\left|\frac{\partial^{m+1} t}{\partial u\partial v^m}(u,v)\right|\le \overline{C}\left(\mathop{sup}\limits_{T_\varepsilon}\left|\frac{\partial^m\alpha}{\partial v^m}\right|+\mathop{sup}\limits_{T_\varepsilon}\left|\frac{\partial^m\beta}{\partial v^m}\right|+\mathop{sup}\limits_{T_\varepsilon}\left|\frac{\partial^m t}{\partial v^m}\right|\right)+C.
\end{equation}
Therefore, we have

\begin{equation*}
\mathop{sup}\limits_{T_\varepsilon}\left|\frac{\partial^m\beta}{\partial v^m}(u,v)\right|\le \mathop{sup}\limits_{[0,\varepsilon]}\left|\frac{d^m\overset{2}{\beta}_+}{dv^m}\right|+\overline{C}\varepsilon\left(\mathop{sup}\limits_{T_\varepsilon}\left|\frac{\partial^m\alpha}{\partial v^m}\right|+\mathop{sup}\limits_{T_\varepsilon}\left|\frac{\partial^m\beta}{\partial v^m}\right|+\varepsilon\mathop{sup}\limits_{T_\varepsilon}\left|\frac{\partial^m t}{\partial v^m}\right|+\mathop{sup}\limits_{T_\varepsilon}\left|\frac{\partial^m t}{\partial u^m}\right|\right)+C.
\end{equation*}
For sufficiently small $\varepsilon$, we obtain

\begin{equation}\label{eq7.8}
\mathop{sup}\limits_{T_\varepsilon}\left|\frac{\partial^m\beta}{\partial v^m}(u,v)\right|\le \frac{1}{1-\overline{C}\varepsilon}\mathop{sup}\limits_{[0,\varepsilon]}\left|\frac{d^m\overset{2}{\beta}_+}{dv^m}\right|+\overline{C}_1\varepsilon\left(\mathop{sup}\limits_{T_\varepsilon}\left|\frac{\partial^m\alpha}{\partial v^m}\right|+\varepsilon\mathop{sup}\limits_{T_\varepsilon}\left|\frac{\partial^m t}{\partial v^m}\right|+\mathop{sup}\limits_{T_\varepsilon}\left|\frac{\partial^m t}{\partial u^m}\right|\right)+C.
\end{equation}

To estimate the $m$th order derivative of $t(u,v)$, we utilize the fact that $t(u,v)$ satisfies $\frac{\partial t}{\partial u}=\frac{1}{c_{in}}\frac{\partial r}{\partial u}$, $\frac{\partial t}{\partial v}=\frac{1}{c_{out}}\frac{\partial r}{\partial v}$, which allows us to express it as

\begin{equation*}
\frac{\partial^m t}{\partial u^m}(u,v)=\frac{1}{c_{in}(u,v)}\frac{\partial^m r}{\partial u^m}(u,v)+LOT,\ \ \ \frac{\partial^m t}{\partial v^m}(u,v)=\frac{1}{c_{out}(u,v)}\frac{\partial^m r}{\partial v^m}(u,v)+LOT,
\end{equation*}
where $LOT$ contains only derivatives of $\alpha$, $\beta$ and $r$ up to the order of $m-1$ with respect to $u$ and $v$. Since $LOT$ is bounded and its size depends on the order $m$, we have

\begin{equation}\label{eq7.9}
\mathop{sup}\limits_{T_\varepsilon}\left|\frac{\partial^m t}{\partial u^m}(u,v)\right|\le \overline{C}\mathop{sup}\limits_{T_\varepsilon}\left|\frac{\partial^m r}{\partial u^m}(u,v)\right|+C,
\end{equation}
\begin{equation}\label{eq7.10}
\mathop{sup}\limits_{T_\varepsilon}\left|\frac{\partial^m t}{\partial v^m}(u,v)\right|\le \overline{C}\mathop{sup}\limits_{T_\varepsilon}\left|\frac{\partial^m r}{\partial v^m}(u,v)\right|+C,
\end{equation}
where $\overline{C}$ and $C$ are constants.

By substituting \eqref{eq7.6}, \eqref{eq7.9}, and \eqref{eq7.10} into \eqref{eq7.4}, we derive

\begin{equation}\label{eq7.11}
\mathop{sup}\limits_{T_\varepsilon}\left|\frac{\partial^m\alpha}{\partial u^m}(u,v)\right|\le\frac{1}{1-\overline{C}\varepsilon}\mathop{sup}\limits_{[0,a\varepsilon]}\left|\frac{d^m\overset{1}{\alpha}_+}{du^m}\right|+\overline{C}_1\varepsilon^2\mathop{sup}\limits_{T_\varepsilon}\left|\frac{\partial^m r}{\partial u^m}\right|+\overline{C}_1\varepsilon\left|\frac{\partial^m r}{\partial v^m}\right|+C,
\end{equation}
where $\overline{C}_1$ is a constant.

Substituting \eqref{eq7.10} into \eqref{eq7.5}, we obtain

\begin{equation}\label{eq7.12}
\mathop{sup}\limits_{T_\varepsilon}\left|\frac{\partial^m\alpha}{\partial v^m}(u,v)\right|\le \overline{C}_1\varepsilon\mathop{sup}\limits_{T_\varepsilon}\left|\frac{\partial^m r}{\partial v^m}\right|+C.
\end{equation}
Similarly, by substituting \eqref{eq7.9} into \eqref{eq7.6}, we arrive at

\begin{equation}\label{eq7.13}
\mathop{sup}\limits_{T_\varepsilon}\left|\frac{\partial^m\beta}{\partial u^m}(u,v)\right|\le \overline{C}_1\varepsilon\mathop{sup}\limits_{T_\varepsilon}\left|\frac{\partial^m r}{\partial u^m}\right|+C.
\end{equation}
Lastly, by substituting \eqref{eq7.5}, \eqref{eq7.9}, and \eqref{eq7.10} into \eqref{eq7.8}, we obtain

\begin{equation}\label{eq7.14}
\mathop{sup}\limits_{T_\varepsilon}\left|\frac{\partial^m\beta}{\partial v^m}(u,v)\right|\le \frac{1}{1-\overline{C}\varepsilon}\mathop{sup}\limits_{[0,\varepsilon]}\left|\frac{d^m\overset{2}{\beta}_+}{dv^m}\right|+\overline{C}_1\varepsilon^2\mathop{sup}\limits_{T_\varepsilon}\left|\frac{\partial^m r}{\partial v^m}\right|+\overline{C}_1\varepsilon\mathop{sup}\limits_{T_\varepsilon}\left|\frac{\partial^m r}{\partial u^m}\right|+C.
\end{equation}

Next, we begin by estimating the $m$th order derivatives of $\overset{1}{\alpha}_+(u)$ and $\overset{2}{\beta}_+(v)$. Given the jump conditions and the non-degeneracy of the partial derivatives of the function $J$ at the point of interaction, we deduce that they satisfy the following equations:

\begin{equation*}
\overset{1}{\alpha}_+(u)=\overset{1}{H}(\overset{1}{\beta}_+(u),\overset{1}{\alpha}_-(u),\overset{1}{\beta}_-(u)),\ \ \ \overset{2}{\beta}_+(v)=\overset{2}{H}(\overset{2}{\alpha}_+(v),\overset{2}{\alpha}_-(v),\overset{2}{\beta}_-(v)).
\end{equation*}
Let's proceed with computing the first derivative of $\overset{1}{\alpha}_+(u)$:

\begin{equation*}
\begin{aligned}
\frac{d\overset{1}{\alpha}_+}{du}(u)=&\overset{1}{F}(\overset{1}{\beta}_+(u),\overset{1}{\alpha}_-(u),\overset{1}{\beta}_-(u))\frac{d\overset{1}{\beta}_+}{du}(u)+\overset{1}{M}_1(\overset{1}{\beta}_+(u),\overset{1}{\alpha}_-(u),\overset{1}{\beta}_-(u))\frac{d\overset{1}{\alpha}_-}{du}(u)\\
&+\overset{1}{M}_2(\overset{1}{\beta}_+(u),\overset{1}{\alpha}_-(u),\overset{1}{\beta}_-(u))\frac{d\overset{1}{\beta}_-}{du}(u),
\end{aligned}
\end{equation*}
where $\overset{1}{\beta}_+(u)=\beta(u,u)$, $\overset{1}{\alpha}_-(u)={\overset{1}{\alpha}}^*(\overset{1}{t}_+(u),\overset{1}{r}_+(u))$, $\overset{1}{\beta}_-(u)={\overset{1}{\beta}}^*(\overset{1}{t}_+(u),\overset{1}{r}_+(u))$. If we take the $m$th order derivative of $\overset{1}{\alpha}_+(u)$, then only three terms on the right hand side, $\frac{d\overset{1}{\beta}_+}{du}(u)$, $\frac{d\overset{1}{\alpha}_-}{du}(u)$ and $\frac{d\overset{1}{\beta}_-}{du}(u)$, may have derivatives of order $m$ with respect to $u$ or $v$. This leads us to the following equation:

\begin{equation*}
\begin{aligned}
\frac{d^m\overset{1}{\alpha}_+}{du^m}(u)=&\overset{1}{F}(\overset{1}{\beta}_+(u),\overset{1}{\alpha}_-(u),\overset{1}{\beta}_-(u))\frac{d^m\overset{1}{\beta}_+}{du^m}(u)+\overset{1}{M}_1(\overset{1}{\beta}_+(u),\overset{1}{\alpha}_-(u),\overset{1}{\beta}_-(u))\frac{d^m\overset{1}{\alpha}_-}{du^m}(u)\\
&+\overset{1}{M}_2(\overset{1}{\beta}_+(u),\overset{1}{\alpha}_-(u),\overset{1}{\beta}_-(u))\frac{d^m\overset{1}{\beta}_-}{du^m}(u)+LOT,
\end{aligned}
\end{equation*}
where $LOT$ contains only derivatives of $\alpha$, $\beta$, $t$, $r$ with respect to $u$ and $v$ up to the order of $m-1$. When $\varepsilon$ is sufficiently small, we have

\begin{equation}\label{eq7.15}
\mathop{sup}\limits_{[0,a\varepsilon]}\left|\frac{d^m\overset{1}{\alpha}_+}{du^m}\right|\le (\overset{1}{F}_0+\overline{C}\varepsilon)\mathop{sup}\limits_{[0,a\varepsilon]}\left|\frac{d^m\overset{1}{\beta}_+}{du^m}\right|+\overline{C}\left(\mathop{sup}\limits_{[0,a\varepsilon]}\left|\frac{d^m\overset{1}{\alpha}_-}{du^m}\right|+\mathop{sup}\limits_{[0,a\varepsilon]}\left|\frac{d^m\overset{1}{\beta}_-}{du^m}\right|\right)+C.
\end{equation}

Similarly, we can compute the first derivative of $\overset{2}{\beta}_+(v)$:

\begin{equation*}
\begin{aligned}
\frac{d\overset{2}{\beta}_+}{dv}(v)=&\overset{2}{F}(\overset{2}{\alpha}_+(v),\overset{2}{\alpha}_-(v),\overset{2}{\beta}_-(v))\frac{d\overset{2}{\alpha}_+}{dv}(v)+\overset{2}{M}_1(\overset{2}{\alpha}_+(v),\overset{2}{\alpha}_-(v),\overset{2}{\beta}_-(v))\frac{d\overset{2}{\alpha}_-}{dv}(v)\\
&+\overset{2}{M}_2(\overset{2}{\alpha}_+(v),\overset{2}{\alpha}_-(v),\overset{2}{\beta}_-(v))\frac{d\overset{2}{\beta}_-}{dv}(v),
\end{aligned}
\end{equation*}
where $\overset{2}{\alpha}_+(v)=\alpha(av,v),\ \ \ \overset{2}{\alpha}_-(v)={\overset{2}{\alpha}}^*(\overset{2}{t}_+(v),\overset{2}{r}_+(v))$, $\overset{2}{\beta}_-(v)={\overset{2}{\beta}}^*(\overset{2}{t}_+(v),\overset{2}{r}_+(v))$. We can obtain the $m$th order derivative of $\overset{2}{\beta}_+(v)$ by focusing on the three terms $\frac{d\overset{2}{\alpha}_+}{dv}(v)$, $\frac{d\overset{2}{\alpha}_-}{dv}(v)$ and $\frac{d\overset{2}{\beta}_-}{dv}(v)$ on the right hand side. Specifically, we have

\begin{equation*}
\begin{aligned}
\frac{d^m\overset{2}{\beta}_+}{dv^m}(v)=&\overset{2}{F}(\overset{2}{\alpha}_+(v),\overset{2}{\alpha}_-(v),\overset{2}{\beta}_-(v))\frac{d^m\overset{2}{\alpha}_+}{dv^m}(v)+\overset{2}{M}_1(\overset{2}{\alpha}_+(v),\overset{2}{\alpha}_-(v),\overset{2}{\beta}_-(v))\frac{d^m\overset{2}{\alpha}_-}{dv^m}(v)\\
&+\overset{2}{M}_2(\overset{2}{\alpha}_+(v),\overset{2}{\alpha}_-(v),\overset{2}{\beta}_-(v))\frac{d^m\overset{2}{\beta}_-}{dv^m}(v)+LOT,
\end{aligned}
\end{equation*}
where $LOT$ contains only derivatives of $\alpha$, $\beta$, $t$ and $r$ with respect to $u$ and $v$ of order up to $m-1$. When $\varepsilon$ is small enough, we can bound the supremum of the $m$th order derivative of $\overset{2}{\beta}_+(v)$ as follows:

\begin{equation}\label{eq7.16}
\mathop{sup}\limits_{[0,\varepsilon]}\left|\frac{d^m\overset{2}{\beta}_+}{dv^m}\right|\le (\overset{2}{F}_0+\overline{C}\varepsilon)\mathop{sup}\limits_{[0,\varepsilon]}\left|\frac{d^m\overset{2}{\alpha}_+}{dv^m}\right|+\overline{C}\left(\mathop{sup}\limits_{[0,\varepsilon]}\left|\frac{d^m\overset{2}{\alpha}_-}{dv^m}\right|+\mathop{sup}\limits_{[0,\varepsilon]}\left|\frac{d^m\overset{2}{\beta}_-}{dv^m}\right|\right)+C.
\end{equation}

Next, we estimate the $m$th order derivatives of $\overset{1}{\beta}_+$, $\overset{2}{\alpha}_+$, $\overset{1}{\alpha}_-$, $\overset{1}{\beta}_-$, $\overset{2}{\alpha}_-$ and $\overset{2}{\beta}_-$ appearing on the right hand side of the above equations. Starting with the expression for $\overset{1}{\beta}_+(u)$ derived previously:

\begin{equation*}
\overset{1}{\beta}_+(u)=\beta(u,u)=\overset{2}{\beta}_+(u)+\int_{au}^u\left(\frac{\partial t}{\partial u}A\right)(\tau,u)d\tau,
\end{equation*}
we take the first derivative of the integral on the right to obtain

\begin{equation*}
\begin{aligned}
\frac{d}{du}\left(\int_{au}^u\left(\frac{\partial t}{\partial u}A\right)(\tau,u)d\tau\right)=&\int_{au}^u\left[\frac{\partial^2 t}{\partial u\partial v}A+\frac{\partial t}{\partial u}\left(\frac{\partial A}{\partial\alpha}\frac{\partial\alpha}{\partial v}+\frac{\partial A}{\partial\beta}\frac{\partial\beta}{\partial v}\right)\right](\tau,u)d\tau\\
&+\left(\frac{\partial t}{\partial u}A\right)(u,u)-a\left(\frac{\partial t}{\partial u}A\right)(au,u).
\end{aligned}
\end{equation*}
If we take the $m$th order derivative with respect to $u$ of the integral term on the left hand side of the above equation, the $m$th order derivative will appear in $\frac{\partial^2 t}{\partial u\partial v}$, $\frac{\partial\alpha}{\partial v}$, $\frac{\partial\beta}{\partial v}$ and $\frac{\partial t}{\partial u}$ on the right hand side, giving:

\begin{equation*}
\begin{aligned}
\frac{d^m}{du^m}\left(\int_{au}^u\left(\frac{\partial t}{\partial u}A\right)(\tau,u)d\tau\right)=&\int_{au}^u\left[\frac{\partial^{m+1} t}{\partial u\partial v^m}A+\frac{\partial t}{\partial u}\left(\frac{\partial A}{\partial\alpha}\frac{\partial^m\alpha}{\partial v^m}+\frac{\partial A}{\partial\beta}\frac{\partial^m \beta}{\partial v^m}\right)\right](\tau,u)d\tau\\
&+\left(\frac{\partial^m t}{\partial u^m}A\right)(u,u)-a^m\left(\frac{\partial^m t}{\partial u^m}A\right)(au,u)+LOT.
\end{aligned}
\end{equation*}
Thus, we can bound the supremum of $\frac{d^m\overset{1}{\beta}_+}{du^m}(u)$ as follows:

\begin{equation*}
\mathop{sup}\limits_{[0,a\varepsilon]}\left|\frac{d^m\overset{1}{\beta}_+}{du^m}\right|\le \mathop{sup}\limits_{[0,\varepsilon]}\left|\frac{d^m\overset{2}{\beta}_+}{dv^m}\right|+\overline{C}\varepsilon\left(\varepsilon\mathop{sup}\limits_{T_\varepsilon}\left|\frac{\partial^{m+1}t}{\partial u\partial v^m}\right|+\mathop{sup}\limits_{T_\varepsilon}\left|\frac{\partial^m\alpha}{\partial v^m}\right|+\mathop{sup}\limits_{T_\varepsilon}\left|\frac{\partial^m\beta}{\partial v^m}\right|+\mathop{sup}\limits_{T_\varepsilon}\left|\frac{\partial^m t}{\partial u^m}\right|\right)+C.
\end{equation*}
By substituting equations \eqref{eq7.7}, \eqref{eq7.12}, \eqref{eq7.8}, \eqref{eq7.9} and \eqref{eq7.10} into the equation above, we can obtain

\begin{equation}\label{eq7.17}
\mathop{sup}\limits_{[0,a\varepsilon]}\left|\frac{d^m\overset{1}{\beta}_+}{du^m}\right|\le (1+\overline{C}_1\varepsilon)\mathop{sup}\limits_{[0,\varepsilon]}\left|\frac{d^m\overset{2}{\beta}_+}{dv^m}\right|+\overline{C}_2\varepsilon^2\mathop{sup}\limits_{T_\varepsilon}\left|\frac{\partial^m r}{\partial v^m}\right|+\overline{C}_3\varepsilon\mathop{sup}\limits_{T_\varepsilon}\left|\frac{\partial^m r}{\partial u^m}\right|+C.
\end{equation}
where $\overline{C}_1$, $\overline{C}_2$, $\overline{C}_3$, and $C$ are constants. 

Additionally, we have $\overset{1}{\alpha}_-(u)={\overset{1}{\alpha}}^*(\overset{1}{t}_+(u),\overset{1}{r}_+(u))$, and its derivative of order $m$ with respect to $u$ is

\begin{equation*}
\frac{d^m\overset{1}{\alpha}_-}{du^m}(u)=\frac{\partial{\overset{1}{\alpha}}^*}{\partial t}(\overset{1}{t}_+(u),\overset{1}{r}_+(u))\left(\frac{\partial^m t}{\partial u^m}+\frac{\partial^m t}{\partial v^m}\right)(u,u)+\frac{\partial{\overset{1}{\alpha}}^*}{\partial r}(\overset{1}{t}_+(u),\overset{1}{r}_+(u))\left(\frac{\partial^m r}{\partial u^m}+\frac{\partial^m r}{\partial v^m}\right)(u,u)+LOT.
\end{equation*}
By substituting \eqref{eq7.9} and \eqref{eq7.10} into the above equation, we can obtain

\begin{equation}\label{eq7.18}
\mathop{sup}\limits_{[0,a\varepsilon]}\left|\frac{d^m\overset{1}{\alpha}_-}{du^m}\right|\le \overline{C}_1\mathop{sup}\limits_{T_\varepsilon}\left|\frac{\partial^m r}{\partial u^m}\right|+\overline{C}_2\mathop{sup}\limits_{T_\varepsilon}\left|\frac{\partial^m r}{\partial v^m}\right|+C.
\end{equation}
Similarly, for $\overset{1}{\beta}_-(u)$, we have

\begin{equation}\label{eq7.19}
\mathop{sup}\limits_{[0,a\varepsilon]}\left|\frac{d^m\overset{1}{\beta}_-}{du^m}\right|\le \overline{C}_1\mathop{sup}\limits_{T_\varepsilon}\left|\frac{\partial^m r}{\partial u^m}\right|+\overline{C}_2\mathop{sup}\limits_{T_\varepsilon}\left|\frac{\partial^m r}{\partial v^m}\right|+C.
\end{equation}
Furthermore, we know that

\begin{equation*}
\overset{2}{\alpha}_+(v)=\alpha(av,v)=\overset{1}{\alpha}_+(av)+\int_{av}^v\left(\frac{\partial t}{\partial v}A\right)(av,s)ds.
\end{equation*}
If we take the $m$th order derivative of the integral term on the right hand side of the second equation above with respect to $v$, we can obtain

\begin{equation*}
\begin{aligned}
\frac{d^m}{dv^m}\left(\int_{av}^v\left(\frac{\partial t}{\partial v}A\right)(av,s)ds\right)=&a^m\int_{av}^v\left[\frac{\partial^{m+1}t}{\partial u^m\partial v}A+\frac{\partial t}{\partial v}\left(\frac{\partial A}{\partial\alpha}\frac{\partial^m\alpha}{\partial u^m}+\frac{\partial A}{\partial\beta}\frac{\partial^m\beta}{\partial u^m}\right)\right](av,s)ds\\
&+\left(\frac{\partial^m t}{\partial v^m}A\right)(av,v)-a^m\left(\frac{\partial^m t}{\partial v^m}A\right)(av,av)+LOT.
\end{aligned}
\end{equation*}
Thus,

\begin{equation*}
\mathop{sup}\limits_{[0,\varepsilon]}\left|\frac{d^m\overset{2}{\alpha}_+}{dv^m}\right|\le a^m\mathop{sup}\limits_{[0,a\varepsilon]}\left|\frac{d^m\overset{1}{\alpha}_+}{du^m}\right|+\overline{C}\varepsilon\left(\varepsilon\mathop{sup}\limits_{T_\varepsilon}\left|\frac{\partial^{m+1}t}{\partial u^m\partial v}\right|+\mathop{sup}\limits_{T_\varepsilon}\left|\frac{\partial^m\alpha}{\partial u^m}\right|+\mathop{sup}\limits_{T_\varepsilon}\left|\frac{\partial^m\beta}{\partial u^m}\right|+\mathop{sup}\limits_{T_\varepsilon}\left|\frac{\partial^m t}{\partial v^m}\right|\right)+C.
\end{equation*}
After substituting \eqref{eq7.3}, \eqref{eq7.11}, \eqref{eq7.6}, \eqref{eq7.9} and \eqref{eq7.10} into the above equation, we obtain

\begin{equation}\label{eq7.20}
\mathop{sup}\limits_{[0,\varepsilon]}\left|\frac{d^m\overset{2}{\alpha}_+}{dv^m}\right|\le (a^m+\overline{C}_1\varepsilon)\mathop{sup}\limits_{[0,a\varepsilon]}\left|\frac{d^m\overset{1}{\alpha}_+}{du^m}\right|+\overline{C}_2\varepsilon^2\mathop{sup}\limits_{T_\varepsilon}\left|\frac{\partial^m r}{\partial u^m}\right|+\overline{C}_3\varepsilon\mathop{sup}\limits_{T_\varepsilon}\left|\frac{\partial^m r}{\partial v^m}\right|+C.
\end{equation}

Furthermore, by definition, we have $\overset{2}{\alpha}_-(v)={\overset{2}{\alpha}}^*(\overset{2}{t}_+(v),\overset{2}{r}_+(v))$. Taking the first derivative of this equation results in

\begin{equation*}
\frac{d\overset{2}{\alpha}_-}{dv}(v)=\frac{\partial{\overset{2}{\alpha}}^*}{\partial t}(\overset{2}{t}_+(v),\overset{2}{r}_+(v))\left(a\frac{\partial t}{\partial u}+\frac{\partial t}{\partial v}\right)(av,v)+\frac{\partial{\overset{2}{\alpha}}^*}{\partial r}(\overset{2}{t}_+(v),\overset{2}{r}_+(v))\left(a\frac{\partial r}{\partial u}+\frac{\partial r}{\partial v}\right)(av,v).
\end{equation*}
If we take the $m$th order derivative with respect to $v$ on the left side of the equation, the derivative of order $m$ of $u$ or $v$ on the right side  will only occur in $\frac{\partial t}{\partial u}$, $\frac{\partial t}{\partial v}$, $\frac{\partial r}{\partial u}$ and $\frac{\partial r}{\partial v}$. Therefore,

\begin{equation*}
\begin{aligned}
\frac{d^m\overset{2}{\alpha}_-}{dv^m}(v)=\frac{\partial{\overset{2}{\alpha}}^*}{\partial t}(\overset{2}{t}_+(v),\overset{2}{r}_+(v))\left(a^m\frac{\partial^m t}{\partial u^m}+\frac{\partial^m t}{\partial v^m}\right)(av,v)+\frac{\partial{\overset{2}{\alpha}}^*}{\partial r}(\overset{2}{t}_+(v),\overset{2}{r}_+(v))\left(a^m\frac{\partial^m r}{\partial u^m}+\frac{\partial^m r}{\partial v^m}\right)(av,v)+LOT.
\end{aligned}
\end{equation*}
By substituting \eqref{eq7.9} and \eqref{eq7.10} into the above equation, we can derive

\begin{equation}\label{eq7.21}
\mathop{sup}\limits_{[0,\varepsilon]}\left|\frac{d^m\overset{2}{\alpha}_-}{dv^m}\right|\le \overline{C}_1\mathop{sup}\limits_{T_\varepsilon}\left|\frac{\partial^m r}{\partial u^m}\right|+\overline{C}_2\mathop{sup}\limits_{T_\varepsilon}\left|\frac{\partial^m r}{\partial v^m}\right|+C.
\end{equation}
Similarly, for $\overset{2}{\beta}_-(v)$, we can obtain

\begin{equation}\label{eq7.22}
\mathop{sup}\limits_{[0,\varepsilon]}\left|\frac{d^m\overset{2}{\beta}_-}{dv^m}\right|\le \overline{C}_1\mathop{sup}\limits_{T_\varepsilon}\left|\frac{\partial^m r}{\partial u^m}\right|+\overline{C}_2\mathop{sup}\limits_{T_\varepsilon}\left|\frac{\partial^m r}{\partial v^m}\right|+C.
\end{equation}
By substituting \eqref{eq7.17}, \eqref{eq7.18}, and \eqref{eq7.19} into \eqref{eq7.15}, we can obtain

\begin{equation}\label{eq7.23}
\mathop{sup}\limits_{[0,a\varepsilon]}\left|\frac{d^m\overset{1}{\alpha}_+}{du^m}\right|\le (\overset{1}{F}_0+\overline{C}_1\varepsilon)\mathop{sup}\limits_{[0,\varepsilon]}\left|\frac{d^m\overset{2}{\beta}_+}{dv^m}\right|+\overline{C}_2\mathop{sup}\limits_{T_\varepsilon}\left|\frac{\partial^m r}{\partial u^m}\right|+\overline{C}_3\mathop{sup}\limits_{T_\varepsilon}\left|\frac{\partial^m r}{\partial v^m}\right|+C.
\end{equation}
Similarly, by substituting \eqref{eq7.20}, \eqref{eq7.21}, and \eqref{eq7.22} into \eqref{eq7.16}, we can obtain

\begin{equation}\label{eq7.24}
\mathop{sup}\limits_{[0,\varepsilon]}\left|\frac{d^m\overset{2}{\beta}_+}{dv^m}\right|\le (a^m\overset{2}{F}_0+\overline{C}_1\varepsilon)\mathop{sup}\limits_{[0,a\varepsilon]}\left|\frac{d^m\overset{1}{\alpha}_+}{du^m}\right|+\overline{C}_2\mathop{sup}\limits_{T_\varepsilon}\left|\frac{\partial^m r}{\partial u^m}\right|+\overline{C}_3\mathop{sup}\limits_{T_\varepsilon}\left|\frac{\partial^m r}{\partial v^m}\right|+C.
\end{equation}
Since $a^m\overset{1}{F}_0\overset{2}{F}_0=a^{m+2}<1$, when $\varepsilon$ is small enough, we obtain

\begin{equation}\label{eq7.25}
\mathop{sup}\limits_{[0,a\varepsilon]}\left|\frac{d^m\overset{1}{\alpha}_+}{du^m}\right|\le \overline{C}_1\mathop{sup}\limits_{T_\varepsilon}\left|\frac{\partial^m r}{\partial u^m}\right|+\overline{C}_2\mathop{sup}\limits_{T_\varepsilon}\left|\frac{\partial^m r}{\partial v^m}\right|+C
\end{equation}
and
\begin{equation}\label{eq7.26}
\mathop{sup}\limits_{[0,\varepsilon]}\left|\frac{d^m\overset{2}{\beta}_+}{dv^m}\right|\le \overline{C}_1\mathop{sup}\limits_{T_\varepsilon}\left|\frac{\partial^m r}{\partial u^m}\right|+\overline{C}_2\mathop{sup}\limits_{T_\varepsilon}\left|\frac{\partial^m r}{\partial v^m}\right|+C
\end{equation}

To estimate the $m$th order derivatives of $r(u,v)$, we can leverage the equations satisfied by $\frac{\partial t}{\partial u}(u,v)$ and $\frac{\partial r}{\partial v}(u,v)$ in Section \ref{5.4}, assuming that $(\alpha,\beta,t,r)$ is the solution of the shock interaction problem. We know that $\overset{1}{\Gamma}$ and $\overset{2}{\Gamma}$ satisfy

\begin{equation*}
\overset{1}{\Gamma}(u):=\frac{\overset{1}{c_{out+}}(u)}{\overset{1}{c_{in+}}(u)}\frac{\overset{1}{V}(u)-\overset{1}{c_{in+}}(u)}{\overset{1}{c_{out+}}(u)-\overset{1}{V}(u)},\ \ \ \overset{2}{\Gamma}(v):=a\frac{\overset{2}{c_{out+}}(v)}{\overset{2}{c_{in+}}(v)}\frac{\overset{2}{V}(v)-\overset{2}{c_{in+}}(v)}{\overset{2}{c_{out+}}(v)-\overset{2}{V}(v)},
\end{equation*}
and $\overset{i}{c_{in}}$, $\overset{i}{c_{out}}$ $(i=1,2)$ are smooth functions of $\alpha$ and $\beta$. Additionally, $\overset{1}{V}(u)$ and $\overset{2}{V}(v)$ satisfy $\overset{i}{V}=\frac{[\overset{i}{\rho}\overset{i}{w}]}{[\overset{i}{\rho}]}$, $i=1,2$, where

\begin{equation*}
\overset{i}{\rho}_+(u)=\rho(\overset{i}{\alpha}_+(u),\overset{i}{\beta}_+(u)),\ \ \ \overset{i}{w}_+(u)=w(\overset{i}{\alpha}_+(u),\overset{i}{\beta}_+(u)),
\end{equation*}
\begin{equation*}
\overset{i}{\rho}_-(u)=\rho(\overset{i}{\alpha}_-(u),\overset{i}{\beta}_-(u)),\ \ \ \overset{i}{w}_-(u)=w(\overset{i}{\alpha}_-(u),\overset{i}{\beta}_-(u)).
\end{equation*}
Consequently, $\overset{i}{V}$ and $\overset{i}{\Gamma}$ also do not contain derivative terms with respect to $u$ or $v$. Taking the derivative of order $m$ of $r$ with respect to $u$, according to \eqref{eq5.51}, we obtain

\begin{equation*}
\frac{\partial^m r}{\partial u^m}(u,v)=a^{m-1}\overset{1}{\gamma}(u)\frac{\partial^m r}{\partial u^m}(au,u)+\frac{1}{\overset{1}{\Gamma}(u)}\int_{au}^u\frac{\partial^{m+1} r}{\partial u\partial v^m}(u^{'},u)du^{'}+\int_u^v\frac{\partial^{m+1} r}{\partial u^m\partial v}(u,v^{'})dv^{'}+LOT,
\end{equation*}
where $LOT$ contains only derivative terms of order up to $m-1$ of $\overset{i}{\Gamma}$ or $\alpha$, $\beta$, $r$. Additionally, from the previous calculations, we know that when $\varepsilon$ is small enough, we have $\left|\overset{1}{\gamma}(u)\right|\le 1+\overline{C}_1\varepsilon$, $\left|\frac{1}{\overset{1}{\Gamma}(u)}\right|\le \frac{1}{\overset{1}{\Gamma}_0}+\overline{C}_2\varepsilon$. Therefore, we can write

\begin{equation*}
\mathop{sup}\limits_{T_\varepsilon}\left|\frac{\partial^m r}{\partial u^m}\right|\le a^{m-1}(1+\overline{C}_1\varepsilon)\mathop{sup}\limits_{T_\varepsilon}\left|\frac{\partial^m r}{\partial u^m}\right|+\overline{C}_2\varepsilon\mathop{sup}\limits_{T_\varepsilon}\left|\frac{\partial^{m+1} r}{\partial u\partial v^m}\right|+\overline{C}_3\varepsilon\mathop{sup}\limits_{T_\varepsilon}\left|\frac{\partial^{m+1}r}{\partial u^m\partial v}\right|+C.
\end{equation*}
When $\varepsilon$ is sufficiently small, we can rewrite the equation as follows:

\begin{equation}\label{eq7.27}
\mathop{sup}\limits_{T_\varepsilon}\left|\frac{\partial^m r}{\partial u^m}\right|\le \overline{C}_1\varepsilon\mathop{sup}\limits_{T_\varepsilon}\left|\frac{\partial^{m+1}r}{\partial u\partial v^m}\right|+\overline{C}_2\varepsilon\mathop{sup}\limits_{T_\varepsilon}\left|\frac{\partial^{m+1}r}{\partial u^m\partial v}\right|+C.
\end{equation}
To estimate $\frac{\partial^{m+1}r}{\partial u\partial v^m}$ and $\frac{\partial^{m+1}r}{\partial u^m\partial v}$, we can refer to the previous calculations

\begin{equation*}
\begin{aligned}
\frac{\partial^2 r}{\partial u\partial v}(u,v)=&\frac{1}{c_{out}-c_{in}}\frac{c_{out}}{c_{in}}\left(\frac{\partial c_{in}}{\partial\alpha}\frac{\partial\alpha}{\partial v}+\frac{\partial c_{in}}{\partial\beta}\frac{\partial\beta}{\partial v}\right)\frac{\partial r}{\partial u}(u,v)\\
&-\frac{1}{c_{out}-c_{in}}\frac{c_{in}}{c_{out}}\left(\frac{\partial c_{out}}{\partial\alpha}\frac{\partial\alpha}{\partial u}+\frac{\partial c_{out}}{\partial \beta}\frac{\partial\beta}{\partial u}\right)\frac{\partial r}{\partial v}(u,v).
\end{aligned}
\end{equation*}
Therefore, we can derive

\begin{equation*}
\begin{aligned}
\frac{\partial^{m+1}r}{\partial u^m\partial v}(u,v)=&\frac{1}{c_{out}-c_{in}}\frac{c_{out}}{c_{in}}\left(\frac{\partial c_{in}}{\partial\alpha}\frac{\partial\alpha}{\partial v}+\frac{\partial c_{in}}{\partial\beta}\frac{\partial\beta}{\partial v}\right)\frac{\partial^m r}{\partial u^m}(u,v)\\
&-\frac{1}{c_{out}-c_{in}}\frac{c_{in}}{c_{out}}\left(\frac{\partial c_{out}}{\partial\alpha}\frac{\partial^m\alpha}{\partial u^m}+\frac{\partial c_{out}}{\partial \beta}\frac{\partial^m\beta}{\partial u^m}\right)\frac{\partial r}{\partial v}(u,v)+LOT,
\end{aligned}
\end{equation*}
\begin{equation*}
\begin{aligned}
\frac{\partial^{m+1}r}{\partial u\partial v^m}(u,v)=&\frac{1}{c_{out}-c_{in}}\frac{c_{out}}{c_{in}}\left(\frac{\partial c_{in}}{\partial\alpha}\frac{\partial^m\alpha}{\partial v^m}+\frac{\partial c_{in}}{\partial\beta}\frac{\partial^m\beta}{\partial v^m}\right)\frac{\partial r}{\partial u}(u,v)\\
&-\frac{1}{c_{out}-c_{in}}\frac{c_{in}}{c_{out}}\left(\frac{\partial c_{out}}{\partial\alpha}\frac{\partial\alpha}{\partial u}+\frac{\partial c_{out}}{\partial \beta}\frac{\partial\beta}{\partial u}\right)\frac{\partial^m r}{\partial v^m}(u,v)+LOT.
\end{aligned}
\end{equation*}
Using these expressions, we can obtain the following estimates:

\begin{equation}\label{eq7.28}
\mathop{sup}\limits_{T_\varepsilon}\left|\frac{\partial^{m+1}r}{\partial u^m\partial v}\right|\le \overline{C}_1\mathop{sup}\limits_{T_\varepsilon}\left|\frac{\partial^m r}{\partial u^m}\right|+\overline{C}_2\mathop{sup}\limits_{T_\varepsilon}\left|\frac{\partial^m\alpha}{\partial u^m}\right|+\overline{C}_3\mathop{sup}\limits_{T_\varepsilon}\left|\frac{\partial^m\beta}{\partial u^m}\right|+C,
\end{equation}
\begin{equation}\label{eq7.29}
\mathop{sup}\limits_{T_\varepsilon}\left|\frac{\partial^{m+1}r}{\partial u\partial v^m}\right|\le \overline{C}_1\mathop{sup}\limits_{T_\varepsilon}\left|\frac{\partial^m r}{\partial v^m}\right|+\overline{C}_2\mathop{sup}\limits_{T_\varepsilon}\left|\frac{\partial^m \alpha}{\partial v^m}\right|+\overline{C}_3\mathop{sup}\limits_{T_\varepsilon}\left|\frac{\partial^m\beta}{\partial v^m}\right|+C.
\end{equation}

Finally, by substituting \eqref{eq7.11}, \eqref{eq7.13}, \eqref{eq7.25}, \eqref{eq7.12}, \eqref{eq7.14} and \eqref{eq7.26} into the above equations, we obtain

\begin{equation}\label{eq7.30}
\mathop{sup}\limits_{T_\varepsilon}\left|\frac{\partial^{m+1}r}{\partial u^m\partial v}\right|\le \overline{C}_1\mathop{sup}\limits_{T_\varepsilon}\left|\frac{\partial^m r}{\partial u^m}\right|+\overline{C}_2\mathop{sup}\limits_{T_\varepsilon}\left|\frac{\partial^m r}{\partial v^m}\right|+C,
\end{equation}
\begin{equation}\label{eq7.31}
\mathop{sup}\limits_{T_\varepsilon}\left|\frac{\partial^{m+1}r}{\partial u\partial v^m}\right|\le \overline{C}_1\mathop{sup}\limits_{T_\varepsilon}\left|\frac{\partial^m r}{\partial u^m}\right|+\overline{C}_2\mathop{sup}\limits_{T_\varepsilon}\left|\frac{\partial^m r}{\partial v^m}\right|+C.
\end{equation}
So \eqref{eq7.27} becomes

\begin{equation*}
\mathop{sup}\limits_{T_\varepsilon}\left|\frac{\partial^m r}{\partial u^m}\right|\le \overline{C}_1\varepsilon\mathop{sup}\limits_{T_\varepsilon}\left|\frac{\partial^m r}{\partial u^m}\right|+\overline{C}_2\varepsilon\mathop{sup}\limits_{T_\varepsilon}\left|\frac{\partial^m r}{\partial v^m}\right|+C.
\end{equation*}
When $\varepsilon$ is sufficiently small, the above equation simplifies to

\begin{equation}\label{eq7.32}
\mathop{sup}\limits_{T_\varepsilon}\left|\frac{\partial^m r}{\partial u^m}\right|\le \overline{C}_1\varepsilon\mathop{sup}\limits_{T_\varepsilon}\left|\frac{\partial^m r}{\partial v^m}\right|+C.
\end{equation}
If we take the derivative of order $m$ of $r$ with respect to $v$, using \eqref{eq5.53}, we have

\begin{equation*}
\frac{\partial^m r}{\partial v^m}(u,v)=a^{m-1}\overset{2}{\gamma}(v)\frac{\partial^m r}{\partial v^m}(av,av)+a^{m-1}\overset{2}{\Gamma}(v)\int_{av}^v\frac{\partial^{m+1}r}{\partial u^m\partial v}(av,v^{'})dv^{'}+\int_{av}^u\frac{\partial^{m+1}r}{\partial u\partial v^m}(u^{'},v)du^{'}+LOT,
\end{equation*}
where $LOT$ contains only derivatives of order up to $m-1$ of $\overset{i}{\Gamma}$ or $\alpha$, $\beta$, $r$. Additionally, based on previous calculations, when $\varepsilon$ is sufficiently small, we have $\left|\overset{2}{\gamma}(v)\right|\le 1+\overline{C}_1\varepsilon$, $\left|\overset{2}{\Gamma}(v)\right|\le \overset{2}{\Gamma}_0+\overline{C}_2\varepsilon$. Therefore, we can get

\begin{equation*}
\mathop{sup}\limits_{T_\varepsilon}\left|\frac{\partial^m r}{\partial v^m}\right|\le a^{m-1}(1+\overline{C}_1\varepsilon)\mathop{sup}\limits_{T_\varepsilon}\left|\frac{\partial^m r}{\partial v^m}\right|+\overline{C}_2\varepsilon\mathop{sup}\limits_{T_\varepsilon}\left|\frac{\partial^{m+1} r}{\partial u^m\partial v}\right|+\overline{C}_3\varepsilon\mathop{sup}\limits_{T_\varepsilon}\left|\frac{\partial^{m+1} r}{\partial u\partial v^m}\right|+C.
\end{equation*}
When $\varepsilon$ is sufficiently small, we can rewrite the above equation as follows:

\begin{equation*}
\mathop{sup}\limits_{T_\varepsilon}\left|\frac{\partial^m r}{\partial v^m}\right|\le \overline{C}_1\varepsilon\mathop{sup}\limits_{T_\varepsilon}\left|\frac{\partial^{m+1}r}{\partial u\partial v^m}\right|+\overline{C}_2\varepsilon\mathop{sup}\limits_{T_\varepsilon}\left|\frac{\partial^{m+1}r}{\partial u^m\partial v}\right|+C.
\end{equation*}
By substituting \eqref{eq7.30} and \eqref{eq7.31} into the above equation, we obtain

\begin{equation*}
\mathop{sup}\limits_{T_\varepsilon}\left|\frac{\partial^m r}{\partial v^m}\right|\le \overline{C}_1\varepsilon\mathop{sup}\limits_{T_\varepsilon}\left|\frac{\partial^m r}{\partial u^m}\right|+\overline{C}_2\varepsilon\mathop{sup}\limits_{T_\varepsilon}\left|\frac{\partial^m r}{\partial v^m}\right|+C.
\end{equation*}
When $\varepsilon$ is sufficiently small, we have

\begin{equation}\label{eq7.33}
\mathop{sup}\limits_{T_\varepsilon}\left|\frac{\partial^m r}{\partial v^m}\right|\le \overline{C}_1\varepsilon\mathop{sup}\limits_{T_\varepsilon}\left|\frac{\partial^m r}{\partial u^m}\right|+C.
\end{equation}
Combining this with \eqref{eq7.32}, we conclude that

\begin{equation}\label{eq7.34}
\mathop{sup}\limits_{T_\varepsilon}\left|\frac{\partial^m r}{\partial u^m}\right|,\mathop{sup}\limits_{T_\varepsilon}\left|\frac{\partial^m r}{\partial v^m}\right|\le C.
\end{equation}
Substituting this into \eqref{eq7.25} and \eqref{eq7.26}, we find that

\begin{equation}\label{eq7.35}
\mathop{sup}\limits_{[0,a\varepsilon]}\left|\frac{d^m\overset{1}{\alpha}_+}{du^m}\right|,\mathop{sup}\limits_{[0,\varepsilon]}\left|\frac{d^m\overset{2}{\beta}_+}{dv^m}\right|\le C.
\end{equation}
Subsequently, by substituting these estimates into \eqref{eq7.11}, \eqref{eq7.12}, \eqref{eq7.13} and \eqref{eq7.14}, we obtain

\begin{equation}\label{eq7.36}
\mathop{sup}\limits_{T_\varepsilon}\left|\frac{\partial^m\alpha}{\partial u^m}\right|,\mathop{sup}\limits_{T_\varepsilon}\left|\frac{\partial^m\alpha}{\partial v^m}\right|,\mathop{sup}\limits_{T_\varepsilon}\left|\frac{\partial^m\beta}{\partial u^m}\right|,\mathop{sup}\limits_{T_\varepsilon}\left|\frac{\partial^m\beta}{\partial v^m}\right|\le C.
\end{equation}
Finally, applying the above estimates to \eqref{eq7.9} and \eqref{eq7.10}, we conclude that

\begin{equation}\label{eq7.37}
\mathop{sup}\limits_{T_\varepsilon}\left|\frac{\partial^m t}{\partial u^m}\right|,\mathop{sup}\limits_{T_\varepsilon}\left|\frac{\partial^m t}{\partial v^m}\right|\le C.
\end{equation}

We have established that all derivatives of order $m$ of the variables $\alpha$, $\beta$, $t$ and $r$ are bounded. Through mathematical induction, we have successfully demonstrated that when $\varepsilon$ is sufficiently small, the derivatives of the solution $(\alpha,\beta,t,r)$ for our formulated shock interaction problem remain bounded up to any desired order within the domain $T_\varepsilon$. This result implies that the solution is infinitely differentiable.

\end{proof}

%

\newpage

\bibliographystyle{abbrv}
\bibliography{reference}

\begin{thebibliography}{10}

\bibitem{Alinhac1999}
S.~Alinhac.
\newblock Blowup of small data solutions for a class of quasilinear wave
  equations in two space dimensions, ii.
\newblock {\em Acta Mathematica}, 182(1):1--23, 1999.

\bibitem{alinhac1999blowup}
S.~Alinhac.
\newblock Blowup of small data solutions for a quasilinear wave equation in two
  space dimensions.
\newblock {\em Annals of Mathematics}, 149(1):97--127, 1999.

\bibitem{bae2009regularity}
M.~Bae, G.-Q. Chen, and M.~Feldman.
\newblock Regularity of solutions to regular shock reflection for potential
  flow.
\newblock {\em Inventiones mathematicae}, 175(3):505--543, 2009.

\bibitem{bae2013prandtl}
M.~Bae, G.-Q. Chen, and M.~Feldman.
\newblock Prandtl-meyer reflection for supersonic flow past a solid ramp.
\newblock {\em Quarterly of Applied Mathematics}, 71(3):583--600, 2013.

\bibitem{bang2009interaction}
S.~Bang.
\newblock Interaction of three and four rarefaction waves of the
  pressure-gradient system.
\newblock {\em Journal of Differential Equations}, 246(2):453--481, 2009.

\bibitem{canic2006free}
S.~Canic, B.~L. Keyfitz, and E.~H. Kim.
\newblock Free boundary problems for nonlinear wave systems: Mach stems for
  interacting shocks.
\newblock {\em SIAM journal on mathematical analysis}, 37(6):1947--1977, 2006.

\bibitem{chen2014shock}
G.-Q. Chen, X.~Deng, and W.~Xiang.
\newblock Shock diffraction by convex cornered wedges for the nonlinear wave
  system.
\newblock {\em Archive for Rational Mechanics and Analysis}, 211:61--112, 2014.

\bibitem{chen2010global}
G.-Q. Chen and M.~Feldman.
\newblock Global solutions of shock reflection by large-angle wedges for
  potential flow.
\newblock {\em Annals of mathematics}, 171:1067--1182, 2010.

\bibitem{chen2018mathematics}
G.-Q. Chen and M.~Feldman.
\newblock {\em The Mathematics of Shock Reflection-Diffraction and von
  Neumann's Conjectures}.
\newblock Princeton University Press, Princeton, 2018.

\bibitem{chen2019uniqueness}
G.-Q.~G. Chen, M.~Feldman, and W.~Xiang.
\newblock Uniqueness and stability for the shock reflection-diffraction problem
  for potential flow.
\newblock {\em arXiv:1904.00114}, 2019.

\bibitem{chen2020convexity}
G.-Q.~G. Chen, M.~Feldman, and W.~Xiang.
\newblock Convexity of self-similar transonic shocks and free boundaries for
  the euler equations for potential flow.
\newblock {\em Archive for Rational Mechanics and Analysis}, 238:47--124, 2020.

\bibitem{christodoulou2007formation}
D.~Christodoulou.
\newblock {\em The formation of shocks in 3-dimensional fluids}, volume~2.
\newblock European Mathematical Society, 2007.

\bibitem{christodoulou2017shock}
D.~Christodoulou.
\newblock The shock development problem.
\newblock {\em arXiv:1705.00828}, 2017.

\bibitem{ChristodoulouKlainerman+1994}
D.~Christodoulou and S.~Klainerman.
\newblock {\em The Global Nonlinear Stability of the Minkowski Space (PMS-41)}.
\newblock Princeton University Press, Princeton, 1994.

\bibitem{christodoulou2016shock}
D.~Christodoulou and A.~Lisibach.
\newblock Shock development in spherical symmetry.
\newblock {\em Annals of PDE}, 2(1):3, 2016.

\bibitem{christodoulou2014compressible}
D.~Christodoulou and S.~Miao.
\newblock {\em Compressible flow and Euler's equations}, volume~9.
\newblock International Press Somerville, MA, 2014.

\bibitem{courant1948supersonic}
R.~Courant and K.~O. Friedrichs.
\newblock {\em Supersonic flow and shock waves}, volume~21.
\newblock Springer Science \& Business Media, 1999.

\bibitem{dafermos2005hyperbolic}
C.~M. Dafermos and C.~M. Dafermos.
\newblock {\em Hyperbolic conservation laws in continuum physics}, volume~3.
\newblock Springer, 2005.

\bibitem{elling2008supersonic}
V.~Elling and T.-P. Liu.
\newblock Supersonic flow onto a solid wedge.
\newblock {\em Communications on Pure and Applied Mathematics: A Journal Issued
  by the Courant Institute of Mathematical Sciences}, 61(10):1347--1448, 2008.

\bibitem{gues2005existence}
O.~Gues, G.~M{\'e}tivier, M.~Williams, and K.~Zumbrun.
\newblock Existence and stability of multidimensional shock fronts in the
  vanishing viscosity limit.
\newblock {\em Archive for rational mechanics and analysis}, 175:151--244,
  2005.

\bibitem{holzegel2016small}
G.~Holzegel, S.~Klainerman, J.~Speck, and W.~W.-Y. Wong.
\newblock Small-data shock formation in solutions to 3d quasilinear wave
  equations: an overview.
\newblock {\em Journal of Hyperbolic Differential Equations}, 13(01):1--105,
  2016.

\bibitem{hu2014semi}
Y.~Hu and G.~Wang.
\newblock Semi-hyperbolic patches of solutions to the two-dimensional nonlinear
  wave system for chaplygin gases.
\newblock {\em Journal of Differential Equations}, 257(5):1567--1590, 2014.

\bibitem{hu2015interaction}
Y.~Hu and G.~Wang.
\newblock The interaction of rarefaction waves of a two-dimensional nonlinear
  wave system.
\newblock {\em Nonlinear Analysis: Real World Applications}, 22:1--15, 2015.

\bibitem{article}
H.~Hugoniot.
\newblock Sur la propagation du mouvement dans les corps et spécialement dans
  les gaz parfaits.
\newblock {\em J. École Polytech.}, 58:1--125, 1889.

\bibitem{jegdic2006transonic}
K.~Jegdi{\'c}, B.~L. Keyfitz, and S.~{\v{C}}ani{\'c}.
\newblock Transonic regular reflection for the nonlinear wave system.
\newblock {\em Journal of Hyperbolic Differential Equations}, 3(03):443--474,
  2006.

\bibitem{1974Formation}
F.~John.
\newblock Formation of singularities in one‐dimensional nonlinear wave
  propagation.
\newblock {\em Communications on Pure and Applied Mathematics}, 27:377--405,
  1974.

\bibitem{kim2012interaction}
E.~H. Kim.
\newblock An interaction of a rarefaction wave and a transonic shock for the
  self-similar two-dimensional nonlinear wave system.
\newblock {\em Communications in Partial Differential Equations},
  37(4):610--646, 2012.

\bibitem{kim2013transonic}
E.~H. Kim and C.-M. Lee.
\newblock Transonic shock reflection problems for the self-similar
  two-dimensional nonlinear wave system.
\newblock {\em Nonlinear Analysis: Theory, Methods \& Applications},
  79:85--102, 2013.

\bibitem{kim2016two}
E.~H. Kim and C.~Tsikkou.
\newblock Two dimensional {R}iemann problems for the nonlinear wave system:
  Rarefaction wave interactions.
\newblock {\em arXiv:1612.04400}, 2016.

\bibitem{klainerman1980formation}
S.~Klainerman and A.~Majda.
\newblock Formation of singularities for wave equations including the nonlinear
  vibrating string.
\newblock {\em Communications on Pure and Applied Mathematics}, 33(3):241--263,
  1980.

\bibitem{10.1215/S0012-7094-03-11711-1}
S.~Klainerman and I.~Rodnianski.
\newblock {Improved local well-posedness for quasilinear wave equations in
  dimension three}.
\newblock {\em Duke Mathematical Journal}, 117(1):1 -- 124, 2003.

\bibitem{https://doi.org/10.1002/cpa.3160100406}
P.~D. Lax.
\newblock Hyperbolic systems of conservation laws ii.
\newblock {\em Communications on Pure and Applied Mathematics}, 10(4):537--566,
  1957.

\bibitem{lax1964development}
P.~D. Lax.
\newblock Development of singularities of solutions of nonlinear hyperbolic
  partial differential equations.
\newblock {\em Journal of Mathematical Physics}, 5(5):611--613, 1964.

\bibitem{lei2007complete}
Z.~Lei and Y.~Zheng.
\newblock A complete global solution to the pressure gradient equation.
\newblock {\em Journal of Differential Equations}, 236(1):280--292, 2007.

\bibitem{li2012interaction}
F.~Li and W.~Xiao.
\newblock Interaction of four rarefaction waves in the bi-symmetric class of
  the pressure-gradient system.
\newblock {\em Journal of Differential Equations}, 252(6):3920--3952, 2012.

\bibitem{li2011characteristic}
J.~Li, Z.~Yang, and Y.~Zheng.
\newblock Characteristic decompositions and interactions of rarefaction waves
  of 2-d euler equations.
\newblock {\em Journal of Differential Equations}, 250(2):782--798, 2011.

\bibitem{li2009interaction}
J.~Li and Y.~Zheng.
\newblock Interaction of rarefaction waves of the two-dimensional self-similar
  euler equations.
\newblock {\em Archive for rational mechanics and analysis}, 193:623--657,
  2009.

\bibitem{lisibach2021shock}
A.~Lisibach.
\newblock Shock reflection in plane symmetry.
\newblock {\em arXiv:2112.15266}, 2021.

\bibitem{lisibach2022shock}
A.~Lisibach.
\newblock Shock interaction in plane symmetry.
\newblock {\em arXiv:2202.08111}, 2022.

\bibitem{liu1979development}
T.-P. Liu.
\newblock Development of singularities in the nonlinear waves for quasi-linear
  hyperbolic partial differential equations.
\newblock {\em Journal of Differential Equations}, 33(1):92--111, 1979.

\bibitem{Luk2018}
J.~Luk and J.~Speck.
\newblock Shock formation in solutions to the 2d compressible euler equations
  in the presence of non-zero vorticity.
\newblock {\em Inventiones mathematicae}, 214(1):1--169, 2018.

\bibitem{luo2023stability}
T.-W. Luo and P.~Yu.
\newblock On the stability of multi-dimensional rarefaction waves i: the energy
  estimates.
\newblock {\em arXiv:2302.09714}, 2023.

\bibitem{luo2023stability1}
T.-W. Luo and P.~Yu.
\newblock On the stability of multi-dimensional rarefaction waves ii: existence
  of solutions and applications to riemann problem.
\newblock {\em arXiv:2305.06308}, 2023.

\bibitem{majda1983existence}
A.~Majda.
\newblock {\em The existence of multi-dimensional shock fronts}, volume 281.
\newblock American Mathematical Soc., 1983.

\bibitem{majda1983stability}
A.~Majda.
\newblock {\em The stability of multi-dimensional shock fronts}, volume 275.
\newblock American Mathematical Soc., 1983.

\bibitem{majda2012compressible}
A.~Majda.
\newblock {\em Compressible fluid flow and systems of conservation laws in
  several space variables}, volume~53.
\newblock Springer Science \& Business Media, 2012.

\bibitem{metivier2001stability}
G.~M{\'e}tivier.
\newblock {\em Stability of Multidimensional Shocks}, pages 25--103.
\newblock Birkh{\"a}user Boston, Boston, MA, 2001.

\bibitem{Miao2017}
S.~Miao and P.~Yu.
\newblock On the formation of shocks for quasilinear wave equations.
\newblock {\em Inventiones mathematicae}, 207(2):697--831, 2017.

\bibitem{riemann1860fortpflanzung}
B.~Riemann.
\newblock {\em {\"U}ber die Fortpflanzung ebener Luftwellen von endlicher
  Schwingungsweite}, volume~8.
\newblock Verlag der Dieterichschen Buchhandlung, 1860.

\bibitem{article1}
D.~Serre.
\newblock {\em Systems of Conservation Laws 2: Geometric Structures,
  Oscillations, and Initial-Boundary Value Problems}, volume~2.
\newblock Cambridge University Press, 1999.

\bibitem{song2009semi}
K.~Song and Y.~Zheng.
\newblock Semi-hyperbolic patches of solutions of the pressure gradient system.
\newblock {\em Discrete Contin. Dyn. Syst}, 24(4):1365--1380, 2009.

\bibitem{speck2016shock}
J.~Speck.
\newblock {\em Shock formation in small-data solutions to 3D quasilinear wave
  equations}, volume 214.
\newblock American Mathematical Soc., 2016.

\bibitem{Temple1983SystemsOC}
B.~Temple.
\newblock Systems of conservation laws with invariant submanifolds.
\newblock {\em Transactions of the American Mathematical Society},
  280:781--795, 1983.

\bibitem{tesdall2007triple}
A.~M. Tesdall, R.~Sanders, and B.~L. Keyfitz.
\newblock The triple point paradox for the nonlinear wave system.
\newblock {\em SIAM Journal on Applied Mathematics}, 67(2):321--336, 2007.

\bibitem{zheng2006two}
Y.~Zheng.
\newblock Two-dimensional regular shock reflection for the pressure gradient
  system of conservation laws.
\newblock {\em Acta Mathematicae Applicatae Sinica}, 22(2):177--210, 2006.

\end{thebibliography}

\end{CJK}

\end{document}